\documentclass[a4paper,12pt]{article}
\usepackage{hyperref}

\usepackage{mathrsfs}
\usepackage{ifthen}
\usepackage{verbatim}
\usepackage[english]{babel}
\usepackage{fancyhdr}
\usepackage{enumitem}
\usepackage{amsmath,amssymb,amsthm,graphicx}
\usepackage{latexsym}
\usepackage{ifpdf}
\usepackage{appendix}
\usepackage{color}
\usepackage{bbm}

\usepackage{tikz}
\usetikzlibrary{matrix,arrows,decorations.pathmorphing}

\newtheorem{Theorem}{Theorem}[section]
\newtheorem{Lemma}[Theorem]{Lemma}
\newtheorem{Corollary}[Theorem]{Corollary}

\newtheorem{Proposition}[Theorem]{Proposition}

\newtheorem{Definition}[Theorem]{Definition}

\newtheorem*{Notation}{Notation}

\theoremstyle{definition}
\newtheorem{Remark}[Theorem]{Remark}
\newtheorem{Example}[Theorem]{Example}

\def\l{\left}
\def\r{\right}
\newcommand{\seq}[1]{\l\{#1\r\}}
\newcommand{\bra}[1]{\l(#1\r)}

\newcommand{\abs}[1]{\l|#1\r|}

\def\norm#1{\left \Vert #1 \right \Vert}

\renewcommand{\hat}{\widehat}
\renewcommand{\bar}{\overline}
\def\gap{\; \; \;}

\newcommand{\ZZ}{\mathbb{Z}}
\newcommand{\PP}{\mathbb{P}}
\newcommand{\EE}[1]{\mathbb{E}\ex{#1}}

\newcommand{\Ff}{\mathcal{F}}
\newcommand{\RR}{\mathbb{R}}

\newcommand{\NN}{\mathbb{N}}
\newcommand{\Gg}{\mathcal{G}}

\def\Chi{\Large{\chi}}

\newcommand{\Ll}{\mathcal{L}}
\newcommand{\Cc}{\mathcal{C}}

\newcommand{\pnorm}[2]{\norm{#1}_{#2\textup{-var;}[0,1]}}

\newcommand{\floor}[1]{\lfloor{#1}\rfloor}

\newcommand{\Mm}{\mathcal{M}}

\def\WW{\mathbb{W}}

\def\EE{\mathbb{E}}
\def\Ww{\mathcal{W}}
\def\Ll{\mathcal{L}}

\def\Mm{\mathcal{M}}

\def\XX{\mathbb{X}}
\def\YY{\mathbb{Y}}

\def\Lip{\textup{Lip}}

\newcommand{\inhomd}[3]{\rho_{#1\textup{-var;}[#2,#3]}}
\newcommand{\Var}[1]{\textup{Var}({#1})}
\usepackage[us]{datetime}
\usepackage{geometry}
\numberwithin{equation}{section}
\numberwithin{figure}{section}

\begin{document}
\title{Pathwise approximation of SDEs by coupling piecewise abelian rough paths}
\author{Guy Flint and Terry Lyons\\ \\ \small{\textit{ Mathematical Institute, University of Oxford, Woodstock Road, OX2 6GG, England}}}
\date{\today}
\maketitle

\begin{changemargin}{0.3cm}{0.3cm} 
\abstract{We present a new pathwise approximation scheme for stochastic differential equations driven by multidimensional Brownian motion which does not require the simulation of L\'{e}vy area and has a Wasserstein convergence rate better than the Euler scheme's strong error rate of $O(\sqrt{h})$, where $h$ is the step-size. 
By using rough path theory we avoid imposing any non-degenerate H\"{o}rmander or ellipticity assumptions on the vector fields of the SDE, in contrast to the similar papers of Alfonsi et al \cite{alfonsi2014optimal,alfonsi2014pathwise}, Davie \cite{davie2014kmt,davie2014pathwise}, and Malliavin et al \cite{cruzeiro2004geometrization}.
The scheme is based on the log-ODE method with the L\'{e}vy area increments replaced by Gaussian approximations with the same covariance structure. The Wasserstein coupling is achieved by making small changes to the argument of Davie in \cite{davie2014kmt}, the latter being an extension of the Koml\'{o}s-Major-Tusn\'{a}dy Theorem. We prove that the convergence of the scheme in the Wasserstein metric is of the order $O(h^{1-2/\gamma-\varepsilon})$ when the vector fields are $\gamma$-Lipschitz in the sense of Stein.}
\end{changemargin}

\smallskip
\smallskip
\noindent \textbf{Keywords.} {Pathwise approximation of SDEs, Wasserstein couplings, Koml\'{o}s-Major-Tusn\'{a}dy Theorem, log-ODE method, rough path theory, It\^{o} map.}

\smallskip
\noindent \texttt{Guy.Flint@maths.ox.ac.uk}


\section{Introduction}

The problem of constructing pathwise approximations of solutions to stochastic differential equations (SDEs) driven by $d$-dimensional Brownian motion is difficult if a strong approximation error of order greater than $\frac{1}{2}$ is desired. This is because one must have the ability to simulate iterated integrals of Brownian motion \cite{clark1980maximum,dickinson2007optimal}, which is hard when $d\geq 2$. Efficient algorithms for generating double integrals, that is L\'{e}vy area increments, do exist for $d=2$ (see \cite{gaines1994random,ryden2001simulation,wiktorsson2001joint}), but the general case of $d>2$ is still an open problem. With this obstacle in mind, the papers \cite{alfonsi2014optimal,alfonsi2014pathwise,alves2008monte,cruzeiro2006numerical,cruzeiro2004geometrization,davie2014kmt,davie2014pathwise,deya2012milstein,gelbrich1995simultaneous,gelbrich1996discretization,math1989rate} and \cite[II.9]{rachev1998mass}, among  others, have studied SDE approximation schemes which do not require L\'{e}vy area increments, but achieve an order of convergence greater than $\frac{1}{2}$. Instead of measuring the success of the scheme in the standard $L^2$-norm, these papers use the Wasserstein metric from optimal transport theory \cite{cedric2003topics}. In particular, one constructs a probabilistic coupling between the SDE solution and an approximation scheme such that the error is measured in the Wasserstein metric, (using some particular cost function on Wiener space). 

To set up notation, let $W=(W_1,\ldots,W_d)$ denote a standard $d$-dimensional Brownian motion. 
In this paper we consider the pathwise approximation of the Stratonovich SDE
\begin{equation}\label{e-intro-sde}
dx_t = V\bra{x_t} \circ dW(t)  + V_0\bra{x_t} dt:= \sum^d_{k=1} V_k\bra{x_t} \circ dW_k(t) + V_0\bra{x_t} dt , \gap t\in [0,1],
\end{equation}
where $x_0\in\RR^q$, $V_0$ is $1$-Lipschitz and the vector field collection $V=\seq{V_k}_{k=1}^d$ are $\gamma$-Lipschitz in the sense of Stein \cite[Chapter XI]{stein1970singular}, (denoted by $V_0 \in \textup{Lip}^1(\RR^q)$ and $V \in \textup{Lip}^\gamma(\RR^q)$). We assume that $\gamma>2$ so that there exists a unique solution to (\ref{e-intro-sde}) almost surely (see \cite[Theorem 17.3]{FV}).

\subsection{Pathwise approximation scheme}

We now introduce our new approximation scheme for (\ref{e-intro-sde}).  
Divide the unit interval $[0,1]$ into $N$ pieces of length $h=N^{-1}$. Let us adopt classical ODE flow notation by setting  $e^F(y_0)=\exp(F)(y_0)$ to be the value of the solution to the following ODE at time $t=1$:
\[
y_t = y_0 + \int^t_0 F\bra{y_s} ds, \gap t \in [0,1],
\]
for some suitably regular vector field $F:\RR^q\to\RR^q$. 
Define the independent normal random variables
\begin{equation}\label{e-normal}
W^{(j)} \sim N\bra{0,hI_d}, \,\,\,\, z^{(j)} \sim N\bra{0,12^{-1}hI_d}, \,\,\,\, \lambda^{(j)} = (\lambda^{(j)}_{kl})_{1\leq k < l \leq d} \sim N\Big(0,12^{-1}h^2I_{\frac{d(d-1)}{2}}\Big),  
\end{equation}
where $j=0,1,\ldots,N-1$, and for $1\leq k<l\leq d$ set 
\begin{equation}\label{e-b-def}
B^{(j)}_{kl} := z^{(j)}_k W_l^{(j)} - z_l^{(j)}W_k^{(j)} + \lambda_{kl}^{(j)}.
\end{equation}
Our scheme $\{x^h_j\}_{j=0}^N$ is defined iteratively; $x^h_0 = x_0$ then for $j=0,1,\ldots,N-1$:
\begin{equation}\label{e-scheme}
x^h_{j+1} := \exp\bra{hV_0 + \sum_{k=1}^d W^{(j)}_k V_k + \sum_{1\leq k < l \leq d} B^{(j)}_{kl} [V_k,V_l]}\bra{x^h_j}.
\end{equation}
We can think of the sequence $\{x^h_j\}_{j=1}^N \subset \bra{\RR^q}^{\times N}$ as taking values in the Euclidean space $\RR^{qN}$, which we equip with the metric $\rho(x,y)=\max_{j=1,\ldots,N} \norm{x_j-y_j}_{\RR^q}$. Given Borel measures $\mu_1,\mu_2$ on $\RR^{qN}$, let $\Mm(\mu_1,\mu_2)$ denote the set of measurable maps $\Psi: \RR^{qN} \to\RR^{qN}$ such that the pushforward measure satisfies $\Psi_*(\mu_1)=\mu_2$. The Wasserstein metric is defined as   
\[
\Ww_2(\mu_1,\mu_2) := \bra{\inf_{\Psi\in \Mm(\mu_1,\mu_2)} \int_{\RR^{qN}} \rho\bra{x,\Psi(x)}^2 \mu_1(dx)}^{1/2}.
\]
The set $\Mm(\mu_1,\mu_2)$ is called the set of couplings of $\mu_1$ and $\mu_2$. 
An equivalent definition is given by $\Ww_2(\mu_1,\mu_2) = \inf \EE\bra{\rho(X,Y)^2}^{1/2}$, 
where the infimum is taken over all joint distributions of the random variables $X$ and $Y$ on $\RR^{qN}$ with marginals $\mu_1$ and $\mu_2$ respectively. 
The metric originates from the {Monge-Kantorovich mass transportation problem}, first introduced by Monge in 1781 \cite{monge1781memoire}, and then rediscovered many times in many forms since by L.V.~Kantorovich \cite{kantorovich_russian}, P.~L\'{e}vy, L.N.~Wasserstein \cite{vaserstein}, among others. For more details we refer to  \cite{malrieu2003convergence,cedric2003topics} and \S12 of \cite{davie2014pathwise}. 

We now state the main result of the paper.

\begin{Theorem}\label{t-big}
Fix $h>0$ and $x_0\in\RR^q$. 
Let $\mu$ denote the law of $\{x_{jh}\}_{j=1}^N$ on $\RR^{qN}$, where $x$ is the solution to (\ref{e-intro-sde}) started at $x_0$, and let $\nu$ denote the measure given by the law of the approximation $x^h=\{x^h_j\}_{j=1}^N$ on $\RR^{qN}$. If $V_0\in \textup{Lip}^1(\RR^q)$ and $V=\{V_1,\ldots,V_d\}\in \textup{Lip}^\gamma(\RR^q)$ for $\gamma>2$, then there exists a constant $C=C(\norm{V}_{\textup{Lip}^\gamma},\gamma)$ such that 
\[
\Ww_2(\mu,\nu) \leq Ch^{1-2/\gamma-\varepsilon}
\]
for all $\varepsilon>0$. 
That is, we can find independent normal random variables as in (\ref{e-normal}), defined on the same probability space as the Brownian motion $W$ driving (\ref{e-intro-sde}), such that 
\[
\EE\bra{\max_{j=1,\ldots,N} \norm{x_{jh} - x^h_{j}}_{\RR^q}^2}^{1/2} \leq Ch^{1-2/\gamma-\varepsilon}.
\]
\end{Theorem}

Thus if the vector fields $V$ of the SDE are sufficiently regular such that $V\in\textup{Lip}^\gamma$ for some $\gamma>4$, then our scheme performs better than the $O(\sqrt{h})$ strong error rate of the traditional Euler scheme (see \cite{maruyama1955continuous} for Maruyama's original proof of the Euler scheme's convergence rate). In the case of polynomial vector fields, we have a Wasserstein rate of $O(h^{1-\varepsilon})$ for any $\varepsilon>0$ by setting $\gamma>\frac{2}{\varepsilon}$. 

\begin{Remark}
We can also consider other $L^p$ (rather than $L^2$) versions of the Wasserstein metric for $p\geq 1$. Of particular note is the metric for $p=1$:
\begin{equation*}
\Ww_1(\mu,\nu) := \inf_{\Psi\in \Mm(\mu,\nu)} \int_{\RR^{qN}} \rho\bra{x,\Psi(x)} \mu(dx).
\end{equation*}
Certainly $\Ww_1(\mu,\nu) \leq \Ww_2(\mu,\nu)$. An elegant feature of this particular Wasserstein metric is its primal representation via functionals using the Kantorovitch-Rubinstein duality formula (\cite[Theorem 5.10 and Remark 5.16]{cedric2003topics}). In particular,
\begin{equation}\label{e-krd}
\Ww_1(\mu,\nu) = \sup_{\substack{\psi \in C\bra{\RR^{qN},\RR} \\ \textup{Lip}(\psi) \leq 1}} \abs{\EE\bra{\psi\bra{\seq{x_{jh}}_{j=1}^N}} - \EE\bra{ \psi\bra{x^h}}},
\end{equation}
where $\textup{Lip}(\psi):= \sup_{x\neq y} \frac{\abs{\psi(x)-\psi(y)}}{\norm{x-y}}$ denotes the Lipschitz constant of $\psi$ in the classical, not Stein, sense.
\end{Remark}

There exist examples of smooth (in fact, polynomial) vector fields $V=\{V_k\}_{k=1}^d$ such that for some constant $c>0$ the corresponding laws $\mu,\nu$ satisfy:
\begin{equation}\label{e-best-possible-rate}
\Ww_2(\mu,\nu)\geq \Ww_1(\mu,\nu) 
\geq ch\log( h^{-1}).
\end{equation}
One example is the SDE defining L\'{e}vy area (see Proposition \ref{p-davie-counter}). So $O(-h\log h)$ is a general upper bound on the convergence rate of our scheme in the Wasserstein metric. Thus for  polynomial (or more generally smooth) vector fields, our scheme achieves a Wasserstein convergence rate which is arbitrarily close (up to a logarithmic factor) to the best possible rate.

\begin{Remark}
Using the Wasserstein metric via the Kantorovich-Rubinstein duality formula is a quick way to establish weak approximation rates for approximation schemes. The disadvantage is that one is restricted to functionals $\psi$ such that $\textup{Lip}(\psi)\leq 1$, while the literature considers more general functionals, (even tempered distributions \cite{gobet2000weak}). However, as a form of compensation, our scheme works for every SDE; it does not demand any ellipticity conditions on the vector fields, unlike many papers covering weak approximations including \cite{alfonsi2014optimal,alfonsi2014pathwise,gobet2000weak}.
  
In the context of options pricing, we can interpret (\ref{e-krd}) as  a measure of the performance of our scheme for weakly approximating the expectation of certain exotic Asian options (that is, functionals of the  path at the times $t\in\{jh\}_{j=0}^N$).  This is in contrast to the weak approximation of vanilla European options, which are functionals of the terminal value of the path.
As an aside, note that if one actually wanted to approximate $\EE\bra{f(x_1)}$ for some function $f\in C_b^\infty(\RR^q,\RR)$, then the algorithm presented by Ninomiya and Victoir in \cite{ninomiya2008weak} does not require L\'{e}vy area simulations either, but produces a much better weak approximation than our scheme. To be precise, they construct a sequential ODE-based scheme which, for a given step-size $h>0$, outputs a point $\hat{x}^h_1$ satisfying:
\[
\abs{\EE\bra{f\bra{\hat{x}^h_1}} - \EE\bra{f(x_1)}} \leq Ch^2. 
\]
This is a whole order better than the best possible rate in general of (\ref{e-best-possible-rate}) for exotic Asian options.
We also comment (cf. \cite[\S1]{alfonsi2014pathwise}) that in the case of  the standard Euler and Milstein schemes, the best order of convergence of the weak error for vanilla functionals is $O(h)$ in general. Indeed, by the work of Talay and Tubaro (\cite[Theorem 1]{talay1990expansion}), this is the case for when $V$ and $V_0$ are non-zero smooth with bounded derivatives of all order and $f\in C^\infty(\RR^q,\RR)$ has polynomial growth together with its derivatives. We stress that our scheme is a pathwise approximation in that  its output is meant to approximate an actual realization of the solution path, rather than the expectation of a given functional, (as in the case of Ninomiya and Victoir). 
\end{Remark} 

In common with the algorithm of \cite{ninomiya2008weak}, our scheme $\{x^h_j\}_{j=1}^N$ is based on the level-2 version of the log-ODE method from rough path theory (see \cite{boutaib2013dimension,gyurko2008rough}). This latter approximation scheme also consists of solving a sequence of ODEs to produce a set of points $\{x^{(j)}\}_{j=1}^N$. In particular: $x^{(0)}=x_0$ and for $j=0,1,\ldots,N-1$:
\begin{equation}\label{e-log-ode-intro}
x^{(j+1)} = \exp\bra{hV_0 + \sum^d_{k=1} W^{(j)}_kV_k + \sum_{1\leq k < l \leq d} A_{kl}^{(j)} [V_k,V_l]}\bra{x^{(j)}}. 
\end{equation}
This scheme requires the L\'{e}vy area increments $A^{(j)}\in[\RR^d,\RR^d]$. For our new scheme $\{x^h_j\}_{j=1}^N$ we replace these increments with the Gaussian random variables $B^{(j)}$ defined above such that the covariance structure is the same. The theory of rough paths allows us to rewrite the original SDE (\ref{e-intro-sde}) as the solution of the following rough differential equation (RDE) with drift:
\[
dx_t = V\bra{x_t)}d\WW_t + V_0\bra{x_t}dt, 
\]
where $\WW \in C([0,1],G^{(2)}(\RR^d))$ is the standard enhanced Brownian rough path. It turns out that the schemes $\{x^{(j)}\}_{j=1}^N$ and $\{x^h_j\}_{j=1}^N$ can also be written as solutions of two RDEs with drift terms. To be precise;  $y_{jh} = x^{(j)}$ and $z_{jh} = x^{(j)}$ for all $j$, where $y,z\in C([0,1],\RR^q)$ solve:
\begin{align*}
d{y}_t &= V\bra{{y}_t}d\WW^h_t + V_0\bra{y_t}dt, \gap y_0 = x_0,\\
dz_t &= V\bra{z_t}d\XX^h_t+ V_0\bra{z_t}dt,\gap  z_0 = x_0.
\end{align*}
Here $\WW^h, \XX^h$ are members of a special class of 2-rough paths which we call \textit{piecewise abelian}. The notion of piecewise abelian rough paths can be thought of as the natural non-commutative, (that is group-valued), analogue of piecewise linear approximations of paths with values in the abelian group $G^{(1)}=\RR^d$. 

Both $\WW^h$ and $\XX^h$ share the same first level, which is the standard $N$-step piecewise linear approximation $W^h$ of the Brownian motion $W$. So we consider $\WW^h$ and $\XX^h$ as two different rough path lifts of the same underlying path $W^h$.  Their difference at level 2 is given by the continuous interpolations of the two discrete random walks $\Chi^h$ and $\Theta^h$ composed of the increments $A^{(j)}$ and $B^{(j)}$ respectively. 

By constructing a probabilistic coupling of these two random walks, (conditional on the underlying Brownian increments of $W^h$), we establish an automatic coupling of $\WW^h$ and $\XX^h$ in the space $G\Omega_2(\RR^d)$ of geometric 2-rough paths. Using the Lipschitz-continuity of the It\^{o} map $\Xi$ of rough path theory in the inhomogeneous $p$-variation metric, this action induces a coupling of the RDE solutions $y$ and $z$ in $C([0,1],\RR^q)$. The situation can be described with the following diagram (where dashed arrows represent couplings):

\begin{center}
\begin{center}
\begin{tikzpicture}

\node (A) {$\Chi^h$};
\node (B) [node distance=2cm, below of=A]{$\Theta^h$};

\node (X) [node distance=1cm, below of=A]{};
\node (X1) [node distance=1.55cm, right of=X] {$W^h$};

\node (C) [node distance=3cm, right of=A]{$\WW^h$};
\node (D) [node distance=3cm, right of=B]{$\XX^h$};

\node (E) [node distance=3cm, right of=C]{$S_\kappa\bra{\WW^h}$};
\node (F) [node distance=3cm, right of=D]{$S_\kappa\bra{\XX^h}$};

\node (G) [node distance=3cm, right of=E]{$y$};
\node (H) [node distance=3cm, right of=F]{$z$};

\node (I) [node distance=3cm, right of=G]{$\seq{x^{(j)}}_{j=1}^N$};
\node (J) [node distance=3cm, right of=H]{$\seq{x^h_j}_{j=1}^N$};

\node (Z1) [node distance=16mm, below of=B]{$[\RR^d,\RR^d]^{\times N}$};
\node (Z2) [node distance=30mm, right of=Z1]{$G\Omega_2 (\RR^d)$};
\node (Z3) [node distance=30mm, right of=Z2]{$G\Omega_\kappa (\RR^d)$};
\node (Z4) [node distance=30mm, right of=Z3]{$C\bra{[0,1],\RR^q}$};
\node (Z5) [node distance=30mm, right of=Z4]{$\RR^{qN}$};

\node (X) [node distance=17mm, below of=X1]{$(\RR^{d})^{\times N}$};

\draw[->] (X) to node {} (Z1);
\draw[->] (X) to node {} (Z2);

\draw[->] (Z1) to node {} (Z2);
\draw[right hook->] (Z2) to node {} (Z3);
\draw[right hook->] (Z2) --(Z3) node[above,midway] {$S_\kappa(\cdot)$};
\draw[right hook->] (Z4) to node {} (Z5);
\draw[right hook->] (Z3) --(Z4) node[above,midway] {$\Xi$};

\draw[->] (A) to node {} (C);
\draw[->] (B) to node {} (D);
\draw[->] (X1) to node {} (C);
\draw[->] (X1) to node {} (D);
\draw[->] (X1) to node {} (A);
\draw[->] (X1) to node {} (B);

\draw[right hook->] (C) to node {} (E);
\draw[right hook->] (D) to node {} (F);

\draw[right hook->] (E) to node {} (G);
\draw[right hook->] (F) to node {} (H);

\draw[right hook->] (G) to node {} (I);
\draw[right hook->] (H) to node {} (J);

\draw[<->, dashed] (A) to node {} (B);
\draw[<->, dashed] (C) to node {} (D);
\draw[<->, dashed] (E) to node {} (F);
\draw[<->, dashed] (I) to node {} (J);
\draw[<->, dashed] (G) to node {} (H);

\end{tikzpicture}
\end{center}
\end{center}

The initial coupling of the random walks $\Chi^h$ and $\Theta^h$ is constructed by using the dyadic coupling argument of Davie's recent paper \cite{davie2014kmt}. In fact all the coupling machinery is his; we only change the original vector to be coupled with a Gaussian vector and the rest of the proof remains the same. Davie's coupling argument is based on a modern extension of the classical Koml\'{o}s-Major-Tusn\'{a}dy Theorem \cite{komlos1975approximation}, also known as the Hungarian Embedding Theorem. Previous papers using the KMT method for Wasserstein approximations of SDEs include \cite{gelbrich1995simultaneous,rachev1998mass,dereich2011multilevel}, where the latter approximated SDEs driven by L\'{e}vy processes. 

\subsection{Previous research}

One benefit of using the technology of rough paths is that we can exploit the Lipschitz-continuity of the It\^{o} map
\[
\Xi: G\Omega_\kappa(\RR^d) \to C\bra{[0,1],\RR^q}.
\]
In our case this allows us to perform the coupling of the SDE and our approximation scheme at the input-side of $\Xi$ rather than directly in the (classical) Wiener space $C([0,1],\RR^q)$.  Therefore our coupling argument is completely independent of the vector fields $V$ of the original SDE (\ref{e-intro-sde}); the vector fields are only relevant once we push the coupling through the It\^{o} map.
Consequently, in contrast to the similar papers \cite{alfonsi2014optimal,alfonsi2014pathwise,alves2008monte,cruzeiro2004geometrization,davie2014kmt,davie2014pathwise}, our approach has the distinct advantage of not imposing any non-degenerate H\"{o}rmander condition on the vector fields $V=\seq{V_k}_{k=1}^d$. The schemes of Malliavin et al \cite{alves2008monte,cruzeiro2004geometrization} and Davie in \cite{davie2014pathwise} demand that the Lie bracket collection satisfies:
\begin{equation}\label{e-malliavin-condition}
\textup
\{[V_k,V_l](x) : 1\leq k<l\leq d, x\in\RR^q\} \textup{ spans } \RR^q,
\end{equation}
while \cite{davie2014kmt} requires the following less stringent version of the H\"{o}rmander condition. 

\begin{Definition}[Davie condition]
For each $x\in\RR^q$, define the linear mapping $L_x : \RR^d \oplus [\RR^d,\RR^d] \to \RR^q$ by 
\[
L_x(r,s) = \sum^d_{k=1} V_k(x)r_k + \sum_{1\leq k < l \leq d} [V_k,V_l](x)s_{kl} \textup { for } r=(r_k)\in\RR^d \textup{ and } s=(s_{kl})\in [\RR^d,\RR^d]. 
\]
Denote the ball of radius $\varepsilon>0$ centred at the origin in $\RR^d\oplus [\RR^d,\RR^d]$ by $B(0,\varepsilon)$. 
The Davie strengthened H\"{o}rmander condition is defined as the existence of constants $\delta>0$ and $K>0$ such that for all $x\in\RR^q$, $B\bra{0,\delta(1+\abs{x})^{-K}} \subseteq L_x\bra{B(0,1)}$. 
\end{Definition}

In other words, Davie assumes that
\begin{equation}\label{e-davie-non-deg}
\seq{V_j(x), [V_k, V_l](x) : 1\leq j,k,l\leq d, x\in \RR^q} \textup{ spans } \RR^q \textup{ uniformly in } x.
\end{equation}

Note that the former restriction (\ref{e-malliavin-condition}) excludes the SDE defining L\'{e}vy area and every non-trivial SDE in the simple case of $q=d=2$. It also excludes the example of the SDE describing Brownian motion on the unit circle in $\RR^2$; in this case, $(d,q)=(2,1)$ and the rank of the associated Lie bracket is equal to 1 (cf. \cite[Remark 2]{cruzeiro2006numerical}). On the other hand, both of these examples satisfy (\ref{e-davie-non-deg}) and so Davie's scheme in \cite{davie2014kmt} can be applied. 

Using a clever rotation of Brownian motion, Cruzeiro, Malliavin and Thalmaier \cite{cruzeiro2004geometrization} establish a Wasserstein metric convergence rate of $O(h)$ for a scheme based on the traditional Milstein approximation which does not require the simulation of L\'{e}vy area. 
Davie also achieves this rate with a perturbation of the Milstein scheme, while Alfonsi,  Jourdain and Kohatsu-Higa prove in \cite{alfonsi2014optimal} that in the case of $d=1$ the traditional Euler scheme has a Wasserstein rate of convergence of $O(h^{2/3-\varepsilon})$ under the assumption of uniform ellipticity of $V$. 
Their proof critically relies upon the Lamperti transform which cannot be extended to the case of $d>1$ nor the non-elliptic case. Maintaining the ellipticity condition, the authors then extended their result to the multidimensional case using Malliavin calculus in \cite{alfonsi2014pathwise}. To be precise, the latter paper improved the rate to $O(h\sqrt{-\log h})$ but used a weaker form of the Wasserstein metric $\tilde{\Ww}_2$ (which is defined in (\ref{e-davie-bound}) below and is called a fixed-time approximation by Davie). These coupling results concerning the Euler scheme cannot be extended to all SDEs. Indeed, as we will discuss below, there exist non-elliptic SDEs for which any coupling of the Euler scheme with the true solution is at least $O(\sqrt{h})$ apart in the Wasserstein metric (\cite[Example 11.1]{davie2014pathwise}). 

\begin{Remark}
Other papers exploiting the separation of the input and output of SDE flows provided by the It\^{o} map of rough path theory include the $\varepsilon$-strong SDE simulation paper \cite{blanchet2014epsilon} of Blanchet et al, the SDE quantization paper of \cite{pages2011convergence}, and  \cite{riedel2014talagrand} by Riedel. The latter paper studied Gaussian rough paths and  transportation-cost inequalities from optimal transport theory. Other examples include \cite{friz2005approximations,ledoux2002large} which respectively prove the Stroock-Varadhan support theorem and the large deviation Freidlin-Wentzall estimates. These papers make elegant use of the It\^{o} map to reduce these non-trivial results to simpler statements about Brownian motion and L\'{e}vy area in the rough path topology (see \cite[Chapter IX]{FV}). 
\end{Remark}

We also mention that in \cite{davie2014pathwise} Davie presents an improvement of his previous approximations based on perturbing the Milstein scheme which achieves achieves a Wasserstein error of order $O(h)$ without requiring L\'{e}vy area simulation nor non-degeneracy assumptions on the vector fields. The proof follows the work of Kloeden, Platen and Wright \cite{kloeden1992approximation} by employing a truncation of the Fourier series expansion of L\'{e}vy area. 

\subsection{Connections with original Davie scheme}

Let us discuss the work of  \cite{davie2014kmt} in more detail from the perspective of the present paper. In \cite{davie2014kmt} Davie assumes there is no drift term in the original SDE (\ref{e-intro-sde}); that is, $V_0=0$. His scheme $\{\tilde{x}^{(j)}\}_{j=0}^N$ is defined by $\tilde{x}^{(0)}=x_0$ and 
\begin{align}
\tilde{x}^{(j+1)} = \tilde{x}^{(j)} + \sum^d_{k=1} W^{(j)}_k & V_k\bra{\tilde{x}^{(j)}} + \sum_{1\leq k < l \leq d} B^{(j)}_{kl} [V_k,V_l]\bra{\tilde{x}^{(j)}}\notag\\
& + \frac{1}{2}\sum_{1\leq k , l \leq d} \bra{W^{(j)}_k W^{(j)}_l - \delta_{kl}h} \sum_{m=1}^q \bra{V^m_k \frac{\partial V_l}{\partial x_m}}\bra{\tilde{x}^{(j)}}.\label{e-davie-scheme-original}
\end{align}
In other words, it is the Milstein scheme (\cite{mil1974approximate}) with the L\'{e}vy area increments $A^{(j)}$ replaced with the Gaussian approximations $B^{(j)}$ defined in (\ref{e-b-def}). Denote the laws of $\{\tilde{x}^{(j)}\}_{j=1}^N$ and  $\{x_{jh}\}_{j=1}^N$ on $\RR^{qN}$ by $\lambda$ and $\mu$ respectively. Under the non-degeneracy condition of (\ref{e-davie-non-deg}), Davie establishes the Wasserstein-type bound (\cite[Theorem 1]{davie2014kmt}):
\begin{equation}\label{e-davie-bound}
\tilde{\Ww}_2(\mu,\lambda) := \inf_{\Psi\in\Mm\bra{\mu,\lambda}} \max_{j=1,\ldots,N} \bra{
\int_{\RR^{qN}}  \norm{x_j - \Psi(x)_j}_{\RR^q}^2\, \mu(dx)}^{1/2} \leq Ch.
\end{equation}

Theorem \ref{t-big} allows us to prove our own coupling result for Davie's scheme (\ref{e-davie-scheme-original}), which holds irrespectively of the rank of the SDE vector fields and their Lie brackets. 

\begin{Corollary}\label{c-milstein-based}
Let $x$ be the solution of (\ref{e-intro-sde}) in the case that $V_0=0$. 
Fix $h>0$ and denote the laws of $\{\tilde{x}^{(j)}\}_{j=1}^N$ and $\{x_{jh}\}_{j=1}^N$ on $\RR^{qN}$ by $\lambda$, $\mu$ respectively. If $V=\{V_1,\ldots,V_d\}\in\textup{Lip}^\gamma(\RR^q)$ for $\gamma>2$, then there exists a constant $C=C(\norm{V}_{\textup{Lip}^\gamma}, \gamma)$ such that for all $\varepsilon>0$
\[
\Ww_2\bra{\mu,\lambda} \leq Ch^{1-2/\gamma-\varepsilon}.
\]
\end{Corollary}
\begin{proof}
The proof is quick. Since the drift is zero, the scheme $\{\tilde{x}^{(j)}\}_{j=0}^N$ given by (\ref{e-davie-scheme-original}) is a discrete martingale with respect to the filtration $\Ff_j := \sigma\{(W^{(i)}, B^{(i)}) : i \leq j\}$. Similarly, since $V_0=0$, Fubini's Theorem gives us 
\[
\EE\bra{x^h_{j+1} \mid \Ff_j} = \EE\bra{\exp\bra{\sum_{k=1}^d W^{(j)}_kV_k + \sum_{1\leq k < l \leq d} B^{(j)}_{kl} [V_k,V_l]}(x^h_j) } = x^h_j.
\]
Thus the difference process $\{x^h_j - \tilde{x}^{(j)}\}_{j=0}^N$ is also a martingale with respect to the filtration $(\Ff_j)_{j=0}^N$.
By considering stochastic and deterministic Taylor expansions (see \cite[\S3]{gyurko2008rough}), it can be shown that $\EE(\norm{x^h_N-\tilde{x}^{(N)}}_{\RR^q}^2)\leq Ch^2$ for some constant $C>0$. Indeed, the scheme (\ref{e-davie-scheme-original}) is the truncated version of the two-step Taylor expansion based numerical approximation of the ODE sequence (\ref{e-scheme}). Alternatively, each vector field in the scheme (\ref{e-scheme}) can be decomposed as
\begin{align*}
\sum_{k=1}^d W^{(j)}_kV_k + \sum_{1\leq k < l \leq d} B^{(j)}_{kl} [V_k,V_l] 
&=
\underbrace{\sum_{k=1}^d W^{(j)}_kV_k + \sum_{1\leq k < l \leq d} A^{(j)}_{kl} [V_k,V_l]}_{=:G^{(j)}} + \underbrace{\sum_{1\leq k < l \leq d} \bra{B^{(j)}_{kl} - A^{(j)}_{kl}} [V_k,V_l]}_{=:H^{(j)}}.
\end{align*}
We can then view the Davie scheme (\ref{e-davie-scheme-original}) as originating when we take the second order expansion of $\exp(G^{(j)})$  (precisely as in the Milstein scheme), and a first order approximation of the $\exp(H^{(j)})$ component, (the latter being composed of coefficients with scaling $O(h)$ in $L^2$). To be precise: $\exp(G^{(j)}+H^{(j)}) \approx \tilde{G}^{(j)} + \tilde{H}^{(j)}$, where
\begin{align*}
\tilde{G}^{(j)}&:= 
\sum^d_{k=1} W^{(j)}_k  V_k + \sum_{1\leq k < l \leq d} A^{(j)}_{kl} [V_k,V_l]
 + \frac{1}{2}\sum_{1\leq k , l \leq d} \bra{W^{(j)}_k W^{(j)}_l - \delta_{kl}h} \sum_{m=1}^q {V^m_k \frac{\partial V_l}{\partial x_m}}\\
\tilde{H}^{(j)} &:=
\sum_{1\leq k < l \leq d} \bra{B^{(j)}_{kl}-A^{(j)}_{kl}} [V_k,V_l],
\end{align*}
and so $\tilde{G}^{(j)}+\tilde{H}^{(j)}$ gives (\ref{e-davie-scheme-original}) exactly. Thus the local $L^2$-error of the scheme is $O(h^{3/2})$, from which it follows that $\EE(\norm{x^h_N-\tilde{x}^{(N)}}_{\RR^q}^2)\leq Ch^2$. 

Hence Doob's maximal inequality yields
\[
\EE\bra{\max_{j=1,\ldots,N} \norm{x^h_j - \tilde{x}^{(j)}}_{\RR^q}^2} \leq 4\EE\bra{\norm{x^h_N-\tilde{x}^{(N)}}_{\RR^q}^2} \leq Ch^2.
\]
Combining this inequality with Theorem \ref{t-big} via the triangle inequality, we arrive at the claim.
\end{proof}

Thus in the case of a polynomial vector field system, we can prove that Davie's scheme achieves a Wasserstein rate of $O(h^{1-\varepsilon})$ for any $\varepsilon>0$ by setting $\gamma>\frac{2}{\varepsilon}$. 
While Theorem \ref{t-big} and Corollary \ref{c-milstein-based} do not achieve a convergence rate of $O(h)$, they are in some sense an improvement on Davie's bound (\ref{e-davie-bound}) since we do not impose any H\"{o}rmander restrictions on our vector fields and we use a stronger form of the Wasserstein metric. Certainly $\tilde{\Ww}_2(\mu,\lambda)\leq \Ww_2(\mu,\lambda)$, (which is analogous to the inequality: $\max_{i} \EE( \norm{X_i-Y_i}^2) \leq \EE(\max_{i} \norm{X_i-Y_i}^2)$). 

In the final remarks of \cite{davie2014kmt} Davie conjectured that it might be possible to prove that $\tilde{W}_2(\mu,\lambda) \leq Ch$ for any SDE (that is, without assuming the non-degeneracy condition (\ref{e-davie-non-deg})). In a sense, Corollary \ref{c-milstein-based} is a partial answer to this conjecture; given a vector field collection of sufficient regularity in the Stein-Lipschitz sense, we can prove that the approximation schemes (\ref{e-scheme}) and (\ref{e-davie-scheme-original}) converge with order $O(h^{1-\varepsilon})$ in our stronger Wasserstein metric. Moreover, this paper proves that the best possible rate in general is $O(-h\log h)$, (see Corollary \ref{c-optimal-bound} and Remark \ref{r-complexity}). This is established by adapting the counterexample given in \cite{davie2014kmt} which originally showed that $\tilde{\Ww}_2(\lambda,\mu) \geq Ch\log (h^{-1})$, where $\lambda,\mu$ correspond to the measures arising from the SDE defining the L\'{e}vy area of Brownian motion. Note that since the vector fields are polynomial in this case, we can set $\gamma>\frac{2}{\varepsilon}$ to achieve a rate of $O(h^{1-\varepsilon})$ from Theorem \ref{t-big}. 

One disadvantage of our approach is that we need to assume moderate  regularity conditions on $\textup{Lip}^\gamma(V)$ with $\gamma>4$. In effect, we are trading the algebraic H\"{o}rmander regularity of our vector fields in exchange for increased analytic Stein-Lipschitz conditions.


\subsection{Optimal Wasserstein rate for Euler scheme}

We now prove that in general the Wasserstein distance between the Euler scheme and the true solution is precisely of order $O(\sqrt{h})$. Alfonsi et al can establish the rate of  $O(h\sqrt{-\log h})$ under the assumption of uniform ellipticity of the vector fields; thus we give the details of the non-elliptic counterexample previously sketched in \cite[\S 11]{davie2014pathwise}. 

\begin{Example}
Consider the system with $d=1$, $q=2$ given by 
\[
dx_1=x_2\, dW, \gap dx_2=-x_1\, dW, \gap x(0) = (1,0).
\]
In this example the vector field is about as about regular as possible without being non-trivial; certainly $V(x)=(x_2,-x_1)^t$ is linear. 
Setting $S(t):= x_1(t)^2+x_2(t)^2$, we find that $S$ satisfies the deterministic differential equation $dS=Sdt$ with $S(0)=1$, and so we must have $S(t)=e^t$. Note that $S(1)=e$. 
The Euler scheme is given by $x^{(j+1)}_1 = x^{(j)}_1+x_2^{(j)}W^{(j)}$ and $x_2^{(j+1)}=x_2^{(j)}-x_1^{(j)}W^{(j)}$. Writing $S^{(j)}:=(x^{(j)}_1)^2+(x^{(j)}_2)^2$, it follows that $S^{(j+1)}=S^{(j)}(1+(W^{(j)})^2)$, and so
\[
S^{(N)} = \prod^{N-1}_{j=0} \bra{1+(W^{(j)})^2}. 
\]
We claim that there exists a universal constant $C_1>0$ such that
\begin{equation}\label{e-uc}
\abs{e-S^{(N)}}_{L^1}\geq C_1\sqrt{h}.
\end{equation} 
Indeed, using the inequality $\abs{e^a-e^b} = \abs{\int^a_b e^x\, dx} \geq \abs{a-b}$ for $a,b\geq 0$, and the fact that $\abs{\log(1+x)-x} = O(x^2)$ for $\abs{x}\leq \frac{1}{2}$, we have 
\begin{align*}
\EE\bra{\abs{e-S^{(N)}}}
\geq\EE\bra{\abs{1-\log\bra{S^{(N)}}}}
&=\abs{1-\sum_{j=1}^N \log\bra{1+hZ_j^2}}_{L^1}\\
&=\abs{1-\sum_{j=1}^N \bra{hZ_j^2 - \frac{1}{2}h^2Z_j^4 + \frac{1}{3}h^3Z_j^6 - \ldots}}_{L^1}\\ 
&\geq\abs{1-h\sum_{j=1}^N Z_j^2}_{L^1} - C_2h,
\end{align*} 
where the $\{Z_j\}_{j=1}^N$ are independent $N(0,1)$ random variables. 
Markov's inequality and the Central Limit Theorem guarantee that 
\begin{align*}
\sqrt{N}\EE\bra{\abs{e-S^{(N)}}} &\geq \sqrt{N}\abs{1-h\sum_{j=1}^N Z_j^2}_{L^1} + O(\sqrt{h})\\
&\geq \PP\bra{\sqrt{N}\bra{h\sum_{j=1}^N Z_j^2-1} \geq 1} + O(\sqrt{h})
\to \Phi\bra{1}> 0,
\end{align*}
and (\ref{e-uc}) follows. 
The Cauchy-Schwarz inequality then implies that 
\begin{align*}
C_1\sqrt{h} \leq \abs{e-S^{(N)}}_{L^1} = \abs{e-\norm{x^{(N)}}^2}_{L^1} &\leq \abs{\sqrt{e}-\norm{x^{(N)}}}_{L^2} \cdot\abs{\sqrt{e}+\norm{x^{(N)}}}_{L^2}\\
&= C_3\abs{\sqrt{e}-\norm{x^{(N)}}}_{L^2}
\end{align*}
where $C_3>0$. As $\norm{x(1)}^2=S(1)=e$ holds independently of the driving Brownian motion, this last line shows that the error of the Euler approximation cannot be less than $O(\sqrt{h})$ no matter what coupling is used. That is, there exists a constant $C>0$ such that 
\begin{align*}
\Ww_2\bra{\Ll\bra{\seq{x(jh)}_{j=1}^N}, \Ll\bra{\seq{x^{(j)})}_{j=1}^N}}  
&\geq 
\Ww_2\bra{\Ll\bra{x(1)}, \Ll\bra{x^{(N)}}}
\geq C\sqrt{h}. 
\end{align*}
On the other hand, since $d=1$, L\'{e}vy area is absent and hence the Milstein and (level 2) log-ODE schemes both give strong convergence errors of order $O(h)$. Moreover, the lack of L\'{e}vy area means that the level 1 and 2 log-ODE methods coincide.
\end{Example}

\subsection{Outline and notation of the paper}

The paper is organised as follows: the next section briefly summaries the necessary elements of rough path theory which is then followed by Section 3, where the original log-ODE method is introduced. We then define the class of piecewise abelian rough paths in Section 4 and show that the log-ODE method can be recast as the solution of a RDE driven by such a rough path $\WW^h$. Section 5 examines the components of L\'{e}vy area and defines our Gaussian piecewise abelian approximation $\XX^h$. The coupling of the random walks using Davie's coupling result is established in the proceeding section. Having constructed this coupling, Section 7 then examines the induced coupling between $\WW^h$ and $\XX^h$. Section 8 provides Wasserstein error estimates for the rough path lifts of the two piecewise abelian rough paths. These estimates are then put to use in Section 9 to prove Corollary \ref{c-main-2}, from which  Theorem \ref{t-big} follows immediately. Section 10 offers some concluding remarks regarding the feasibility of extensions of the main results to the case of fractional Brownian motion and higher order approximations. The paper is concluded with an appendix on the iterated Baker-Campbell-Hausdorff formula for random walks on Lie groups. 

\begin{Notation}
\textup{Throughout the paper, $C,c,\ldots$ denote various deterministic constants (that may vary from line to line), which are independent of $h$ or $n,m,k$. Constants which are dependent upon a variable will have the dependency explicitly stated. The usual bracket operation on $\RR^d$ is given by $[x,y]=x\otimes y - y\otimes x = \sum_{1\leq k < l \leq d}(x_ky_l-x_ly_k)[e_k,e_l]$, where $[e_k,e_l]:=e_k\otimes e_l - e_l\otimes e_k\in [\RR^d,\RR^d]$, and $\seq{e_k}_{k=1}^d$ denotes the canonical basis of $\RR^d$}. 
\end{Notation}


\section{Elements of rough path theory}

This section provides a quick tailored overview of relevant rough path theory and we take the opportunity to establish notation. Rough path analysis provides a method of constructing solutions to differential equations driven by paths that are not of bounded variation but have controlled roughness. A measure of this roughness is given by the $p$-variation of the path (see (\ref{e-p-var-def}) below).  For a detailed overview of the theory we direct the reader to \cite{friz2014course,FV,lejay2003importance,lyons1994differential,lyons1998differential,lyons2004differential,lyons2002system} among a multitude of others. 

\subsection{Algebraic preliminaries}

We introduce the necessary algebraic and geometric machinery in order to define rough path analysis. The foundation of the theory is given by the free tensor algebra and the free nilpotent Lie group embedded in it. 
Denote the space of continuous paths $x:[0,1]\to\RR^d$ by $C\bra{[0,1],\RR^d}$. Writing $x_{s,t}:=x_t-x_s$ for the increment, given $p\geq 1$ we define the $p$-variation norm of $x$ by 
\begin{equation}\label{e-p-var-def}
\norm{x}_{p\textup{-var;}[0,1]} := \sup_{D=(t_i)\subset [0,1]} \bra{\sum_i \norm{x_{t_i,t_{i+1}}}_{\RR^d}^p}^{1/p}.
\end{equation}
Let us denote by $C^{p\textup{-var}}([0,1],\RR^d)$ the linear subspace of $C([0,1],\RR^d)$ consisting of paths of finite $p$-variation. In the case of $x\in C^{p\textup{-var}}\bra{[0,1],\RR^d}$ for $p\in [1,2)$, the iterated integrals of $x$ are canonically defined via Young integration \cite{young1936inequality}. The collection of all these integrals as an object in itself is called the signature of the path:
\[
S(x)_{s,t} := 1 + \sum^\infty_{k=1} \int_{s<t_1<\ldots < t_k < t} dx_{t_1}\otimes \ldots \otimes dx_{t_k},
\]
where $(s,t)\in\Delta_{[0,1]}:=\{(s,t): 0\leq s \leq t \leq 1\}$. Adopting the convention that $(\RR^d)^{\otimes 0} = \RR$, we formally define the tensor algebras 
\[
T^{(\infty)}(\RR^d) := \bigoplus_{k=0}^\infty (\RR^d)^{\otimes k}, \gap\gap T^{(n)}(\RR^d) := \bigoplus_{k=0}^n (\RR^d)^{\otimes k}.
\]
We can see that the signature of $x$ takes its values in $T^{(\infty)}(\RR^d)$. Defining the canonical projections $\pi_n : T^{(\infty)}(\RR^d) \to (\RR^d)^{\otimes n}$ and $\pi_{0,n}: T^{(\infty)}(\RR^d) \to T^{(n)}(\RR^d)$, we can also consider the truncated signature:
\[
S_n(x)_{s,t} := \pi_{0,n}\bra{S(x)_{s,t}} = 1 + \sum^n_{k=1}  \int_{s<t_1<\ldots < t_k < t} dx_{t_1}\otimes \ldots \otimes dx_{t_k} \in \bigoplus_{k=0}^n (\RR^d)^{\otimes k} = T^{(n)}(\RR^d).
\]
Thus we can view $S_n(\cdot)$ as a continuous mapping from $\Delta_{[0,1]}$ to $T^{(n)}(\RR^d)$. 
Given a coordinate $e_{i_1} \otimes \ldots \otimes e_{i_k} \in (\RR^d)^{\otimes k}$, we define the corresponding projection of the signature via the dual space. For example, 
\[
\l<e_{i_1}^* \otimes \ldots \otimes e_{i_k}^*, S(x)_{s,t}\r> = \int_{s<t_1< \ldots < t_k < t} dx^{i_1}_{t_1} \ldots dx^{i_k}_{t_k} \in \RR.
\]
We equip each $(\RR^d)^{\otimes k}$ with the tensor algebra norm $\norm{\cdot}_{(\RR^d)^{\otimes k}}$ defined by 
\[
\norm{a}_{(\RR^d)^{\otimes k}} = \sqrt{\sum_{1\leq i_1,\ldots,i_k\leq d} \abs{\l< e_{i_1}^* \otimes \ldots \otimes e_{i_k}^*, a \r> }^2},
\] 
and when no confusion arises we shall simply write $\norm{a}$. This norm satisfies a compatibility relation between the tensor norms on the respective tensor levels in that
\[
\forall (a,b) \in (\RR^d)^{\otimes i} \times (\RR^d)^{\otimes (k-i)}, \gap \norm{a\otimes b}_{(\RR^d)^{\otimes k}} = \norm{a}_{(\RR^d)^{\otimes i}} \norm{b}_{(\RR^d)^{\otimes (k-i)}}.
\] 
We define a norm on $T^{(n)}(\RR^d)$ by
\[
\norm{g-h}_{T^{(n)}(\RR^d)} := \max_{k=1,\ldots,n}\norm{g-h}_{(\RR^d)^{\otimes k}}, \gap g,h \in T^{(n)}(\RR^d),
\]
which turns $T^{(n)}$ into a Banach algebra. 
It is a well-known fact that the signature $S_n(x)$ not only takes its values in $T^{(n)}(\RR^d)$, but it lies in a special nilpotent Lie group embedded in the tensor algebra. To be precise, the level-$n$ signature takes its values in the free step-$n$ nilpotent Lie group with the generators $\seq{e_k}_{k=1}^d$, which we denote by $G^{(n)}(\RR^d)$. Indeed, defining the free step-$n$ Lie algebra $\mathfrak{g}^{(n)}(\RR^d)$ by 
\[
\mathfrak{g}^{(n)}(\RR^d) := \RR^d \oplus [\RR^d,\RR^d] \oplus \ldots \oplus \underbrace{[\RR^d,[\ldots,[\RR^d,\RR^d]]]}_{(n-1) \textup{ brackets}},
\]
and the natural non-commutative exponential $\exp_n : T^{(n)}(\RR^d) \to T^{(n)}(\RR^d)$ by 
\[
\exp_n(a) := 1 + \sum^n_{k=1} \frac{a^{\otimes k}}{k!}, 
\]
we have $G^{(n)}(\RR^d) = \exp_n\bra{\mathfrak{g}^{(n)}(\RR^d)}$. Again using a formal power series, we can also define the truncated logarithm on $T^{(n)}(\RR^d)$:
\[
\log_n(a) := \sum_{k=1}^n \frac{(-1)^k}{k} (1-a)^{\otimes a} \textup{ for } a\in T^{(n)}(\RR^d) \textup{ such that } \pi_0\bra{a} = 1.
\]
The following characterization summarises the situation, (a proof can be found in \cite[Theorem 7.30]{FV}). 

\begin{Theorem}[\textbf{Chow-Rashevskii}]
We have 
\[
G^{(n)}(\RR^d) := \seq{S_n(x)_{0,1} : x \in C^{1\textup{-var}}\bra{[0,1],\RR^d}}. 
\]
\end{Theorem}

More abstractly, after fixing $p\geq 1$ we can consider a continuous group-valued path $\XX: [0,1] \to G^{(\floor{p})}(\RR^d)$:
\[
\XX_t = \bra{1,\XX^1_t, \ldots, \XX^{\floor{p}}_t} \in G^{(\floor{p})}(\RR^d)\textup{ where } \pi_k(\XX_t) = \XX^k_t.
\]
Importantly, the group structure provides a natural non-commutative notion of increment: $\XX_{s,t} := \XX_{s}^{-1} \otimes \XX_{t}$.
This multiplication operation is well-defined by Chen's Theorem \cite[Theorem 2.9]{lyons2004differential}. 

\subsection{Carnot-Carath\'{e}odory norm}

There exists a symmetric and sub-additive norm on $G^{(\floor{p})}(\RR^d)$ which is homogeneous with respect to the natural  dilation operator on the tensor algebra, (see \cite{FV} for details). This so-called Carnot-Carath\'{e}odory norm is given by 
\[
\norm{g}_{C} := \inf\seq{\int^1_0 \abs{d\gamma} : \gamma\in C^{1\textup{-var}}\bra{[0,1],\RR^d} \textup{ such that } S_N(\gamma)_{0,1} = g}, 
\]
which is well-defined by the Chow-Rashevskii Theorem. 
We may then define the homogeneous $p$-variation metric between $p$-rough paths $\XX,\YY\in C([0,1],G^{(\floor{p})}(\RR^d))$:
\begin{equation*}
d_{p\textup{-var;}[0,1]}\bra{\XX,\YY} := \sup_{D=(t_i)\subset [0,1]} \bra{\sum_{i} \norm{\XX_{t_i,t_{i+1}}^{-1}\otimes \YY_{t_i,t_{i+1}}}^p_{C}}^{1/p}.
\end{equation*}
The $p$-variation norm is given by $\norm{\XX}_{p\textup{-var;}[0,1]} = d_{p\textup{-var;}[0,1]}\bra{\XX,1}$. If this latter quantity is finite, then $\omega(s,t):=\norm{\XX}^p_{p\textup{-var;}[0,1]}$ is a control; a continuous bounded function, which vanishes on the diagonal $\seq{(t,t): t\in [0,1]}$, and is super-additive in that for all $s<t<u$ in $[0,1]$:
\[
\omega(s,u) + \omega(u,t) \leq \omega(s,t).
\]
Similarly we define the homogeneous $1/p$-H\"{o}lder metric and norm  by 
\[
d_{1/p\textup{-H\"{o}l;}[0,1]}\bra{\XX,\YY} := \sup_{0\leq s < t \leq 1} \frac{\norm{\XX_{s,t}^{-1}\otimes \YY_{s,t}}_{C}}{\abs{t-s}^{1/p}}, \gap \norm{\XX}_{1/p\textup{-H\"{o}l;}[0,1]}(\XX) = d_{1/p\textup{-H\"{o}l;}[0,1]}\bra{\XX,1},
\]
and define the rough path spaces:
\begin{align*}
C^{p\textup{-var}}\bra{[0,1],G^{(\floor{p})}(\RR^d)} &= \seq{\XX \in C\bra{[0,1],G^{(\floor{p})}(\RR^d)} : \norm{\XX}_{p\textup{-var;}[0,1]}(\XX) < \infty}\\
C^{1/p\textup{-H\"{o}l}}\bra{[0,1],G^{(\floor{p})}(\RR^d)} &= \seq{\XX \in C\bra{[0,1],G^{(\floor{p})}(\RR^d)} : \norm{\XX}_{1/p\textup{-H\"{o}l;}[0,1]}(\XX) < \infty}. 
\end{align*}
We stress that these spaces are not vector spaces; the addition of two rough paths, while being well-defined in the tensor algebra, may not sum up to a group element. 

\subsection{Inhomogeneous metrics on rough path space}

The inhomogeneous $p$-variation and $1/p$-H\"{o}lder metrics for $p$-rough paths are defined by ignoring the group structure of $G^{(\floor{p})}(\RR^d)$, and instead using the inherited norm from the tensor algebra:
\begin{align*}
\inhomd{p}{s}{t}\bra{\XX,\YY}:&=\max_{k=1,\ldots,\floor{p}} \inhomd{p}{s}{t}^{(k)}\bra{\XX,\YY}, \gap\gap
\rho_{1/p\textup{-H\"{o}l;}[0,1]}\bra{\XX,\YY}:=\max_{k=1,\ldots,\floor{p}} \rho^{(k)}_{1/p\textup{-H\"{o}l;}[0,1]}\bra{\XX,\YY}
\end{align*}
where 
\begin{align*}
\inhomd{p}{s}{t}^{(k)}\bra{\XX,\YY}
:&= \sup_{(t_i)\subset [s,t]} \bra{\sum_{i} \norm{\pi_k\bra{\XX_{t_i,t_{i+1}} - \YY_{t_i,t_{i+1}}}}_{(\RR^d)^{\otimes k}}^{p/k}}^{k/p}\\
\rho^{(k)}_{1/p\textup{-H\"{o}l;}[0,1]}\bra{\XX,\YY} :&= \sup_{0\leq s < t \leq 1} \frac{\norm{\pi_k\bra{\XX_{s,t}-\YY_{s,t}}}_{(\RR^d)^{\otimes k}}}{\abs{t-s}^{k/p}}.
\end{align*}
Given a control function $\omega:\Delta_{[0,1]}\to [0,\infty)$, we also define the metric
\begin{align*}
\rho_{p\textup{-}\omega\textup{;}[0,1]}\bra{\XX,\YY} &= \max_{k=1,\ldots,\floor{p}} \rho^{(k)}_{p\textup{-}\omega\textup{;}[0,1]}\bra{\XX,\YY} 
\textup{where } 
\rho^{(k)}_{p\textup{-}\omega\textup{;}[0,1]}\bra{\XX,\YY} = \sup_{0\leq s < t \leq 1} \frac{\norm{\pi_k\bra{\XX_{s,t}-\YY_{s,t}}}_{(\RR^d)^{\otimes k}}}{\omega(s,t)^{k/p}}.
\end{align*}

\subsection{Geometric rough paths}

The space of weakly geometric $p$-rough paths will be denoted by $WG\Omega_p(\RR^d)$. This is the set of continuous paths $\XX$ with values in $G^{(\floor{p})}(\RR^d)$ such that $\norm{\XX}_{p\textup{-var;}[0,1]}<\infty$. A refinement of this notion is the space of geometric $p$-rough paths, denoted by $G\Omega_p(\RR^d)$, which is the closure of 
\[
\seq{S_{\floor{p}}(x) : x\in C^{1\textup{-var}}\bra{[0,1],\RR^d}}
\]
with respect to the rough path metric $d_{p\textup{-var}}$, (or equivalently $\rho_{p\textup{-var}}$). Certainly we have the inclusion $G\Omega_p(\RR^d) \subset WG\Omega_p(\RR^d)$ and it turns out that this inclusion is strict. As described in \cite[\S3.2.2]{lyons2004differential}, this insignificant difference between geometric and weakly geometric rough paths can be compared to the difference between $C^1$ and Lipschitz functions. 

We make a note that given $\XX\in WG\Omega_p(\RR^d)$ and some $q\geq p$, there is a canonical extension of $\XX$ to a $q$-rough path $S_{\floor{q}}(\XX)\in WG\Omega_q(\RR^d)$ such that $\pi_{0,\floor{p}}\bra{S_{\floor{q}}(\XX)} = \XX$. 
This so-called rough path lift operation is unique in that sense that there exists a constant $C=C(p,q)$ such that 
\[
\norm{\XX}_{p\textup{-var;}[0,1]} \leq \norm{S_{\floor{q}}(\XX)}_{p\textup{-var;}[0,1]} \leq  C\norm{\XX}_{p\textup{-var;}[0,1]}.
\]

\subsection{Rough differential equations}

For now let $x\in C^{1\textup{-var}}([0,1],\RR^d)$ and let $y\in C([0,1],\RR^q)$ denote the solution of the (controlled) ordinary differential equation
\[
dy_t = V(y_t) \,dx_t + V_0(y_t)\,dt := \sum^d_{k=1} V_k(y_t)\,dx^k_t + V_0(y_t)\,dt, \gap y_0 \in \RR^q,
\]
which we summarise with the notation: $y=\pi_{(V,V_0)}(y_0,x)$. Here $\seq{V_k}_{k=0}^d$ is a collection of suitably regular vector fields $V_k : \RR^q \to \RR^q$. 

\begin{Definition}
Let $\XX\in WG\Omega_p(\RR^d)$ for some $p\geq 1$. We say that $y\in C\bra{[0,1],\RR^q}$ is a solution to the rough differential equation (RDE) with drift driven by $\XX$ along the collection of vector fields $V=\{V_k\}_{k=1}^d,V_0$, and started from $y_0 \in \RR^q$, if there exists a sequence $\seq{x_n}_{n=1}^\infty \subset C^{1\textup{-var}}\bra{[0,1],\RR^d}$ such that:
\[
\lim_{n\to\infty}  \sup_{0\leq s < t \leq 1} \norm{\XX_{s,t}^{-1} \otimes S_{\floor{p}}(x_n)_{s,t}}_{C}=0, \gap\gap\gap
\sup_n \pnorm{S_{\floor{p}}\bra{x_n}}{p}<\infty,
\]
and the sequence of ODE solutions $y_n := \pi_{(V,V_0)}(y_0,x_n)$ satisfies:
\[
\norm{y_n - y}_\infty \to 0 \textup{ as } n\to\infty.
\]
We denote this situation with the (formal) equation: $dy_t = V(y_t)d\XX_t + V_0(y_t)dt$, which we refer to as a rough differential equation (with drift), retaining the notation $y=\pi_{(V,V_0)}(y_0,\XX)$, 
\end{Definition}

Given $y_0 \in \RR^q$, the mapping
\begin{align*}
\Xi : WG\Omega_p(\RR^d) &\to C\bra{[0,1],\RR^q} : 
\XX \mapsto \pi_{(V,V_0)}(y_0, \XX), 
\end{align*}
is known as the It\^{o} map. 
The initial raison d'\^{e}tre of rough path theory was that the map
\[
\Xi \circ S_n(\cdot) :  x\in C^{p\textup{-var}}\bra{[0,1],\RR^d} \mapsto S_n(x) \in WG\Omega_p(\RR^d) \mapsto \pi_{(V,V_0)}\bra{y_0, S_{\floor{p}}(x)} \in C\bra{[0,1],\RR^q}
\]
is not continuous with respect to standard uniform topology on $C([0,1],\RR^d)$ if $d>1$.  Counterexamples are easy to construct, (see \cite[\S1.5]{lyons2004differential}). In fact, the mapping $\Xi \circ S_n(\cdot)$ is continuous in the $p$-variation topology; that is, using the metric $d_{p\textup{-var;}[0,1]}$, (equivalently $\rho_{p\textup{-var;}[0,1]}$). One can think of the It\^{o} map as a generalization of the classical Lamperti transform of SDE theory to multidimensional Brownian motion. 

\subsection{SDEs as RDEs}

Part of the success of rough path theory has been its power of providing a pathwise construction of stochastic calculus; we fix a sample path of Brownian motion and can define the solution of the SDE without probability theory. As before, let $W$ denote standard $d$-dimensional Brownian motion. 
We first define enhanced Brownian motion as the rough path $\WW\in C([0,1],G^{(2)}(\RR^d))$ given by 
\begin{align}
\WW_{s,t} &= 1 + W_{s,t} + \int^t_s W_{s,u} \otimes \circ dW_{u}\notag\\
&=\exp_2\bra{\sum^d_{k=1} W_k(s,t) + \sum_{1\leq k < l \leq d} A_{kl}(s,t)}\notag\\ 
&= 1 + \sum^d_{k=1} W_k(s,t) + \frac{1}{2}\sum_{1\leq k, l \leq d} W_k(s,t) \otimes W_l(s,t) + \sum_{1\leq k < l \leq d} A_{kl}(s,t) \in G^{(2)}(\RR^d),\label{e-brown-log-sig}
\end{align}
where $A_{kl} : \Delta_{[0,1]} \to [\RR^d,\RR^d]$ is the L\'{e}vy area of $W$:
\[
A_{kl}(s,t) := \frac{1}{2}\int^{t}_{s} \bra{W_k(s,u)\, dW_l(u) - W_l(s,u)\, dW_k(u)}.
\]
It can be shown that $\WW\in G\Omega_p(\RR^d)$ almost surely. 
Consider the following Stratonovich SDE driven by $W$:
\begin{equation}\label{e-rough-path-sde}
dx(t) = V(x(t))\circ dW(t) + V_0(x(t))\, dt := \sum_{k=1}^d V_k(x(t))\, dW_k(t) + V_0(x(t))\,dt, \gap x_0 = \xi_0 \in \RR^q, 
\end{equation}
where $V_0\in\textup{Lip}^1(\RR^q)$ and $V=\{V_k\}^d_{k=1} \in \textup{Lip}^\gamma(\RR^q)$, $\gamma>2$. Then it can be proven, (see \cite[Theorem 17.3]{FV}), that $x$ coincides with the unique solution $y$ of the following RDE with drift:
\[
dy_t = V\bra{y_t}d\WW_t + V_0(y_t) \, dt, \gap y_0 = \xi_0 \in \RR^q.
\]
In terms of the It\^{o} map of the previous subsection:
\[
\Xi : \WW \in G\Omega_p(\RR^d) \mapsto \pi_{V,V_0}\bra{\xi_0, \WW} = y \in C\bra{[0,1],\RR^q}. 
\]

\begin{Remark}
From the point of view of existence and uniqueness results, the appropriate way to measure the regularity of the $V=\{V_k\}_{k=1}^d$ turns out to be the notion of $\gamma$-Lipschitz in the sense of Stein \cite[Chapter XI]{stein1970singular}. Since this notion of Lipschitz is standard throughout the rough path literature, we omit the definition for the sake of brevity (see \cite[\S1.2.2]{lyons1998differential} and \cite[Definition 1.21]{lyons2004differential} for precise details). Informally, the definition states that the vector field can be approximated locally by a function taking values in polynomial functions. In contrast, Taylor expansions view a classical Lipschitz function as a function taking values in a power series; that is, a polynomial itself (cf. \cite[\S2]{lyons2015theory}). 

The notion provides a norm on the space of such vector fields, which we denote by 
\[
\norm{V}_{\textup{Lip}^\gamma} = \max_{k=1,\ldots,d} \norm{V_k}_{\textup{Lip}^\gamma}.
\]
If $\gamma>2$ then it can be proven that there exists a unique solution to (\ref{e-rough-path-sde}) almost surely (\cite[Theorem 17.3]{FV}).
\end{Remark}

\subsection{RDE Lipschitz estimates}

It is well-known that the Lipschitz constant of the It\^{o} map $\Xi$ for a RDE driven by a $p$-rough path $\XX$ is of the order $O(\exp\{C\norm{\XX}^p_{p\textup{-var};[0,1]}\})$. Even for the well-studied case of Gaussian rough paths, this random variable constant fails to be finite in any $L^q$-norm, $q\geq 1$. Cass, Litterer and Lyons rectified this problem by refining the deterministic estimates for the Lipschitz constant in \cite{cass2013integrability}. 
We start by recalling their definition of the so-called greedy $p$-variation partition.

\begin{Definition}
Let $\XX\in WG\Omega_p(\RR^d)$. For $\alpha>0$ and $[s,t]\subseteq [0,1]$, set
\begin{align*}
\tau_0(\alpha) &= s\\
\tau_{n+1}(\alpha) &= \inf\seq{ u: \norm{\XX}_{p\textup{-var;}[\tau_n,u]}^p \geq \alpha \textup{ and } \tau_n(\alpha) < u \leq t} \wedge t.
\end{align*}
Define the integer
\[
N_{\alpha,p}(\XX,[s,t]):= \sup\seq{n\in\NN\cup\seq{0} : \tau_n(\alpha) < t}. 
\]
\end{Definition}

Certainly $N_{\alpha,p}(\XX,[s,t]) \leq C \norm{\XX}^p_{p\textup{-var;}[s,t]}$ for some constant $C>0$, but more importantly the tail estimates for $N_{\alpha,p}(\XX,[s,t])$ are significantly tighter than for $\norm{\XX}^p_{p\textup{-var;}[s,t]}$ when we consider Gaussian rough paths $\XX$ (cf. \cite{cass2013integrability,friz2013integrability}).  The following result is a slight variation of \cite[Lemma 4.2]{cass2013evolving}. 

\begin{Proposition}\label{p-lipschitz-rde}
Let $\gamma>p\geq 1$ and suppose $\XX^1,\XX^2 \in WG\Omega_p(\RR^d)$.
Define the control $\omega:\Delta_{[0,1]}\to[0,\infty)$ by 
\[
\omega(s,t):=\abs{t-s} + \sum^2_{i=1}\norm{\XX^i}_{p\textup{-var;}[s,t]}^p + \sum^{\floor{p}}_{k=1} \bra{\frac{\rho_{p\textup{-var;}[s,t]}\bra{\XX^1,\XX^2}}{\rho_{p\textup{-var;}[0,1]}\bra{\XX^1,\XX^2}}}^{p/k}.
\]
Finally, let $V=\seq{V_k}_{k=1}^d$ be a collection of $\textup{Lip}^\gamma(\RR^q)$ vector fields and let $V_0\in \textup{Lip}^1(\RR^q)$. Then the RDEs 
\[
dy^i_t = V(y^i_t)\, d\XX^i_t + V_0(y^i_t) \,dt, \gap y^1_0=y^2_0\in\RR^q,
\]
have unique solutions. 
Moreover, for every $\alpha>0$ there exists some constant $C=C(\alpha, \gamma, p,\norm{V}_{\textup{Lip}^\gamma})$ such that 
\begin{align*}
\rho_{p\textup{-var;}[0,1]}(y^1,y^2) 
&\leq C\omega(0,1)\bra{1\vee \omega(0,1)^{\floor{p}+1}} \rho_{p\textup{-var;}[0,1]}\bra{\XX^1,\XX^2}\\
&\gap\gap\gap\gap\gap\gap\gap\gap\gap\gap\gap\gap
\cdot\exp\bra{C\seq{1+N_{\alpha,p}(\XX^1,[0,1])+N_{\alpha,p}(\XX^2,[0,1])}}.
\end{align*}
\end{Proposition}
\begin{proof}
The proof is obtained from following the arguments of \cite[Lemma 4.2]{cass2013evolving} with some minor modifications. Indeed, the latter result guarantees that 
\begin{equation*}
\rho_{p\textup{-}\omega\textup{;}[0,1]}(y^1,y^2) \leq C\rho_{p\textup{-}\omega\textup{;}[0,1]}\bra{\XX^1,\XX^2}\exp\bra{CM_{\alpha,[0,1]}(\omega)}, 
\end{equation*}
where 
\[ 
M_{\alpha,[s,t]}(\omega) := \sup_{\substack{D=(t_i)\subset [s,t] \\ \omega(t_i,t_{i+1})\leq \alpha}} \sum_{i} \omega(t_i,t_{i+1}). 
\]
It can be shown that $N_{\alpha,p}\bra{\omega,[0,1]} \leq M_{\alpha,[0,1]}(\omega) \leq 2N_{\alpha,p}\bra{\omega,[0,1]} + 1$. 
Moreover, by \cite[Lemma 6]{bayer2013rough} there exists a constant $C>0$ such that 
\[
N_\alpha\bra{\omega,[0,1]} \leq C\bra{1+ N_{\alpha,p}\bra{\XX^1,[0,1]} + N_{\alpha,p}\bra{\XX^2,[0,1]}}.
\]
Therefore, 
\begin{equation}\label{e-new-bound}
\rho_{p\textup{-}\omega\textup{;}[0,1]}(y^1,y^2) \leq C\rho_{p\textup{-}\omega\textup{;}[0,1]}\bra{\XX^1,\XX^2}\exp\bra{C\seq{1+N_{\alpha,p}(\XX^1,[0,1])+N_{\alpha,p}(\XX^2,[0,1])}}.
\end{equation}
Under the assumption that $\omega(0,1)\leq 1$, we can refer to the proof of \cite[Theorem 4]{bayer2013rough} to find that
\begin{align*}
&\norm{\XX^i}_{p\textup{-}\omega\textup{;}[0,1]} \leq 1,\, i=1,2 \textup{ and } \rho_{p\textup{-}\omega\textup{;}[0,1]}\bra{\XX^1,\XX^2} \leq \rho_{p\textup{-var;}[0,1]}\bra{\XX^1,\XX^2}.
\end{align*}
For the general case, a normalization argument gives 
\begin{equation}\label{e-general-case}
\rho_{p\textup{-}\omega\textup{;}[0,1]}\bra{\XX^1,\XX^2} \leq \bra{1 \vee\omega(0,1)^{\floor{p}}}\rho_{p\textup{-var;}[0,1]}\bra{\XX^1,\XX^2}.
\end{equation}
As a consequence of the super-additivity of controls (cf. \cite[\S8.1]{FV}),
\[
\rho_{p\textup{-var;}[0,1]}(y^1,y^2)\leq 
\omega(0,1)\rho_{p\textup{-}\omega\textup{;}[0,1]}(y^1,y^2),
\]
and combining this inequality with (\ref{e-new-bound}) and (\ref{e-general-case}) concludes the proof. 
\end{proof}


\section{Log-ODE method}\label{s-log-ode}

The log-ODE method is a powerful technique for numerically simulating the solutions of SDEs or, more generally, RDEs \cite{friz2008euler,lyons2014rough}. The method even holds for RDEs in infinite-dimensional Banach space, (see  \cite{bailleul2014flows,boutaib2013dimension,danyu2014}). For this paper we restrict our study of the log-ODE method to level $m=2$, though the scheme can be extended to an arbitrary level $m$.
We give a simplified, tailored overview based on \cite{gyurko2008rough,lyons1998differential,lyons2007extension} and \cite[\S4]{sipilainen1993pathwise}. The reader can also find a concise introduction in \S7 of \cite{lyons2014rough}. The method can also be found in the papers \cite{arous1989flots,castell1993asymptotic,castell1995efficient,castell1996ordinary,fliess,lyons2004cubature}, which were inspired by the pioneering work of Chen \cite{chen1957integration} and Magnus \cite{magnus1954exponential}.

We begin by introducing a special Lie algebra homomorphism. A vector field $V:\RR^q\to\RR^q$ can be interpreted as a differential operator:
\[
V(f) = \sum_{i=1}^q V^i \frac{\partial f}{\partial x_i}, \gap f\in C^1\bra{\RR^q,\RR}.
\]
This allows us to define the Lie bracket operation on two continuously differentiable vector fields:
\[
[V_k,V_l]:= \sum_{i,j=1}^d \bra{ V_k^j \frac{\partial V_l^i}{\partial x_j} - V_l^j \frac{\partial V_k^i}{\partial x_j}} e_i.
\]
Define the Lie algebra homomorphism $\Phi  : (\RR\oplus\RR^d)\oplus [\RR^d,\RR^d] \to \Lip^{\gamma-2}(\RR^q)$ by 
\begin{align*}
\Phi(e_0)&=V_0, \gap \Phi(e_k)=V_k, \gap k = 1,\ldots,d.
\end{align*}
Thus $\Phi$ maps Lie algebra elements in $\mathfrak{g}^{(2)}$ to vector fields on $\RR^q$, while preserving Lie brackets. For example, $\Phi\bra{e_0 + e_i+[e_k,e_l]} = V_0 + V_i+ [V_k,V_l]$. 

Recall (\ref{e-intro-sde}); the Stratonovich SDE to be approximated is given by
\begin{equation}\label{e-stratonovich-sde}
dx(t) = V(x(t)) \circ dW(t) + V_0(x(t))\,dt  = \sum^d_{k=1} V_k(x(t)) \circ dW_k(t)+  V_0(x(t))\, dt , \gap t\in [0,1], 
\end{equation}
where $x_0\in\RR^q$, $V_0 \in \textup{Lip}^1(\RR^q)$ and $V=\{V_k\}_{k=1}^d \in \textup{Lip}^\gamma(\RR^q)$ for $\gamma>2$.
We first consider the problem of approximating the solution of (\ref{e-stratonovich-sde}) at time $t=h$. To do this we consider the classical ODE:
\begin{align*}
dy(t) &= \Phi\bra{he_0+{\log_2 \WW_{0,h}}}(y(t))\, dt, \gap t\in [0,1],\\
y(0) &= x(0),
\end{align*}
where $\WW$ denotes enhanced Brownian motion. 
Adopting the classical flow notation, we write $y(t)=e^{tF}(x(0))$, $t\in [0,1]$, where $F=\Phi\bra{he_0+{\log_2 \WW_{0,h}}}$. Recalling (\ref{e-brown-log-sig}), $\log_2 \WW_{0,h}$ has the form:
\[
\log_2 \bra{\WW_{0,h} }= \sum_{k=1}^d W^{(0)}_k e_k + \sum_{1\leq k < l\leq d} A^{(0)}_{kl}[e_k,e_l]. 
\] 
The stochastic Taylor expansion gives us the following error estimate.

\begin{Proposition}
There exists a constant $C=C(\gamma)$ such that 
\[
\EE\bra{\norm{x(h)-y(1)}_{\RR^q}^2} \leq Ch^{3}.
\]
\end{Proposition}

We can now construct a numerical SDE scheme $\{x^{(j)}\}_{j=0}^N$ by repeating the ODE approximations given in the previous lemma successively over each interval $[jh,(j+1)h]$, where $h=N^{-1}$. 
In particular, set $x^{(0)}=x(0)$, then 
\begin{align}
x^{(j+1)} &= \exp\bra{F_j}\bra{x^{(j)}}, \gap j=0,1,\ldots,N-1,\label{e-ode-definition}\\
\textup{where } F_j &= \Phi\bra{he_0+{{\log_2 \WW_{jh,(j+1)h}}}}.\notag
\end{align}
This is precisely the so-called {log-ODE method}. Moving from local to global error costs half an order of magnitude in the $L^2$-norm using Doob's maximal inequality.

\begin{Theorem}[\textbf{Log-ODE method}]\label{t-log-ode}
There exists a constant $C=C(\gamma)$ such that
\[
\EE\bra{\max_{j=1,\ldots,N} \norm{x^{(j)} - x(jh)}_{\RR^q}^2} \leq Ch^2.
\]
\end{Theorem}

In other words the method provides a strong approximation scheme of order $O(h)$.  A complete proof of Theorem \ref{t-log-ode} for the general case of arbitrary $m$ can be found in \cite{gyurko2008rough}. 

\begin{Remark}
As noted in Remark 1.1 of \cite{ninomiya2008weak}, if one approximates each ODE (\ref{e-ode-definition}) by its order 1 Taylor expansion, then we fall back on the traditional Euler scheme. Similarly, taking the better approximation offered by the order 2 Taylor expansion gives the Milstein scheme. However, the log-ODE method has one important advantage over these more popular schemes.
As outlined in \cite[\S7]{lyons2014rough}, the previous methods are based on Taylor expansions and thus can produce approximations whose law is not absolutely continuous with respect to the measure of the solution on Wiener space. For example, we could consider a SDE whose solution is constrained by its Stratonovich formulation to lie on the unit sphere at all times.
Both the Euler and Milstein schemes are numerically unstable in that they will output approximations which do not live on the sphere (see \cite[\S17.5]{press2007numerical}).
By contrast, the log-ODE method only returns solutions which could have originated from an actual realization of the SDE, (assuming that the ODE solver used is sufficiently accurate over small time steps; for example, the Runge-Kutta method or an adaptive step-size method). This is because the technique restricts the approximations to only flow along the vector fields (and their nested Lie brackets) of the original SDE. In turn, this ensures that feasibility constraints imposed on the original SDE law will also be satisfied by the log-ODE output. A more complicated example could be a system with Hamiltonian vector fields \cite{talay2002stochastic} or the stochastic  volatility example of \cite[\S3]{ninomiya2008weak}.

Even in one dimension this instability becomes apparent. We repeat the simple but well-known counterexample offered by Hutzenthaler, Jentzen and Kloede in \cite{hutzenthaler2011strong}. Consider the following stochastic differential equation with cubic drift and additive noise, where $d=q=1$:
\[
dx_t= dW_t- x_t^3\, dt, \gap x_0 = 0, \gap t\in [0,1].
\]
The corresponding Euler scheme is given by $\tilde{x}^{(j+1)} = \tilde{x}^{(j)} - \bra{\tilde{x}^{(j)}}^3h + W^{(j)}$, with $\tilde{x}^{(0)}=0$. It can be shown that this scheme does not converge strongly or weakly (see \cite[\S3.5.1]{jentzen2011taylor} and \cite[Theorem 3.4]{Kloeden_Neuenkirch_2013} for details):
\[
\lim_{N\to\infty} \EE\bra{\abs{x_1-\tilde{x}^{(N)}}^q} = \infty = \lim_{N\to\infty} \abs{\EE\bra{ \abs{x_1}^q} -\EE\bra{\abs{\tilde{x}^{(N)}}^q }} 
\]
for every $q \in [1,\infty)$. 
\end{Remark}

We conclude our introduction to the log-ODE by considering the case of nilpotent vector fields (for simplicity, let us assume that $V_0=0$). In particular, suppose our vector field system is $3$-nilpotent: $[V_j,[V_k,V_l]]=0$ for every triple $(j,k,l)\in\seq{1,\ldots,d}^3$. In this case the level 2 log-ODE method is exact: $x^{(j)} = x(jh)$ for $j=1,\ldots,N$. To see this, note that the condition implies that 
\[
\Phi\bra{[e_{i_1},[e_{i_2},\ldots, [e_{i_{n-1}}, e_{i_{n}}]]]} = 0 \textup{ for all } (i_1,\ldots,i_n) \in \seq{1,\ldots,d}^n \textup{ where } n\geq 3.
\]
Thus, 
\[
x(h) = \exp\bra{\Phi\bra{\log \WW_{0,h}}}(x(0)) 
= \exp\bra{\Phi\bra{\pi_2\bra{\log \WW_{0,h}}}}(x(0)) = x^{(1)},
\]
and by induction: $x(jh) = x^{(j)}$ for all $j$. As noted in \cite{gyurko2008rough}, we cannot guarantee that equality holds at intermediate times $t\in (jh,(j+1)h)$. 
This phenomenon holds in greater generality for higher levels; if a vector field system is $m$-nilpotent, then the level $m=2$ log-ODE scheme is exact. Similarly, if a vector field system is $2$-nilpotent ($[V_k,V_l]=0$ for all $k,l$), then the level-1 and level-2 log-ODE coincide and are exact. 


\section{Piecewise abelian rough paths}\label{s-rough-view}

We can recast the log-ODE technique in the language of rough path theory using our notion of piecewise abelian rough paths. 
As before the unit interval $[0,1]$ is partitioned into intervals $[jh,(j+1)h]$ of equal length $h=N^{-1}$.
First we define a piecewise abelian rough path.

\begin{Definition}
We call a $p$-rough path $\XX \in C\bra{[0,1],G^{(\floor{p})}(\RR^d)}$ a piecewise abelian $p$-rough path if the identity
\begin{equation}\label{e-pa}
\XX_{s_1,t_1} \otimes \XX_{s_2,t_2} = \XX_{s_2,t_2} \otimes \XX_{s_1,t_1}
\end{equation}
holds for all $(s_1,t_1), (s_2,t_2) \in \Delta_{[jh,(j+1)h]}$ for each $j$. 
\end{Definition}

In the case of $\floor{p}=1$, the definition is trivial since the group $G^{(1)}(\RR^d)=\RR^d$ is abelian.
However, suppose we have a bounded variation path $X:[0,1]\to \RR^d$ such that its level-2 enhancement $\XX:=S_2(X)$ is a piecewise abelian $2$-rough path. Then (\ref{e-pa}) implies that
\[
0=\l[X_{jh,u},X_{u,(j+1)h}\r] = \l[X_{jh,u},X_{jh,u}+X_{u,(j+1)h}\r] = \l[X_{jh,u}, X_{jh,(j+1)h}\r],
\]
and so for every $t\in [jh,(j+1)h]$, $X_{jh,t}$ and $X_{jh,(j+1)h}$ are parallel vectors. Since both vectors start from $X_{jh}$,  we conclude that $X$ is piecewise linear over the increment $[jh,(j+1)h]$. Due to this observation, at least in some heuristic sense, we can think of piecewise abelian rough paths as the natural non-commutative (that is group-valued), equivalent of piecewise linear paths in $\RR^d$. 

We now return to our original task of rewriting the log-ODE method in terms of this new class of rough paths.
The original SDE (\ref{e-intro-sde}) can be rewritten as a RDE with drift, driven by enhanced Brownian motion $\WW$:
\begin{equation}\label{e-sde-rde}
dx_t = V\bra{x_t} d\WW_t + V_0\bra{x_t}dt, \gap x_0 \in \RR^q.
\end{equation} 
For a proof of this equivalence we refer to Theorem 17.3 of \cite{FV}. 
From this perspective, the underlying idea of the log-ODE method is to approximate solutions of the SDE/RDE (\ref{e-intro-sde}/\ref{e-sde-rde}) by replacing $\WW$ with its much simpler $N$-step piecewise abelian approximation $\WW^h$.  
In order to introduce this finite-dimensional approximation we need to set notation for the Brownian and L\'{e}vy area increments as follows:
\begin{align*}
W^{(j)}_k :&= W_k(jh,(j+1)h) := W_k((j+1)h)-W_k(jh),\\
A^{(j)}_{kl} :&= \frac{1}{2}\int^{(j+1)h}_{jh} \bra{W_k(jh,t)\, dW_l(t) - W_l(jh,t)\, dW_k(t)} = A_{kl}(jh,(j+1)h).
\end{align*}

\begin{Definition}\label{d-pa-rp}
Define $\WW^h \in C^{p\textup{-var}}([0,1],G^{(2)}(\RR^d))$ to be the piecewise abelian rough path given iteratively by $\WW^h_0 = 1$, then 
\[
\WW^h_{0,t} = \WW^h_{0,jh} \otimes \exp_2\bra{(t-jh)h^{-1}\xi^{(j)}}, \gap t\in [jh, (j+1)h], 
\]
where 
\[
\xi^{(j)} := \sum^d_{k=1} W_k^{(j)} e_k + \sum_{1\leq k < l \leq d} A^{(j)}_{kl} [e_k,e_l] \in \RR^d \oplus [\RR^d,\RR^d] = \mathfrak{g}^{(2)}(\RR^d). 
\] 
\end{Definition}

Chen's Theorem and the identity $\WW^h_{s,t} = \bra{\WW^h_{jh,s}}^{-1}\otimes \WW^j_{jh,t}$ give $\WW^h_{s,t} = \exp_2\bra{(t-s)\xi^{(j)}}$ for $[s,t]\subseteq [jh,(j+1)h]$. Then the Baker-Campbell-Hausdorff formula confirms that $\WW^h$ actually is a piecewise abelian rough path.

\begin{Remark}
One can also think of $\WW^h$ as the $N$-step random walk in the $2$-nilpotent Lie group $G^{(2)}(\RR^d)$ with i.i.d increments $\exp_2(\xi^{(j)}) \in G^{(2)}(\RR^d)$ (see \cite{breuillard2005local,pap1993central}). Here the corresponding increments $\xi^{(j)}$ live in the $2$-nilpotent Lie algebra $\mathfrak{g}^{(2)}(\RR^d)$. 
Previous research has studied this random walk interpretation of what we have called piecewise abelian rough paths. In \cite{breuillard2009random} the authors used the Central Limit Theorem in nilpotent Lie groups to prove a Donsker-type weak limit theorem for similar random walks converging to enhanced Brownian motion in a particular rough path H\"{o}lder topology. 
\end{Remark}

The rough path $\WW^h$ is certainly a $2$-geometric rough path as it can be approximated in the $2$-variation topology by the signatures of a sequence of bounded variation paths.  
Moreover, since the rough path is defined by linear interpolation in the Lie algebra $\mathfrak{g}^{(2)}(\RR^d)$, over each interval $[jh,(j+1)h]$ $\WW^h$ is also the shortest rough path candidate, (measured in $p$-variation), with its corresponding group increment $\WW^h_{jh,(j+1)h}$ matching $\xi^{(j)} \in \mathfrak{g}^{(2)}(\RR^d)$. Therefore $\WW^h$ is also piecewise geodesic. 

\begin{Remark}
Note that $\WW^h$ is (almost surely) not the lift of an actual path in $\RR^d$; that is, $\WW^h\neq S_2(Y)$ for some stochastic process $Y:[0,1]\to\RR^d$. 
Indeed if otherwise, $Y$ would enclose non-zero area over each increment yet be piecewise linear, implying a contradiction. 
Given an increment $\xi^{(0)} \in G^{(2)}(\RR^d)$, the interesting problem of finding a helix path $\gamma \in C\bra{[0,1],\RR^d}$ with minimal length such that $S_2(\gamma)=\xi^{(0)}$ is known as \textit{reconstruction} (\cite[\S5]{bass2002extending} and \cite{lyons2005sound}).
\end{Remark}

Having defined $\WW^h$, we now reformulate the log-ODE method in terms of a rough differential equation:

\begin{Proposition}\label{p-log-ode-coincide}
Let ${y} \in C([0,1],\RR^q)$ be the solution of the RDE with drift
\begin{equation}\label{e-rde-log-ode-result}
dy_t = V\bra{y_t} d\WW^h_t + V_0\bra{y_t}dt, \gap y_0 = x_0 \in \RR^q, 
\end{equation}
and let $\tilde{x}^h \in C([0,1],\RR^q)$ denote the entire path of the approximation produced by the log-ODE method given by (\ref{e-ode-definition}). Then $y$ and $\tilde{x}^h$ coincide at all times: $y(t)=\tilde{x}^h(t)$ for all $t \in [0,1]$. 
\end{Proposition}
\begin{proof}
The claim follows immediately from \cite[Theorem 2]{friz2009rough}.
\end{proof}

\begin{Remark}
To prove the weaker statement:
\[
y\bra{jh}=\tilde{x}^h(jh){=x^{(j)}}\gap \textup{ for all }  j=1,\ldots,N,
\] 
a more intuitive proof would be to compare the stochastic Taylor expansions of $y(jh)$ and $x^{(j)}$. One finds that the expansions agree at the times $t=jh$ up to every order (cf. \cite{friz2008euler}) and hence the approximations must coincide. 
\end{Remark}

We can also consider the rough path lift of $\WW^h$ to level $\kappa>2$:
\[
S_\kappa(\WW^h) \in C^{p\textup{-var}}([0,1],G^{(\kappa)}(\RR^d)).
\]
This lift is unique in the sense that the $p$-variation of $S_\kappa(\WW^h)$ is equal, (up to multiplicative constants), to that of $\WW^h$. Indeed, by the Lipschitz-continuity of the lift operator (\cite[Theorem 9.5]{FV}), given $p\in [2,3)$ there exists a constant $C=C(\kappa,p)$ such that 
\[
\pnorm{\WW^h}{p} \leq \pnorm{S_\kappa(\WW^h)}{p} \leq C\pnorm{\WW^h}{p}.
\]
We can ask what is $S_\kappa(\WW^h)$ precisely? To answer this we treat the increments 
\[
\xi^{(j)} = W^{(j)} + A^{(j)} \in \RR^d \oplus [\RR^d, \RR^d] = \mathfrak{g}^{(2)}(\RR^d)
\]
as Lie increments $\hat{\xi}^{(j)}$ in the larger free $\kappa$-step nilpotent Lie algebra $\mathfrak{g}^{(\kappa)}(\RR^d) $. Then we can define a piecewise abelian $\kappa$-rough path $\ZZ^h \in C([0,1],G^{(\kappa)}(\RR^d))$ by $\ZZ^h_0 = 1$ and
\[
\ZZ^h_{0,t} = \ZZ^h_{0,jh} \otimes \exp_\kappa\bra{(t-jh)h^{-1}\hat{\xi}^{(j)}}, \gap t\in [jh,(j+1)h].
\]
We claim that $S_{\kappa}(\WW^h) = \ZZ^h$. To prove this, first note that the first two levels of $\WW^h$ and $\ZZ^h$ agree:
$\pi_{0,2}(\ZZ^h_{s,t} - \WW^h_{s,t})=0$ for all $s,t\in [0,1]$.
Moreover, by its construction $\ZZ^h$ is certainly a multiplicative functional. By the rough path extension theorem of \cite[Theorem 3.7]{lyons2004differential} it remains to show that the $p$-variation of the higher tensor level increments of $\ZZ^h$ are controlled by that of $\WW^h$. By using Chen's Theorem, it suffices to consider the $p$-variation over the single interval $[0,h]$. Exploiting the piecewise geodesic nature of $\WW^h$ and $\ZZ^h$, there exists constants $C_i=C_i(\kappa)>0$, $i=1,2$, such that  
\[
\norm{\WW^h}_{p\textup{-var;}[0,h]} = \norm{\xi^{(0)}}_C = C_1\bra{ \norm{W^{(0)}} \vee \sqrt{\norm{A^{(0)}}}} = C_2  \norm{\hat{\xi}^{(0)}}_C = C_2 \norm{\ZZ^h}_{p\textup{-var;}[0,h]},
\]
where the Carnot-Carath\'{e}odory norms are taken over $G^{(2)}(\RR^d)$ and $G^{(\kappa)}(\RR^d)$ respectively. We conclude that $S_\kappa(\WW^h)=\ZZ^h$, as claimed. 

\begin{Remark}\label{r-eq}
Our interest in the enhancement $S_\kappa(\WW^h)$ comes from the following observation that will become critical in the proof of our main result: if we replace the driving rough path $\WW^h$ by $S_\kappa(\WW^h)$ in the RDE (\ref{e-rde-log-ode-result}), the solution of Proposition \ref{p-log-ode-coincide} remains the same, (although we must assume stronger conditions on our vector fields in order for the RDE to have a unique solution: namely that $V = \seq{V_k}_{k=1}^d \in \textup{Lip}^\gamma$, where $\gamma>\kappa>2$ instead of simply $\gamma>2$).
To see this directly, we note that by their construction the log-signatures of $\WW^h$ and $S_\kappa(\WW^h)$ coincide over each increment $[jh,(j+1)h]$:
\[
\log\bra{S_\kappa(\WW^h)_{jh,(j+1)h}} = \log\bra{\WW^h_{jh,(j+1)h}}.
\]
This is a slight abuse of notation since the left-hand side lives in a larger Lie algebra,  albeit with zero terms of multiplicity greater than $2$.
Therefore RDEs driven by $S_\kappa(\WW^h)$ are precisely the same as that driven by the original $\WW^h$. For a formal proof see the perturbation result of \cite[Theorem 12.14]{FV} and \cite[Theorem 2]{friz2009rough}, (the latter covers the drift case). An important consequence of this is that the log-ODE technique produces the same approximation points whether we use $\WW^h$ or $S_\kappa(\WW^h)$ to drive the RDE in Proposition \ref{p-log-ode-coincide}. In fact, this phenomenon holds for arbitrary rough paths and their lifts.
\end{Remark}


\section{Gaussian approximations of L\'{e}vy area}

In practice, it is difficult to drive RDEs with the piecewise abelian rough path $\WW^h$ because we must be able to generate L\'{e}vy area increments $A^{(j)}$, (which is numerically challenging if $d>2$). 
Following the recent papers \cite{davie2014kmt,davie2014pathwise} of Davie, we propose another piecewise abelian $2$-rough path $\XX^h$ which substitutes each $A^{(j)}$ with a suitable Gaussian random variable $B^{(j)}$, thereby being much easier to generate. In particular, the $A^{(j)}$ and $B^{(j)}$ share the same mean and covariance structure, (that is, we are moment matching up to order 2). But before we go into detail and define $\XX^h$, let us closely examine the area increments. The following lemma gives a simple decomposition of $A^{(j)}$ into parts dependent and independent of the corresponding Brownian increment $W^{(j)}:=W(jh,(j+1)h)$. 

\begin{Lemma}\label{l-area-decomposition}
For all $1\leq k < l \leq d$:
\[
A^{(j)}_{kl} = \zeta^{(j)}_k W^{(j)}_l - \zeta^{(j)}_l W^{(j)}_k + K^{(j)}_{kl},
\]
where the $\zeta^{(j)}_k$: $k=1,\ldots,d$, $K^{(j)}_{kl}$: $1\leq k < l \leq d$, are mutually uncorrelated (but not independent), independent of $W^{(j)}$ and have mean zero. Moreover,  $\Var{\zeta^{(j)}_k} = \frac{h}{12}$ and $\Var{K^{(j)}_{kl}}=\frac{h^2}{12}$. 
\end{Lemma}
\begin{proof}
We suppose $j=0$ for simplicity and begin by decomposing the $A^{(0)}$ increment into parts dependent and independent of the Brownian increment $W(h)$. To this end, following \cite[\S7]{davie2014pathwise}, we can write $W_k(t)=h^{1/2}B_k(t/h)+th^{-1/2}V_k$ for $t\in [0,h]$, where $B_1,\ldots,B_d$ are independent standard Brownian bridges on $[0,1]$ and $V_k=h^{-1/2}W_k(h)$ are independent $N(0,1)$ (and are independent of the $B_j$). Also write $B_0(t)=t$ and set 
\[
K_\alpha := \int^1_0 \int^{t_l}_0 \ldots \int^{t_2}_0 dB_{j_1}(t_1)\ldots dB_{j_l}(t_l)
\] 
for an index $\alpha=(j_1,\ldots,j_l) \in \{0,1\ldots,d\}^l$. For such an index it can be shown that 
\[
I_\alpha := \int^h_0 \int^{t_l}_0 \ldots \int^{t_2}_0 dW_{j_1}(t_1)\ldots dW_{j_l}(t_l) = h^{(l(\alpha)+n(\alpha))/2}\sum_{\beta=(i_1,\ldots,i_l)} K_\beta \prod_{k:i_k<j_k} V_{j_k},
\]
where the sum is over all $\beta=(i_1,\ldots,i_l)$ such that for each $k\in\seq{1,\ldots,l}$ we have either $i_k=j_k$ or $i_k=0<j_k$. Here we have used $l(\alpha)$ and $n(\alpha)$ to denote the length and number of zero entries of $\alpha$ respectively. 
Noting that $K_{kl}=-K_{lk}$ for $0\leq k < l$, it follows that 
\begin{align*}
A^{(0)}_{12} 
&= \frac{1}{2}\bra{I_{12}-I_{21}} = h\bra{K_{10}V_2 - K_{20}V_1+K_{12}},
\end{align*}
where 
\[
K_{12} = \int^1_0 B_1(t)\, dB_2(t) \textup{ and } K_{j0} = \int^1_0 B_j(t)\, dt \textup{ for } j=1,2. 
\]
Thus $\zeta^{(0)}_j := h^{1/2}K^{(0)}_{j0}$ for $j=1,2$, and $K^{(0)}_{12} := hK_{12}$ gives the claimed decomposition. 
The variances follow from It\^{o}'s isometry. For details we refer to Lemma 7 of \cite{davie2014pathwise}. 
\end{proof}

The fact that $\zeta^{(j)}$ and $K^{(j)}$ are not independent makes them (and consequently $A^{(j)}$) very difficult to simulate numerically. A natural solution would be to approximate these two variables with normal random variables $z^{(j)}$, $\lambda^{(j)}$ with the correct mean and moments, to produce a Gaussian approximation $B^{(j)}$ for $A^{(j)}$. Since uncorrelated Gaussian random variables are necessarily independent, simulation is much easier. This is precisely what Davie proposes:
\begin{equation}\label{e-davie-normal-area}
B^{(j)}_{kl} := z^{(j)}_k \tilde{W}^{(j)}_l - z^{(j)}_l \tilde{W}^{(j)}_k + \lambda^{(j)}_{kl},
\end{equation}
where $z^{(j)}_k$ and $\lambda^{(j)}_{kl}$ are independent normal random variables with $z^{(j)}_k \sim N(0,\frac{h}{12})$ and $\lambda^{(j)}_{kl}\sim N(0,\frac{h^2}{12})$. Here $\tilde{W}^{(j)}\sim N(0,h)$ but this increment may not necessarily be equal to the original Brownian increment $W^{(j)}$. As before, the $z^{(j)}_k$, $\lambda^{(j)}_{kl}$ are all independent of $\tilde{W}^{(j)}$. 

In the same fashion as the construction of $\WW^h$ in Definition \ref{d-pa-rp}, we define $\XX^h\in C\bra{[0,1],G^{(2)}(\RR^d)}$ to be the piecewise abelian, geometric $2$-rough path given by $\XX^h_0 = 1$ and 
\[
\XX^h_{0,t} = \XX^h_{0,jh} \otimes \exp_2\bra{(t-jh)h^{-1}\eta^{(j)}}, \gap t\in [jh, (j+1)h],
\]
where the increments $\eta^{(j)} \in \mathfrak{g}^{(2)}(\RR^d)$ are defined as:
\[
\eta^{(j)} := \sum^d_{k=1} \tilde{W}^{(j)}_k e_k + \sum_{1\leq k < l \leq d} B^{(j)}_{kl} [e_k, e_l]. 
\]
As with $\WW^h$, we may consider the $\kappa$-lift $S_\kappa(\XX^h)$ of $\XX^h$ and, as before, RDEs driven by $\XX^h$ and $S_\kappa(\XX^h)$ coincide. Moreover, again using Theorem 2 of \cite{friz2009rough} we can prove that the approximation scheme $\{x^h_{j}\}_{j=0}^N$ given by (\ref{e-scheme}) coincides with the solution of the following RDE with drift:
\[
dz_t = V\bra{z_t} d\XX^h_t +V_0\bra{z_t}dt, \gap z_0 = x_0\in\RR^q. 
\]
In particular, $z_{jh} = x^h_j$ for $j=1,\ldots,N$. 


\section{Wasserstein coupling of L\'{e}vy area}

In this section we construct a probabilistic coupling of the constituent random variables of $\WW^h$ and $\XX^h$ so as to achieve a coupling of the two piecewise abelian rough paths. We directly follow the dyadic coupling argument of \cite{davie2014kmt} in which Davie presented a numerical approximation scheme for SDEs based on a variant of the famous 1975 theorem of Koml\'{o}s, Major and Tusnady \cite{komlos1975approximation}. 
The latter result is a form of the simultaneous Central Limit Theorem using couplings. As Davie writes, it states that if $\PP$ is a suitably non-degenerate probability measure on $\RR$ with mean zero, variance $1$ and zero third moment, then there exists a universal constant $C>0$ such that the following holds: for each $n\in\NN$, one can construct a probability space on which there exists a sequence of i.i.d random variables $X_1,\ldots,X_n$ with law $\PP$ and a corresponding sequence of i.i.d $N(0,1)$ variables $Y_1,\ldots,Y_n$ such that 
\[
\max_{k=1,\ldots,n} \abs{ \sum_{i=1}^k (X_i - Y_i)}_{L^2} \leq C.
\]
This original KMT Theorem was then extended to vector random variables by Einmahl in \cite{einmahl1989extensions} and then Zaitsev established the result for the case of non-identical distributions which are uniformly non-degenerate in a series of papers \cite{zaitsevALL}. 
For approximating non-Gaussian distributions $\PP$, both coupling results are proven to be optimal among all couplings (see \cite{zaitsev1996estimates,zaitsev1998multidimensional}).

\begin{Remark}
Following Zaitsev's work, Davie proves his own variation which allows the underlying distributions to be random themselves. One cannot just apply the original KMT theorem or Zaitsev's version to the random walk composed of the L\'{e}vy area increments because  one would be unable to say anything about how close the increments of the Brownian motion $W^{(j)}$ and the Gaussian approximation $\tilde{W}^{(j)}$ are. Moreover, in the case of $d=2$, directly applying the classical KMT theorem to the one-dimensional random walk composed of the $A^{(j)}$ increments would give a Wasserstein rate of convergence of $O(-\sqrt{h}\log h)$ by scaling (see \cite[Theorem 1]{zaitsev1996estimates}). Therefore a more sophisticated argument is needed.
\end{Remark}

For our coupling of $\WW^h$ and $\XX^h$ we change a small part of Davie's argument. 
In the original paper \cite{davie2014kmt}, the approximation points of the Milstein scheme were coupled as a vector with the corresponding points of the SDE solution. That is, speaking from the perspective of rough path theory, the coupling took place at the output side of the It\^{o} map $\Xi$.  In contrast, our approach is to use Davie's coupling argument at the input side of $\Xi$. Moreover, in his case Davie coupled the Brownian increments $W^{(j)}$ and $\tilde{W}^{(j)}$ (that is, they were not necessarily equal). In our application, we assume that $\WW^h$ and $\XX^h$ share the same increments: $W^{(j)}=\tilde{W}^{(j)}$ for all $j$. As we will see, this simplification makes calculations using the Baker-Campbell-Hausdorff formula (Theorem \ref{t-BCHf}) much easier to handle. 
What remains is to couple the L\'{e}vy area increments $A^{(j)}$ with the Gaussian approximations $B^{(j)}$ defined above. 

Before stating the coupling result let us introduce some notation. 
Without loss of generality suppose that $N=h^{-1}=2^m$ for some integer $m$. This can always be arranged by extending the SDE to the interval $[0,2^mh]$, where $m$ is the smallest integer such that $2^m\geq N$. Define a dyadic set to be a subset $E \subseteq \{0,1,\ldots,2^m-1\}$ of the form 
\[
E=\seq{k2^n,k2^n+1,\ldots,(k+1)2^n-1},
\]
for some $n\in\{0,1,\ldots,m\}$ and $k\in\{0,1,\ldots,2^{m-n}-1\}$. Define the $[\RR^d,\RR^d]$-valued partial sums
\[
\gamma_E := \sum_{r\in E} A^{(r)}, \gap \lambda_E := \sum_{r\in E} B^{(r)} \textup{ for all } E\subseteq \seq{0,1,\ldots,2^m-1}.
\]

\begin{Proposition}\label{p-kmt-coupling}
There exists a constant $C>0$ and a probability space on which one can define $\{W^{(j)}\}_{j=0}^{2^m-1}$, and $\gamma_E, \lambda_E$ for all subsets $E \subseteq \{0,1,\ldots,2^m-1\}$ (both using the $W^{(j)}$), such that 
\[
\abs{\norm{\gamma_E - \lambda_E}_{(\RR^d)^{\otimes 2}}}_{L^{5/2}} \leq Ch\log (h^{-1}).
\]
In the special case of $E$ being a dyadic subset we have  $\abs{\norm{\gamma_E - \lambda_E}}_{L^{5/2}} \leq Ch$. 
\end{Proposition}

\begin{Remark}
Regardless of our coupling, by scaling we automatically have
\[
\abs{\norm{\gamma_E - \lambda_E}}_{L^{5/2}} \leq \abs{\norm{\gamma_E}}_{L^{5/2}} + \abs{\norm{\lambda_E}}_{L^{5/2}} \leq Ch \textup{ when } E=\seq{i},
\]
and so Proposition \ref{p-kmt-coupling} is trivial when considered locally, (that is on single intervals). The point is that the coupling performs well globally and it is this property which we will exploit when considering $p$-variation of the coupled lifted piecewise abelian rough paths. Heuristically, one can think of this as a probabilistic analogue of the fact that there exists paths which are far apart in $1$-variation but very close in $p$-variation for large $p\gg1$.
\end{Remark}

We repeat again: the proof is essentially a special case of Davie's original Theorem 1 of \cite{davie2014kmt}. 

\begin{proof}
Let $\Gg$ denote the $\sigma$-algebra generated by the increments $W^{(0)}, \ldots, W^{(N-1)}$. 
For each $r$ define a random vector $X^{(r)}\in \RR^{\frac{d}{2}(d+1)}$ by $X^{(r)}_k = (\frac{h}{12})^{-1/2}\zeta^{(r)}_k$ for $k=1,\ldots,d$ and $X^{(r)}_{\frac{k}{2}(2d-k-1)+l}= (\frac{h^2}{12})^{-1/2} K^{(r)}_{kl}$ for $1\leq k < l \leq d$. 
Then, (conditional on $\Gg$), $X^{(r)}$ has mean zero and covariance matrix $I_{\frac{d}{2}(d+1)}$. We can then write 
\[
h^{-1}A^{(r)} = G_r X^{(r)},
\]
where $G_r$ is a $\frac{d}{2}(d-1) \times \frac{d}{2}(d+1)$ matrix defined in terms of the $W^{(j)}$. Specifically
\[
G_{r}=\frac{1}{\sqrt{12}}\left(\begin{array}{c|c}
M_{r} & I_{\frac{d}{2}(d-1)}\end{array}\right), 
\]
setting $M_r$ to be the $\frac{d}{2}(d-1) \times d$ matrix defined by the rows
\[
\bra{M_r}_{\frac{k}{2}(2d-k-1)+(l-d)} = h^{-1/2} \bra{ W^{(r)}_l e_k - W^{(r)}_k e_l}.
\] 
This makes $M_r$ have the form:
\[
M_r = h^{-1/2}
\left(\begin{array}{cccccc}
W_{2}^{(r)} & -W_{1}^{(r)} & 0 & \cdots & 0 & 0\\
W_{3}^{(r)} & 0 & -W_{1}^{(r)} & \cdots & 0 & 0\\
\vdots & \vdots & \vdots & \ddots & \vdots & \vdots\\
0 & 0 & 0 & \cdots & W_{d}^{(r)} & -W_{d-1}^{(r)}
\end{array}\right).
\]
In the same way we have $h^{-1}B^{(r)} = G_r \tilde{X}^{(r)}$, where $\tilde{X}^{(r)}$ is $N(0,I_{\frac{d}{2}(d+1)})$. 

It follows that $G_r G_r^t = \frac{1}{12}(I + M_r M_r^t)$ is a positive-definite symmetric matrix. Since $h^{-1/2}W^{(r)} \sim N(0,1)$, certainly $\EE\bra{\norm{G_r}^q} \leq C(q)$ for all $q\geq 1$. 
Moreover, the eigenvalues of $G_r G_r^t$ are bounded below by $\frac{1}{12}$, hence $\norm{(G_r G_r^t)^{-1}} \leq 12$. Note that conditional on $\Gg$, $A^{(r)}$ and $B^{(r)}$ have the same covariance matrix $h^2 G_r G_r^t$.  

For each dyadic set $E$ of size $2^n$ define the matrix $H_E=2^{-n}\sum_{r\in E} G_r G_r^t$. Since, conditional on $\Gg$, the random variables $A^{(0)}, \ldots, A^{(N-1)}$ are independent, $H_E$ is the (conditional) covariance matrix of $Y_E:= 2^{-n/2}h^{-1}\gamma_E$. Similarly $H_E$ is also the (conditional) covariance matrix of $Z_E:= 2^{-n/2}h^{-1}\lambda_E$.
Note that $H_E^{-1}$ is well defined since $G_r G_r^t$ are positive-definite symmetric matrices. Moreover, $H_E = \frac{1}{12}( I +  2^{-n} \sum_{r\in E} M_r M_r^t)$, so the eigenvalues of $H_E$ are also bounded below by $\frac{1}{12}$ and hence $\norm{H_E^{-1}} \leq 12$. It follows that $\EE\bra{\norm{H_E}^q} \leq C(q)$, $q\geq 1$. 

Having established suitable $L^q$-bounds on the matrices $H_E$ and $H_E^{-1}$, the proof then follows precisely the same course as that of \cite{davie2014kmt}. The idea of Davie's proof is to construct couplings of $Y_E$ and $Z_E$ recursively, starting with the base case $E_0 = \{0,1,\ldots,2^m-1\}$ and proceeding by successive bisection. Since the proof is precisely the same we omit the details for the sake of brevity. 

\begin{Proposition}[Davie \cite{davie2014kmt}]
Suppose that for each $q\geq 1$, there exists a constant $C=C(q)$ such that $
\EE\bra{\norm{H_E}^q}, \EE\bra{\norm{H_E^{-1}}^q} \leq C$ for every dyadic set $E$. Then there exists a constant $C>0$ and a probability space on which we can define $\{W^{(j)}\}_{j=0}^{2^m-1}$, and $\gamma_E, \lambda_E$ for all subsets $E\subseteq \{0,1,\ldots,2^m-1\}$, such that
\[
\abs{\norm{Y_E - Z_E}}_{L^{5/2}} \leq C2^{-n/2}
\] 
for every dyadic set $E$ of size $2^n$.
\end{Proposition}

Returning to the proof of Proposition \ref{p-kmt-coupling}, Davie's result yields 
\[
\abs{\norm{\gamma_E - \lambda_E}}_{L^{5/2}} = 2^{n/2}h\abs{\norm{Y_E-Z_E}}_{L^{5/2}} \leq Ch,
\]
whenever $E$ is a dyadic set of size $2^n$. A general, not necessarily dyadic, subset $E$ can be expressed as the disjoint union of at most $\log_2 N=\log_2 (h^{-1})$ dyadic sets $E_1,\ldots,E_k$ of different sizes. It follows that 
\[
\abs{\norm{\gamma_E-\lambda_E}}_{L^{5/2}} \leq \sum_{j=1}^k \abs{\norm{\gamma_{E_j}-\lambda_{E_j}}}_{L^{5/2}} \leq Ch\log ( h^{-1}).
\]
The proof is complete. 
\end{proof}

\begin{Remark}
One could argue that we could have saved ourselves trouble by simply applying Davie's original coupling result to the SDE defining the L\'{e}vy area. This approach is perfectly sound since this SDE satisfies the non-degeneracy condition (\ref{e-davie-non-deg}) and we would end up with an approximation $\{\tilde{B}^{(j)}\}_{j=0}^{N-1}$  of the L\'{e}vy area increments such that 
\[
\tilde{\lambda}_E := \sum_{r\in E} \tilde{B}^{(r)} \gap\textup{satisfies}\gap \abs{\norm{\gamma_E-\tilde{\lambda}_E}}_{L^2} \leq Ch\log (h^{-1}), \gap \forall E\subseteq \seq{0,1,\ldots,N-1}.
\]
However, we would not be able to guarantee that the increments $W^{(j)}$ and $\tilde{W}^{(j)}$ making up $\gamma_E$ and $\tilde{\lambda}_E$ would necessarily be equal. 
As mentioned, we will see in the proceeding sections that having the increments $W^{(j)}=\tilde{W}^{(j)}$ equal in the definition of $\WW^h$ and $\XX^h$ greatly reduces the complexity of some of the computations involving the Baker-Campbell-Hausdorff formula.
\end{Remark}

As defined in the introduction, let $\Chi^h$ and $\Theta^h$ denote the $[\RR^d,\RR^d]$-valued $N(=h^{-1})$-step random walks made up of the increments $A^{(j)}$ and $B^{(j)}$ respectively.
Proposition \ref{p-kmt-coupling} establishes a non-local coupling in the sense that the $\Chi^h$ and $\Theta^h$ are \textit{not} adapted to the same filtration. Importantly, this means that the error given by the discrete process $\delta^h_j:= \Chi^h_j - \Theta^h_j$ is not a martingale, and so we cannot employ Doob's maximal $L^2$-martingale inequality to arrive quickly and painlessly at a useful maximal inequality for the coupling error. 


\section{Coupling piecewise abelian rough paths}\label{s-measure}

The coupling provided by Proposition \ref{p-kmt-coupling} automatically induces a coupling between $\WW^h$ and $\XX^h$. One may ask how well this coupling performs in a given rough path metric topology. 
Since we are interested in numerically approximating SDEs using these rough paths, we need to employ the useful RDE Lipschitz estimates of \cite{bayer2013rough,cass2013integrability,friz2013integrability}. These estimates use the the inhomogeneous $p$-variation metric and so this is the metric we focus on (see Proposition \ref{p-lipschitz-rde}). 

It can be shown that there is an upper bound of $O(-h\log h)$ on the performance of any coupling of $\WW^h$ and $\XX^h$ under the restriction that their underlying Gaussian increments agree. To be precise, the performance is measured in the Wasserstein metric using the inhomogeneous $p$-variation metric as the cost function. This is the content of Corollary \ref{c-optimal-bound} below. 
The first step of proving this upper bound is the following proposition which is built directly upon a similar result established in the final example of \cite{davie2014kmt}. 
For ease of notation, denote the Cartesian product $[\RR^d,\RR^d]^{\times N}$ by $[\RR^d,\RR^d]^N$.

\begin{Proposition}\label{p-davie-counter}
Suppose $d\geq 2$ and let $\mu$ and $\nu$ denote the laws of $\{\Chi^h_j\}_{j=1}^N$ and $\{\Theta^h\}_{j=1}^N$ in $[\RR^d,\RR^d]^N$. Set 
\[
\Mm_*(\mu,\nu) := \seq{\Psi : [\RR^d,\RR^d]^N \to [\RR^d,\RR^d]^N : \Psi_*(\mu)=\nu \textup{ and } \Psi \textup{ measurable}}.
\]
Then for all $q\geq 1$, there exists a constant $c=c(q)>0$ such that 
\[
\Ww^*_q(\mu,\nu) := \bra{\inf_{\Psi \in \Mm_*(\mu,\nu)} \int_{[\RR^d,\RR^d]^N} \bra{\max_{k=1,\ldots,N} \norm{\sum_{j=1}^k \bra{x_j-\Psi(x)_j}}_{(\RR^d)^{\otimes 2}}}^q \mu(dx)}^{1/q} \geq ch\log(h^{-1}).
\]
\end{Proposition}
\begin{proof}
It suffices to prove the proposition for $d=2$. 
Consider the SDE defining L\'{e}vy area:
\[
dx_1 = dW_1, \gap dx_2 = dW_2, \gap dx_3 = \frac{1}{2}(x_1\, dW_2 - x_2\, dW_1), 
\]
on the time interval $[0,1]$, with initial condition $x_i(0)=0$. It follows that $x_3(t)=A_{12}(0,t)$ and by direct calculation it can be shown that the corresponding Milstein scheme $\{{y}^{(j)}\}_{j=1}^N$ is exact in that $y^{(j)}=x(jh)$. In particular, ${y}_{3}^{(j)}=A_{12}(0,jh)$ for all $j$. 
At each time $t=jh$, the accumulated area of $\Chi^h$ corresponds to $A_{12}(0,jh)={y}_{3}^{(j)}$, while the accumulated area of $\Theta^h$ equals $\tilde{x}^{(j)}$, the point produced by the approximation scheme proposed in \cite{davie2014kmt} using the $B^{(j)}$ increments (where we do not assume that $\tilde{W}^{(j)}$ is necessarily equal to $W^{(j)}$ in our coupling).
Davie proves that there is a constant $c>0$ such that, for every integer $N$ and any coupling between the random variables $\tilde{W}^{(j)},z^{(j)}, \lambda^{(j)}$ (used to define $B^{(j)}$) and the Brownian motion $W$ and its L\'{e}vy area increments $A^{(j)}$, we have 
\begin{equation}\label{e-davie-example}
\PP\bra{\max_{j=1,\ldots,N} \norm{\tilde{x}^{(j)}-x(jh)}_{\RR^3} \geq ch\log(h^{-1})} > 2^{-1}.
\end{equation}
Markov's inequality then finishes the proof. 
\end{proof}

\begin{Corollary}\label{c-optimal-bound}
Suppose $d\geq 2$, $p\in [2,3)$, and $q\geq 1$. Let $\mu,\nu$ denote the respective measures of the piecewise abelian rough paths $\WW^h$ and $\XX^h$ on $G\Omega_p(\RR^d)$, and set 
\[
\Mm_\star(\mu,\nu) := \seq{\Psi : G\Omega_p(\RR^d) \to G\Omega_p(\RR^d) : \Psi_*(\mu) = \nu, \Psi \textup{ measurable and } \pi_1\bra{\Psi(\XX)}=\pi_1(\XX)}.
\]
Then there exists a constant $c=c(q)>0$ such that 
\[
\Ww_q^\star(\mu,\nu) := \bra{\inf_{\Psi\in\Mm_\star(\mu,\nu)} \int_{G\Omega_p(\RR^d)} \rho_{p\textup{-var;}[0,1]} \bra{\XX,\Psi(\XX)}^q \mu(d\XX)}^{1/q} \geq ch\log (h^{-1}). 
\]
\end{Corollary}
\begin{proof}
We apply Proposition \ref{p-davie-counter} and restrict ourselves to couplings with the underlying Gaussian increments of $A^{(j)}$ and $B^{(j)}$ are equal; that is, $\tilde{W}^{(j)}=W^{(j)}$ for all $j$. Since $\pi_1(\WW^h) = \pi_1(\XX^h)$, the difference $\pi_2\bra{\WW^h-\XX^h}$ lies in the centre $[\RR^d,\RR^d]$ of the Lie group $G^{(2)}(\RR^d)$ and is thus equal to the difference between the piecewise linear interpolations of the accumulated areas of $\WW^h$ and $\XX^h$. These accumulated areas correspond precisely to the piecewise linear interpolation of the random walks $\Chi^h$ and $\Theta^h$. The result then follows immediately from the upper bound provided by Proposition \ref{p-davie-counter}. Indeed, this latter proposition establishes a lower bound for all couplings of $A^{(j)}$ and $B^{(j)}$, while we are only interested couplings where the underlying Brownian increments equal. Certainly the lower bound still holds for this particular subset of couplings.  
\end{proof}

We use the notation $\Ww^*_q(\mu,\nu)$ and $\Ww^\star_q(\mu,\nu)$ to differentiate from the Wasserstein metric $\Ww_q(\mu,\nu)$ on $C\bra{[0,1],\RR^q}$ as defined for Theorem \ref{t-big}, (where $q=2$). Note that $\Ww^*_q(\cdot)$ is a legitimate Wasserstein metric on $[\RR^d,\RR^d]^{\times N}$, while $\Ww^\star_q(\cdot)$ is not a true metric on $G\Omega_p(\RR^d)$. For the quantity $\Ww^\star_q(\mu,\nu)$ to be well-defined we require that the set $\Mm_*(\mu,\nu)$ is non-empty, which will only be the case if we have the following equality of the pushforward measures:
\[
\bra{\pi_1}_*\bra{\mu} = \bra{\pi_1}_*\bra{\nu}.
\]
That is, the laws of the two rough path measures at level 1 must be equal for $\Ww^\star_q(\mu,\nu)$ to be well-defined. 
To our knowledge, \cite[\S3]{cass2013evolving} is the first paper to explicitly use the Wasserstein metric on the space of geometric $p$-rough paths by using the $p$-variation metric as a cost function. 

\begin{Remark}\label{r-complexity}
Let us return to considering the SDE defining L\'{e}vy area, (remaining in the case of $d=2$ for simplicity):
\begin{equation}\label{e-sde-levy-area}
dx_1 = dW_1, \gap dx_2 = dW_2, \gap dx_3 = \frac{1}{2}\bra{x_1\, dW_2 - x_2\, dW_1}.
\end{equation}
In common with the Milstein scheme, a simple calculation shows that the original log-ODE method is exact $x^{(j)}=A_{12}(0,jh) = \sum_{k=0}^{j-1} A^{(k)}$. 
Similarly, our new approximation scheme $\{x^h_j\}_{j=1}^N$ can be shown to satisfy $x^h_j = \sum_{k=0}^{j-1} B^{(k)}$. 
Thus Proposition \ref{p-davie-counter} implies that among all the couplings of $A^{(j)}$ and $B^{(j)}$ such that the underlying Brownian increments are equal (that is, $W^{(j)}=\tilde{W}^{(j)}$), the optimal rate of convergence of our approximation scheme in the Wasserstein metric (as defined in Theorem \ref{t-big}), must be at least worse than $O(-h\log h)$. Theorem \ref{t-big} gives a rate of convergence of $O(h^{1-2/\gamma-\varepsilon})$, where $\varepsilon>0$ is arbitrary and $\gamma$ is the degree of the Stein-Lipschitz norm of the vector fields of the original SDE. In the present case of (\ref{e-sde-levy-area}), the vector fields are polynomial and thus $\gamma$ can be chosen to be arbitrarily large. Therefore the Wasserstein rate is arbitrarily close to the optimal rate of convergence (up to a logarithmic factor). We make the disclaimer that while this argument is true from a theoretical point of view, increasing $\gamma$ will cause a corresponding exponential increase of the constant in the RDE Lipschitz-estimate needed to prove Theorem \ref{t-big}. 
\end{Remark}

Next, we focus on establishing H\"{o}lder bounds for our original coupling of $(\WW^h,\XX^h$) as provided by Proposition \ref{p-kmt-coupling}. 
First it can be shown that the commmon first level satisfies
\[
\abs{\norm{\pi_1(\WW^h_{s,t})}}_{L^q} = \abs{\norm{\pi_1(\XX^h_{s,t})}}_{L^q} = C(q)\abs{t-s}^{1/2}.
\]
Similarly, $\abs{\norm{\pi_2\bra{\WW^h_{s,t}}}}_{L^q},\, \abs{\norm{\pi_2\bra{\XX^h_{s,t}}}}_{L^q} \leq C(q)\abs{t-s}$ by scaling.
Thus a standard Kolmogorov regularity result (\cite[Theorem A.12]{FV}) implies that for all $\alpha\in [0,\frac{1}{2})$,
\[
\abs{\rho_{\alpha\textup{-H\"{o}l;}[0,1]}(\WW^h)}_{L^q},\, \abs{\rho_{\alpha\textup{-H\"{o}l;}[0,1]}(\XX^h)}_{L^q} \leq C=C(q).
\]
We now consider the H\"{o}lder norm between $\WW^h$ and $\XX^h$.
Using Proposition \ref{p-kmt-coupling} with the dyadic singleton set $E=\seq{jh}$ gives
\begin{equation}\label{e-holder-one}
\abs{\norm{\pi_2\bra{\WW^h_{s,t}-\XX^h_{s,t}}}}_{L^{5/2}} \leq C\abs{t-s}h^{-1}h = C\abs{t-s}, \gap \forall [s,t] \subseteq [jh,(j+1)h].
\end{equation}
Similarly, for larger increments $\abs{t-s}\geq h$ our coupling ensures that
\begin{equation}\label{e-holder-two}
\abs{\norm{\pi_2\bra{\WW^h_{s,t}-\XX^h_{s,t}}}}_{L^{5/2}}  
\leq \max_{\substack{E \subseteq \seq{0,1\ldots,2^m-1} : \\
2^m \cdot [s,t] \cap \seq{0,1,\ldots, 2^m-1} \subseteq E}} \abs{\norm{{\gamma}_E - \lambda_E}}_{L^{5/2}}
\leq Ch\log( h^{-1}).
\end{equation}
Thus, 
\[
\abs{\norm{\pi_2\bra{\WW^h_{s,t}-\XX^h_{s,t}}}}_{L^{5/2}} \leq C\bra{h\log(h^{-1}) \wedge \abs{t-s}}.
\]
Certainly the stochastic processes $\pi_1(\WW^h)=\pi_1(\XX^h)$ and $\pi_2(\WW^h)$ take their values in the 1st and 2nd inhomogeneous Wiener chaos $\Cc^{1}(\PP)$, $\Cc^2(\PP)$ respectively, (\cite[Proposition 15.20]{FV}). Recall the decompositon (\ref{e-davie-normal-area}):
\[
B^{(j)}_{kl} = z^{(j)}_k W^{(j)}_l - z^{(j)}_l W^{(j)}_k + \lambda^{(j)}_{kl}.
\]
That is, each $B^{(j)}$ is a quadratic polynomial of Gaussian random variables, (albeit with a complicated covariance structure with respect to the $\{A^{(j)}\}_{j=0}^{N-1}$ increments). Moreover, 
\[
\pi_2\bra{\XX^h_{jh,(j+1)h}} = \pi_2\bra{\WW^h_{jh,(j+1)h}} - A^{(j)} +B^{(j)}.
\]
Therefore $\pi_2(\XX^h)$ also takes values in the 2nd inhomogeneous Wiener chaos $\Cc^{2}(\PP)$. 
Combining the equivalence of $L^q$-norms on inhomogeneous Wiener chaos space (specifically \cite[Proposition 15.25]{FV}) with Proposition \ref{p-kmt-coupling} yields the following difference estimate.

\begin{Proposition}\label{p-coupling-lp}
There exists a constant $C>0$ such that for all $q\geq 1$, we have
\[
\abs{\norm{\pi_2\bra{\WW^h_{s,t}-\XX^h_{s,t}}}_{(\RR^d)^{\otimes 2}}}_{L^{q}} \leq Cq\bra{h\log(h^{-1}) \wedge \abs{t-s}}.
\]
\end{Proposition}
For further details on the integrability of Wiener chaos expansions we refer to  \cite[Proposition 3]{friz2014convergence}, \cite[Exercise 13.6, Theorem D.8]{FV}, and \cite{riedel2013simple}.

\begin{Lemma}\label{l-holder-bound}
For all $q\geq 1$, there exists a constant $C=C(q)$ such that for all $\alpha \in [0,\frac{\theta}{2})$ we have 
\[
\abs{\rho_{\alpha\textup{-H\"{o}l;}[0,1]}\bra{\WW^h,\XX^h}}_{L^q} 
\leq C\bra{h\log(h^{-1})}^{1-\theta}. 
\]
\end{Lemma}
\begin{proof}
Since $\WW^h$ and $\XX^h$ share a common first level, 
\[
\rho_{\alpha\textup{-H\"{o}l;}[0,1]}\bra{\WW^h,\XX^h}
=
\rho^{(2)}_{\alpha\textup{-H\"{o}l;}[0,1]}\bra{\WW^h,\XX^h}
=
\sup_{0\leq s < t \leq 1} \frac{\norm{\pi_2(\WW^h_{s,t}-\XX^h_{s,t})}_{(\RR^d)^{\otimes 2}}}{\abs{t-s}^{2\alpha}}
\]
Combining (\ref{e-holder-one}) and (\ref{e-holder-two}) guarantees that for every $\theta\in [0,1]$,
\begin{align*}
\abs{\norm{\pi_2\bra{\WW^h_{s,t} - \XX^h_{s,t}}}}_{L^q} 
\leq C \bra{ h \log(h^{-1}) \wedge \abs{t-s}}
\leq C\bra{h\log(h^{-1})}^{1-\theta} \abs{t-s}^\theta.
\end{align*}
Appealing to a standard Kolomogorov result for rough paths (\cite[Theorem A.13]{FV}), we arrive at the claim.
\end{proof}

For our approximation $\XX^h$ of $\WW^h$ to be of any use in RDE Lipschitz estimates, we need the previous quantity $(h\log(h^{-1}))^{1-\theta}$ to be less than $O(\sqrt{h})$, or else we might as well have used the level-$1$ log-ODE method which has order $O(\sqrt{h})$, and in common with $\XX^h$, neither needs L\'{e}vy area increments. 
However this requires that $\theta<\frac{1}{2}$, which in turn demands that $\alpha \in [0,\frac{1}{4})$. 
Unfortunately for $\alpha < \frac{1}{4}$, $\rho_{\alpha\textup{-H\"{o}l}}(\cdot)$ is no longer a rough path metric for $2$-rough paths.   
Therefore to make the  H\"{o}lder bound of Lemma \ref{l-holder-bound} useful at all we would need to not only compute \textit{but control} the $p$-variation, (where $p=\alpha^{-1}>4$), of at least the first 4 levels of the lifts $S_\kappa(\WW^h)$ and $S_\kappa(\XX^h)$, ($\kappa\geq 4$), and the corresponding difference at each tensor level: 
\[
\max_{k=2,\ldots,\kappa} \abs{\norm{\pi_k\bra{S_\kappa(\WW^h)_{s,t}-S_\kappa(\XX^h)_{s,t}}}_{(\RR^d)^{\otimes k}}}_{L^q}.
\]
This is precisely what we will establish in the next section. 
Importantly, we know from Remark \ref{r-eq} that lifting the piecewise abelian rough paths $\WW^h$, $\XX^h$ does not change the RDE defining the log-ODE method and our approximation scheme. 

\begin{Remark}
This technique of lifting the rough paths $\WW^h$ and $\XX^h$ to higher levels $\kappa>2$ in order to get a better $p$-variation distance bound is inspired by the same technique for dealing with the convergence rates of Gaussian rough paths by Friz, Riedel and Xu in \cite{friz2014convergence,riedel2013simple} (see Remark 5.2 in the latter and the first remark after Corollary 1 of the former paper). The basic idea is that by lifting the driving rough path to a higher level, $p$ is allowed to be larger, and hence a better estimate may be hoped for. In order for the RDE to still be well-defined and have a unique solution, we are penalised by requiring that the Stein-Lipschitz order of the vector fields is of order $\gamma>p$. In their particular application, the previous authors recover (almost) optimal rates of convergence for Wong-Zakai approximations of RDEs driven by various Gaussian rough path lifts, including fractional Brownian motion with Hurst index $H \in (\frac{1}{3},1]$, (see the final remark after Corollary 1 of \cite{friz2014convergence}).
\end{Remark}


\section{Lifted rough path estimates}\label{s-lifted-estimates}

The aim of this section is to prove that the lifting trick discussed in the previous section will actually work and give the inhomogeneous $p$-variation estimates needed to establish Corollary \ref{c-main-2} in the next section. This is the content of the final result of this section, Theorem \ref{t-main-bound}. 

As our first step in proving Theorem \ref{t-main-bound}, we consider an error estimate for our lifted rough path coupling. 

\begin{Proposition}\label{p-main-estimate-prop}
Fix an integer $\kappa>2$. Then for all $q\geq 1$ there exists a constant $C=C(\kappa,q)$ such that for each integer $ 2 \leq m \leq \kappa$:
\[
\abs{\norm{\pi_m\bra{S_\kappa(\WW^h)_{s,t}-S_\kappa(\XX^h)_{s,t}}}_{(\RR^d)^{\otimes m}}}_{L^q} \leq C \abs{t-s}^{\frac{m}{2}-1}\bra{h\log(h^{-1}) \wedge \abs{t-s}} \gap \forall\, s<t \in [0,1].
\]
\end{Proposition}

Before giving the proof we present some useful technical results, beginning with the following lemma (cf. \S3 of \cite{lyons2007extension}).

\begin{Lemma}\label{l-new-log}
Fix an integer $\kappa>2$, $p\in (2,3)$ and suppose $\XX\in C^{1/p\textup{-H\"{o}l}}\bra{[0,1],G^{(2)}(\RR^d)}$. Then there exists a constant $C=C(\kappa,p)$ such that for all $k=1,\ldots,\kappa$, 
\[
\norm{\pi_k\bra{\log S_\kappa\bra{\XX}_{s,t}}}_{(\RR^d)^{\otimes k}} \leq C\abs{t-s}^{k/p}\norm{\XX}_{1/p\textup{-H\"{o}l;}[s,t]}^k.
\]
\end{Lemma}
\begin{proof}
The function 
\[
g \in G^{(\kappa)}(\RR^d) \mapsto \max_{m=1,\ldots,\kappa} \norm{\pi_m\bra{\log g}}_{(\RR^d)^{\otimes m}}^{1/m}
\]
is a homogeneous norm on $G^{(\kappa)}(\RR^d)$ (cf. Exercise 7.38 of \cite{FV}). 
Therefore by the equivalence of such norms on $G^{(\kappa)}(\RR^d)$, (\cite[Theorem 7.44]{FV}), there exists a constant $C=C(\kappa)\geq 1$ such that
\[
\norm{\pi_k\bra{\log S_\kappa(\XX)_{s,t}}} \leq \bra{\max_{m=1,\ldots,\kappa} \norm{\pi_m\bra{\log S_\kappa(\XX)_{s,t}}}^{1/m}}^k \leq C\norm{S_\kappa(\XX)_{s,t}}_{{C}}^k \leq C\norm{S_\kappa(\XX)}_{p\textup{-var;}[s,t]}^k.
\]
The Lipschitz-continuity of the rough path lift in $p$-variation guarantees the existence of a constant $C=C(\kappa,p)$ such that $\norm{S_\kappa(\XX)}_{p\textup{-var;}[s,t]} \leq C\norm{\XX}_{p\textup{-var;}[s,t]}$. Moreover, a simple consequence of the super-additivity of controls (cf. \cite[\S8.1]{FV}) gives $\norm{\XX}_{p\textup{-var;}[s,t]} \leq \abs{t-s}^{1/p} \norm{\XX}_{1/p\textup{-H\"{o}l;}[s,t]}$. Putting the last three inequalities together completes the proof. 
\end{proof}

We apply Lemma \ref{l-useful}  with 
\begin{equation}\label{e-x-notation}
x_{j+1}=\xi^{(j)}=W^{(j)}+A^{(j)}, \gap y_{j+1}=\eta^{(j)}=W^{(j)}+B^{(j)}, \textup{ for } j=0,1,\ldots,N-1.
\end{equation} 
Then exploiting the compatibility of the tensor algebra norm, we have for all $m\leq \kappa$:
\begin{align*}
&\norm{\pi_m\bra{S_\kappa(\WW^h)_{0,nh} - S_\kappa(\XX^h)_{0,nh}}}_{(\RR^d)^{\otimes m}}\\ 
&= \norm{\pi_m\bra{e^{x_1} \otimes \ldots \otimes e^{x_{n}} - e^{y_1} \otimes \ldots \otimes e^{y_n}}}_{(\RR^d)^{\otimes m}}\\
&\leq \sum^m_{k=1} \frac{1}{k!}\sum_{\substack{i_1,\ldots,i_k > 0\\i_1+\ldots+i_k=m}} \norm{ \sum^k_{j=1} g_{i_1} \otimes \ldots \otimes  g_{i_{j-1}} \otimes (g_{i_{j}} - h_{i_j}) \otimes h_{i_{j+1}} \otimes \ldots \otimes h_{i_k} }_{(\RR^d)^{\otimes m}}\\
&\leq \sum^m_{k=1} \frac{1}{k!}\sum_{\substack{i_1,\ldots,i_k > 0\\i_1+\ldots+i_k=m}} \sum^k_{j=1} \norm{  g_{i_1} \otimes \ldots \otimes  g_{i_{j-1}}} \norm{ g_{i_{j}} - h_{i_j}} \norm{ h_{i_{j+1}} \otimes \ldots \otimes h_{i_k} }
\end{align*}
where 
\begin{align*}
g_{i} :&= \pi_i\bra{ H(x_1,\ldots,x_n)} = \pi_i\bra{\log S_\kappa(\WW^h)_{0,nh}},\notag\\
h_{i} :&= \pi_i\bra{ H(y_1,\ldots,y_n)} = \pi_i\bra{\log S_\kappa(\XX^h)_{0,nh}} \in \mathfrak{g}^{(k)}(\RR^d).\label{e-log-sig-def}
\end{align*}
Since the increments of $\WW^h$ and $\XX^h$ agree, $g_1-h_1=0$ and 
\[
g_2 - h_2 = \sum^{n-1}_{j=0} (A^{(j)}-B^{(j)}) \in [\RR^d,\RR^d].
\]
Moreover, by Lemma \ref{l-new-log}, given $p\in (2,3)$, there exists some constant $C=C(p)$ such that 
\[
\norm{g_i}_{(\RR^d)^{\otimes i}} = \norm{\pi_i\bra{\log S_\kappa(\WW^h)_{0,nh}}}_{(\RR^d)^{\otimes i}} \leq C(nh)^{i/p}\norm{\WW^h}^i_{1/p\textup{-H\"{o}l;}[0,nh]}
\]
with a similar bound for $h_i$ with $\XX^h$.
Therefore the previous inequality becomes
\begin{align}
&\norm{\pi_m\bra{S_\kappa(\WW^h)_{0,nh} - S_\kappa(\XX^h)_{0,nh}}}_{(\RR^d)^{\otimes m}}\label{e-prev-inequality-new}\\ 
&\leq C \sum^m_{k=1} \frac{1}{k!}\sum_{\substack{i_1,\ldots,i_k > 0\\i_1+\ldots+i_k=m}}\sum^k_{j=1} (nh)^{\frac{1}{p}(m-i_j)} \norm{\WW^h}_{1/p\textup{-H\"{o}l;}[0,nh]}^{\frac{1}{p}(i_1+\ldots +i_{j-1})} \norm{\XX^h}_{1/p\textup{-H\"{o}l;}[0,nh]}^{\frac{1}{p}(i_{j+1}+\ldots+i_k)} \norm{g_{i_j}-h_{i_j}}_{(\RR^d)^{\otimes i_j}}\notag\\
&\leq C \sum^m_{k=1} \frac{1}{k!}\sum_{\substack{i_1,\ldots,i_k > 0\\i_1+\ldots+i_k=m}}\sum^k_{j=1} (nh)^{\frac{1}{p}(m-i_j)} \bra{\norm{\WW^h}_{1/p\textup{-H\"{o}l;}[0,nh]}\vee\norm{\XX^h}_{1/p\textup{-H\"{o}l;}[0,nh]}}^{m-i_j} \norm{g_{i_j}-h_{i_j}}.\notag
\end{align}
So it remains to bound $\norm{g_{i}-h_{i}}_{(\RR^d)^{\otimes i}}$ in $L^q$ for each integer $3\leq i\leq \kappa$. This is the content of Lemma \ref{l-main-moment-estimates} below, which takes particular care in separating out the dependencies on the variables $m$ and $n$.  For the specific calculations of $g_m-h_m$ for $m\leq 5$ we refer to the iterated Baker-Campbell-Hausdorff formula (\ref{e-BCHf-d}) found in the appendix at the end of the paper. But first we need the following consequence of symmetry.

\begin{Lemma}\label{l-symmetry}
Unless the indices $(j_1,\ldots,j_{2n}) \in \{0,1,\ldots,N-1\}^{2n}$ are present in pairs, we have
\[
\EE\bra{W^{(j_1)} \otimes \ldots \otimes W^{(j_{2n})} \big| \seq{A^{(j)}, B^{(j)}}_{j=0}^{N-1}} =0.
\]
\end{Lemma}
\begin{proof}
Suppose $j_i$ is without a pair in $\seq{j_1,\ldots,j_{i-1},j_{i+1},\ldots,j_{2n}}=:I$, (that is, $j_i \notin I$). Then setting $\sigma(A,B):=\seq{A^{(j)}, B^{(j)} : 0\leq j \leq N-1}$, the tower property of conditional expectation gives
\begin{align*}
&\EE\bra{W^{(j_1)} \otimes \ldots \otimes W^{(j_i)} \otimes \ldots \otimes W^{(j_{2n})} \big| \sigma(A,B)}\\
&= \EE\bra{ W^{(j_1)} \otimes \ldots W^{(j_{i-1})} \otimes \EE\bra{ W^{(j_i)} \big| \sigma(A,B), \seq{W^{(j_k)}}_{k\neq i}}  \otimes W^{(j_{i+1})} \otimes \ldots \otimes W^{(j_{2n})} \big| \sigma(A,B)}.
\end{align*}
There exists a matrix $Z^{(j_i)}\in \RR^{d(d-1)\times d}$ with entries taking values in $\{\pm z^{(j_i)}_k, \pm\zeta^{(j_i)}_k\}_{k=1}^d$ such that 
\[
\left(\begin{array}{c}
A^{(j_{i})}\\
B^{(j_{i})}
\end{array}\right) =: Z^{(j_i)}W^{(j_i)} + \left(\begin{array}{c}
K^{(j_{i})}\\
\lambda^{(j_{i})}
\end{array}\right),
\]
where $A^{(j_i)}=\bra{A^{(j_i)}_{12},\ldots,A^{(j_i)}_{d-1,d}}^t \in \RR^{\frac{d}{2}(d-1)}$ and $B^{(j_i)}, K^{(j_i)}, \lambda^{(j_i)}$ defined similarly. 
By symmetry we can change the sign of all entries on the right-hand side without changing the law of $(A^{(j_i)}, B^{(j_i)})^t$; that is, 
\begin{equation}\label{e-law-eq}
\left(\begin{array}{c}
A^{(j_{i})}\\
B^{(j_{i})}
\end{array}\right) \overset{\mathcal{L}}{=} \bra{-Z^{(j_i)}}\bra{-W^{(j_i)}} + 
 \left(\begin{array}{c}
-K^{(j_{i})}\\
-\lambda^{(j_{i})}
\end{array}\right)
\end{equation}
Recalling that $\{z^{(j)}, \zeta^{(j)}, K^{(j)},\lambda^{(j)}\}_{j=0}^{N-1}$ are independent of all Brownian increments $\{W^{(j)}\}_{j=0}^{N-1}$, it follows from (\ref{e-law-eq}) that:
\[
\PP\bra{W^{(j_i)}_c \geq x \big| \sigma(A,B), \seq{W^{(j_k)}}_{k\neq i}} = \PP\bra{W^{(j_i)}_c \leq -x \big| \sigma(A,B), \seq{W^{(j_k)}}_{k\neq i}} \textup{ for all } x\geq 0,
\]
for any given coordinate $c\in\seq{1,\ldots,d}$. This symmetry of the conditional density yields
\begin{align*}
\l< e_c^* , \EE\bra{ W^{(j_i)} \big| \sigma(A,B), \seq{W^{(j_k)}}_{k\neq i}}\r>
&= \EE\bra{ \l<e_c^*, W^{(j_i)} \r> \big| \sigma(A,B), \seq{W^{(j_k)}}_{k\neq i}}\\
&= \EE\bra{W^{(j_i)}_c \big| \sigma(A,B), \seq{W^{(j_k)}}_{k\neq i}} 
=0,
\end{align*}
and so 
\begin{align*}
&\EE\bra{W^{(j_1)} \otimes W^{(j_i)} \otimes \ldots \otimes W^{(j_{2n})} \big| \sigma(A,B)}\\
&\gap\gap\gap\gap\gap\gap=  \EE\bra{ W^{(j_1)} \otimes \ldots W^{(j_{i-1})} \otimes 0  \otimes W^{(j_{i+1})} \otimes \ldots \otimes W^{(j_{2n})} \big| \sigma(A,B)} = 0.
\end{align*}
The proof is complete. 
\end{proof}

For convenience, let us define the associative operation $*$ on $\mathfrak{g}^{(m)}(\RR^d)$ by
\begin{equation}\label{e-star-notation}
x_1 * x_2 * \cdots *  x_{k-1} * x_k := [x_1,[x_2,[\ldots,[x_{k-1},x_k]]]].
\end{equation}

\begin{Lemma}\label{l-unfold-norm}
There exists a universal constant $C=C(m)$ such that for every $x_1,\ldots,x_k \in \mathfrak{g}^{(m)}(\RR^d)$, we have 
\[
\norm{\pi_m\bra{x_1*\cdots*x_k}}_{(\RR^d)^{\otimes m}} 
\leq C \sum_{1\leq j_1,\ldots,j_m\leq d}\abs{ \l< e_{j_1}^* \otimes \ldots \otimes e_{j_m}^*, \pi_m\bra{x_1 \otimes \ldots \otimes x_k}\r>}.
\]
\end{Lemma}
\begin{proof}
If $k>m$ then the inequality is trivial: both sides are zero. So we suppose that $k\leq m$. 
Unfolding the nested Lie brackets of $x_1*\cdots*x_k$, there will be at most $2^k (\leq 2^m)$ non-zero terms in the resultant expansion, with each term taking the form
\[
\pm x_{\sigma(1)} \otimes x_{\sigma(2)} \otimes \ldots \otimes x_{\sigma(k)}, 
\]
for some $\sigma \in S_k$, where $S_k$ is the symmetric group of permutations on $\seq{1,\ldots,k}$ (see Lemma \ref{l-lie-tech}). 
In other words, every $x_i$ will be present exactly once in each term in the expansion of $x_1*\cdots*x_k$. By symmetry, $\norm{x_{\sigma(1)} \otimes \ldots \otimes x_{\sigma(k)}}_{(\RR^d)^{\otimes m}} = \norm{x_1\otimes \ldots \otimes x_k}_{(\RR^d)^{\otimes m}}$. 
Since $k< m$, it follows that there exists a constant $C=C(m)$ such that 
\begin{align*}
\norm{\pi_m\bra{x_1*\cdots*x_k}}_{(\RR^d)^{\otimes m}} 
&\leq C\norm{\pi_m\bra{x_1\otimes\ldots\otimes x_k}}_{(\RR^d)^{\otimes m}}\\
&= C \sqrt{\sum_{1\leq j_1,\ldots,j_m\leq d} \abs{\l<e_{j_1}^* \otimes \ldots \otimes e_{j_m}^*,\pi_m\bra{x_1\otimes \ldots \otimes x_k}\r>}^2}.
\end{align*}
Using the equivalence of norms on finite-dimensional vectors spaces (in particular, $\norm{a}_{l^2} \leq \norm{a}_{l^1}$ for $a\in \RR^p$) yields
\begin{equation*}
\norm{\pi_m\bra{x_1\otimes \ldots \otimes x_k}}_{(\RR^d)^{\otimes m}} \leq \sum_{1\leq j_1,\ldots,j_m\leq d}\abs{ \l< e_{j_1}^* \otimes \ldots \otimes e_{j_m}^*, \pi_m\bra{x_1 \otimes \ldots \otimes x_k}\r>}. 
\end{equation*}
The claim follows. 
\end{proof}

\begin{Lemma}\label{l-main-moment-estimates}
Fix $q\geq 1$, an integer $n\geq 1$, and for $m=2,\ldots,\kappa$ define:
\begin{align*}
g_{m} &=   \pi_m\bra{\log S_\kappa(\WW^h)_{0,nh}}, \gap
h_{m}=  \pi_m\bra{\log S_\kappa(\XX^h)_{0,nh}} \in \mathfrak{g}^{(m)}(\RR^d).
\end{align*}
Then for each $m$, there exists a constant $C_m=C_m(\kappa,q)$ such that
\begin{equation}\label{e-lemma-aim}
\abs{\norm{g_m-h_m}_{(\RR^d)^{\otimes m}}}_{L^q} \leq C_m  (nh)^{\frac{m}{2}-1}h\log(h^{-1}). 
\end{equation}
\end{Lemma}
\begin{proof}
Without loss of generality we may assume $q=2r$ for some integer $r\geq 1$. 
The case of $m=2$ is precisely the content of our coupling result in Proposition \ref{p-kmt-coupling} so we restrict the proof to  $3\leq m \leq \kappa$.
We define $x_i, y_i$ as in (\ref{e-x-notation}).
Let $\tau_k: \mathfrak{g}^{(m)}(\RR^d) \to \mathfrak{g}^{(m)}(\RR^d)$ denote the canonical projection onto nested Lie brackets terms of length $k$. Note that $\tau_k\bra{g_m-h_m} = 0$ for all $k\geq m$. Indeed, for $k>m$ this is a consequence of truncation, while in the case of $k=m$, the nested Lie brackets will only consist of $W^{(j)}$ increment terms, which are common to both expressions.  
Using Lemma \ref{l-lie-tech} for rearranging nested Lie brackets into brackets of the form of (\ref{e-star-notation}) if needed, the Baker-Campbell-Hausdorff formula gives index sets $\lambda \subset \seq{1,\ldots,N-1}^k$ and corresponding coefficients $c_\lambda$ such that  
\[
\tau_k\bra{g_m} = \sum_{\lambda \in \Lambda_k} c_\lambda \sum_{(i_1,\ldots,i_k)\in\lambda} \pi_m\bra{x_{i_1}*\cdots*x_{i_k}}.
\]
Setting $z_i := x_i - y_i=A^{(i-1)}-B^{(i-1)}$, we can exploit the property that a nested bracket of length $k$ is a $k$-multilinear map and employ a telescoping sum to find that 
\begin{align*}
&\tau_k(g_m-h_m)\\
&= \sum_{\lambda \in \Lambda_k} c_\lambda {\sum_{(i_1,\ldots,i_k) \in \lambda} \pi_m\bra{x_{i_1}*\cdots*x_{i_k} - y_{i_1}*\cdots*y_{i_k}}}\\
&= \sum_{\lambda \in \Lambda_k} c_\lambda \sum_{(i_1,\ldots,i_k) \in \lambda} \pi_m\bra{z_{i_1}*x_{i_2}*\cdots *x_{i_k} + y_{i_1} * z_{i_2} * x_{i_3}* \cdots * x_{i_k} + \ldots + y_{i_1} *\cdots*y_{i_{k-1}}*z_{i_k} }\\
&=  \sum_{r=1}^k \sum_{\lambda \in \Lambda_k} c_\lambda \sum_{(i_1,\cdots,i_k) \in \lambda}  \pi_m\bra{y_{i_1} * \cdots * y_{i_{r-1}} * z_{i_r} * x_{i_{r+1}} * \cdots * x_{i_k}}.
\end{align*}
Note that each nested Lie bracket term in the difference $\tau_k(g_m-h_m)$ will contain $N_1$ increment terms (that is, $W^{(i)}$) and $N_2$ area terms (either $A^{(i)}, B^{(i)}$) such that $N_1+N_2 = k$ and $N_1+ 2N_2 = m$. 
By changing the order of summation if necessary, for each $(r,\lambda)$ pair there exist index sets $\hat{\lambda}^r$ and $\lambda_I^r$, (the latter consisting of necessarily consecutive elements), such that
\begin{align}
&\sum_{(i_1,\ldots,i_k) \in \lambda}  \pi_m\bra{y_{i_1} * \cdots * y_{i_{r-1}} * z_{i_r} * x_{i_{r+1}} * \cdots * x_{i_k}}\notag\\
&= \sum_{{I}=(i_1,\ldots,i_{r-1},i_{r+1},\ldots,i_k)\in\hat{\lambda}^r}  \sum_{j\in\lambda^r_I}  \pi_m\bra{y_{i_1} * \cdots * y_{i_{r-1}} * z_{i_r} * x_{i_{r+1}} * \cdots * x_{i_k}}\notag\\
&=\sum_{{I}=(i_1,\ldots,i_{r-1},i_{r+1},\ldots,i_k)\in\hat{\lambda}^r} 
\pi_m\bra{y_{i_1}*\cdots*y_{i_{r-1}}*\bra{\sum_{j\in\lambda^r_I} z_j} * x_{i_{r+1}} * \cdots * x_{i_k}}.\label{e-to-be-projected}
\end{align}
We can always choose $\lambda_I^r$ to be made up of consecutive elements because the induction proof of the iterated Baker-Campbell-Hausdorff formula (\ref{e-BCHf-d}) and the bilinearity of nested Lie brackets guarantee that for any fixed indices $i_1,\ldots,i_{r-1},i_{r+1},\ldots,i_k \in \seq{1,\ldots,n}$, the set
\[
\seq{j : (i_1,\ldots,i_{r-1},j,i_{r+1},\ldots,i_k) \in \lambda}
\]
is either empty or made up of consecutive elements.

Since the indices of $\lambda^r_I$ are consecutive,
\[
\sum_{j\in \lambda_I^r} z_j = \sum_{j\in\lambda_I^r} \bra{A^{(j-1)}-B^{(j-1)}} 
= \pi_2\bra{\WW^h_{s,t} - \XX^h_{s,t}},
\]
for some pair $(s,t) = (n_1h,n_2h)$ where $n_1, n_2\in \{0,1,\ldots,N\}$. Applying Proposition \ref{p-coupling-lp} to the right-hand side yields the estimate
\begin{equation}\label{e-coupling-nice}
\abs{\norm{{\sum_{j\in\lambda^r_I} z_j}}_{(\RR^d)^{\otimes 2}}}_{L^q} 
\leq C(q)h\log(h^{-1}),
\end{equation}
Alternatively, we could have argued that for every set of fixed indices $i_1,\ldots,i_{r-1},i_{r+1},\ldots,i_k \in \seq{1,\ldots,n}$, the set
\[
\seq{j : (i_1,\ldots,i_{r-1},j,i_{r+1},\ldots,i_k) \in \lambda}
\]
can be decomposed into at most $k$ unique sets, each composed of consecutive elements themselves. Indeed, the other $k-1$ indices impose at most $k$ restrictions on the free index $j$ (see the iterated formula (\ref{e-BCHf-d})). Since $k\leq m$, it follows that the constant $C(q)$ in (\ref{e-coupling-nice}) would have to be replaced by $C(q,m)$, (that is a constant dependent on both $q$ and $m$). 

Set $Z$ to be the projection of (\ref{e-to-be-projected}) onto an arbitrary coordinate $(c_1,\ldots,c_m)$ of $(\RR^d)^{\otimes m}$:
\begin{align*}
&\l< e_{c_1}^* \otimes \ldots \otimes e_{c_m}^*, \sum_{{I}=(i_1,\ldots,i_{r-1},i_{r+1},\ldots,i_k)\in\hat{\lambda}^r}\pi_m\bra{y_{i_1} \otimes \ldots \otimes y_{i_{r-1}} \otimes \bra{\sum_{j\in\lambda^r_I} z_j} \otimes x_{i_{r+1}} \otimes \ldots \otimes x_{i_k}}\r>
\end{align*}
and for $I=(i_1,\ldots,i_{r-1},i_{r+1},\ldots,i_k)\in \hat{\lambda}^r$ define
\[
\varphi_I :=\l< e^*_{c_1} \otimes \ldots \otimes e^*_{c_m}, \pi_m\bra{y_{i_1} \otimes \ldots \otimes y_{i_{r-1}} \otimes \bra{\sum_{j\in\lambda^r_I} z_j} \otimes x_{i_{r+1}} \otimes \ldots \otimes x_{i_k}}\r>.
\]
By the equivalence of norms given by Lemma \ref{l-unfold-norm}, it suffices to prove that
\[
\abs{Z}_{L^q} \leq C(nh)^{\frac{m}{2}-1}(h\log(h^{-1})).
\] 
in order to establish (\ref{e-lemma-aim}). 
To this end, note that since $q=2r$ for some integer $r$, we have
\[
\abs{Z}_{L^q}^q = \abs{Z}_{L^{2r}}^{2r}
= \EE\bra{ \bra{\sum_{I\in \hat{\lambda}^r} \varphi_I}^{2r}}
= \sum_{I_1 \in \hat{\lambda}^r} \ldots \sum_{I_{2r} \in \hat{\lambda}^r} \EE\bra{\varphi_{I_1} \ldots \varphi_{I_{2r}}}.
\]
The $qN_1=2rN_1$ increment terms in the expression for $\abs{Z}^{q}_{L^q}$ must be paired by Lemma \ref{l-symmetry} and the tower property of conditional expectation. There are $(qN_1-1)!! \leq C(m,q)$ possible perfect matchings, (where $N!! := \frac{(2k)!}{2^k k!}$ for $N=2k+1$). Each matching is summed over at most $n$ possible indices and has scaling $O(h)$ in $L^p$ (for all $p\geq 1$). The remaining $(qN_2-2)$ area sums are taken over at most $n$ indices each, with each constituent area term having scaling $O(h)$ in $L^p$.
By combining these observations with (\ref{e-coupling-nice}) and the Cauchy-Schwarz inequality, we conclude that
\begin{align*}
\abs{Z}^{2r}_{L^{2r}} \leq C(nh)^{rN_1} (nh)^{2r(N_2-1)}(h\log(h^{-1}))^{2r} &= C(nh)^{r(N_1+2N_2-2)}(h\log(h^{-1}))^{2r}\\
&= C(nh)^{r(m-2)}(h\log(h^{-1}))^{2r}
\end{align*}
and the result follows. 
\end{proof}

\begin{proof}[Proof of Proposition \ref{p-main-estimate-prop}]
For $\abs{t-s} \leq h$ the result is immediate from scaling so we restrict ourselves to $\abs{t-s}\geq h$. 
Since $\WW^h$ and $\XX^h$ are both piecewise abelian over each interval $[jh,(j+1)h]$ it suffices to prove the claim for $(s,t)=(0,nh)$. 

To this end we combine (\ref{e-prev-inequality-new}) with the estimates of Lemma \ref{l-main-moment-estimates} using the Cauchy-Schwarz inequality and find
\begin{align*}
&\abs{\norm{\pi_m\bra{S_\kappa(\WW^h)_{0,nh} - S_\kappa(\XX^h)_{0,nh}}}_{(\RR^d)^{\otimes m}}}_{L^q}\\
&\leq C_1 \sum^m_{k=1} \frac{1}{k!}\sum_{\substack{i_1,\ldots,i_k > 0\\i_1+\ldots+i_k=m}}\sum^k_{j=1} (nh)^{\frac{1}{p}(m-i_j)} \bra{ \abs{\norm{\WW^h}_{1/p\textup{-H\"{o}l;}[0,nh]}^{m-i_j}}_{L^{2q}} + \abs{\norm{\XX^h}_{1/p\textup{-H\"{o}l;}[0,nh]}^{m-i_j}}_{L^{2q}}} \abs{\norm{g_{i_j} - h_{i_j}}}_{L^{2q}}\\
&\leq C_2 \sum^m_{k=1} \frac{1}{k!}\sum_{\substack{i_1,\ldots,i_k > 0\\i_1+\ldots+i_k=m}}\sum^k_{j=1} (nh)^{\frac{1}{p}(m-i_j)+\frac{m}{2}-1}h \log( h^{-1}).  
\end{align*}
The last line follows uses the integrability of the random quantities 
\[
\exp\bra{\alpha\norm{\WW^h}^2_{1/p\textup{-H\"{o}l;}[0,1]}}, \gap \exp\bra{\alpha\norm{\XX^h}^2_{1/p\textup{-H\"{o}l;}[0,1]}}
\]
for some $\alpha>0$ in any given $L^q$-norm (cf. Corollary 13.14 of \cite{FV}). 
Since $nh\leq Nh=1$, we have
\begin{align*}
\norm{\pi_m\bra{S_\kappa(\WW^h)_{0,nh} - S_\kappa(\XX^h)_{0,nh}}}_{(\RR^d)^{\otimes m}}
&\leq C_2 \sum^m_{k=1} \frac{1}{k!}\sum_{\substack{i_1,\ldots,i_k > 0\\i_1+\ldots+i_k=m}}\sum^k_{j=1} (nh)^{\frac{m}{2}-1}h \log(h^{-1})\\
&\leq C_3(nh)^{\frac{m}{2}-1}h\log(h^{-1}), 
\end{align*}
for some constant $C_3=C_3(m,\kappa)>0$. The proof is complete. 
\end{proof}

We are now in a position to prove the main result of this section, namely the claimed lifted rough path estimates. 

\begin{Theorem}\label{t-main-bound}
Fix an integer $\kappa>2$ and set $p\geq \kappa$ such that $\floor{p}=\kappa$. 
Then for all $q\geq 1$, there exists a constant $C=C(\kappa,p,q)$ such that for for all $\varepsilon>0$, 
\[
\abs{\rho_{1/p\textup{-H\"{o}l;}[0,1]}\bra{S_\kappa(\WW^h),S_\kappa(\XX^h)}}_{L^q} 
\leq C h^{1-2/p-\varepsilon}
\]
\end{Theorem}
\begin{proof}
Fix $\theta \in [0,1]$ and note that 
\[
\abs{t-s}^{\frac{m}{2}-1}\bra{h\log(h^{-1}) \wedge \abs{t-s}} \leq (h\log(h^{-1}))^{1-\theta}\abs{t-s}^{\frac{m}{2}-1+\theta} \leq (h\log(h^{-1}))^{1-\theta}\abs{t-s}^{\frac{m\theta}{2}},
\]
since $\frac{m}{2}-1+\theta\geq \frac{m\theta}{2}$ for $m\geq 2$.
Then combining the fact that $\pi_1(\WW^h_{s,t}-\XX^h_{s,t})=0$ with Proposition \ref{p-main-estimate-prop} guarantees that 
\[
\max_{m=1,\ldots,\kappa} \abs{\norm{\pi_m\bra{S_\kappa(\WW^h)_{s,t}-S_\kappa(\XX^h)_{s,t}}}_{(\RR^d)^{\otimes m}}}_{L^q} \leq C(h\log(h^{-1}))^{1-\theta} \abs{t-s}^{\frac{m\theta}{2}}.
\]
Appealing to the rough path Kolmogorov regularity theorem of \cite[Theorem A.13]{FV}, it follows that for all $\gamma<\frac{\theta}{2}$ we have 
\[
\max_{m=1,\ldots,\kappa} \abs{\sup_{0\leq s < t\leq 1} \frac{\norm{\pi_m\bra{S_\kappa(\WW^h)_{s,t}-S_\kappa(\XX^h)_{s,t}}}_{(\RR^d)^{\otimes m}}}{\abs{t-s}^{m\gamma}} }_{L^q} \leq C (h\log(h^{-1}))^{1-\theta}. 
\]
Therefore so long as $p=\gamma^{-1}>\frac{2}{\theta}$, 
\[
\abs{\rho_{1/p\textup{-H\"{o}l;}[0,1]}\bra{S_\kappa(\WW^h),S_\kappa(\XX^h)}}_{L^q} 
\leq C (h\log(h^{-1}))^{1-\theta}.
\]
The condition $p>\frac{2}{\theta}$ is satisfied by taking $\theta=\frac{2}{p}-\varepsilon$ for some small constant $\varepsilon>0$. Since this latter constant is arbitrary, we may drop the $\log(h^{-1})$ factor from our bound and the result follows. 
\end{proof}

\begin{Remark}
Although Theorem \ref{t-main-bound} requires that $p>2$, as a quick heuristic and sanity check we note that if $p<2$ then the right-hand side of its bound explodes as $N=h^{-1}\to\infty$, (as was indicated by our previous calculations in Section \ref{s-measure}). Moreover, if $p\in [2,4)$ then the convergence rate is worse than $O(\sqrt{h})$, as predicted. 
\end{Remark}


\section{Application to coupling SDE solutions}

We immediately put to work the inhomogeneous $p$-variation estimate of Theorem \ref{t-main-bound} in order to derive Wasserstein convergence rates between the solutions of the previous RDEs (with drift) driven by $\WW^h$ and $\XX^h$ (or, equivalently, their rough path lifts). Recall that the solutions coincide with the approximation scheme $\{x^h_{j}\}_{j=1}^N$ and the output of the log-ODE method. By coupling the rough paths $\WW^h, \XX^h$, (or more precisely their lifts), the following result is an almost immediate corollary of the Lipschitz-continuity of the It\^{o} map. In turn the following corollary immediately implies the main result of the paper, Theorem \ref{t-big}. 

\begin{Corollary}\label{c-main-2}
Let $x$ denote the solution to the SDE (\ref{e-stratonovich-sde}) where we assume that $V_0\in \textup{Lip}^1(\RR^q)$ and $V=\seq{V_k}_{k=1}^d \in \textup{Lip}^\gamma(\RR^q)$ for $\gamma > p > 2$. Let $y$ be the solution to the following RDE driven by $\XX^h$:
\begin{equation}\label{e-rde}
d{y}_t = V\bra{{y}_t} d\XX^h_t+V_0\bra{{y}_t}dt, \gap {y}_0 = x_0 \in \RR^q.
\end{equation}
Then there exists a constant $C=C(p,\gamma, \norm{V}_{\textup{Lip}^\gamma})$ such that for all $\varepsilon>0$, 
\[
\abs{\max_{j=1,\ldots,N} \norm{x(jh)-{y}\bra{jh}}_{\RR^q}}_{L^2} \leq C h^{1-2/p-\varepsilon}.
\]
\end{Corollary}
\begin{proof}
Let $\kappa\geq2$ be the unique integer satisfying $\floor{p}=\kappa$. 
Recalling Remark \ref{r-eq}, we can rewrite the RDE (\ref{e-rde}) as 
\[
d{y}_t = V(y_t)\, dS_\kappa(\XX^h)_t + V_0\bra{{y}_t}dt, \gap {y}_0=x_0\in \RR^q.
\]
A corresponding statement holds for the pair $\WW^h,S_\kappa(\WW^h)$; let ${z}$ be the solution to the equivalent RDEs with drift:
\[
d{z}_t = V({z}_t)\, d\WW^h_t  + V_0(z_t)\,dt = V({z}_t)\, dS_\kappa(\WW^h)_t + V_0({z}_t)\,dt, \gap {z}_0 = x_0.
\]
From Theorem \ref{t-log-ode} and Proposition \ref{p-log-ode-coincide} we already have
\[
\abs{\max_{j=1,\ldots,N}\norm{x(jh)-{z}\bra{jh}}_{\RR^q}}_{L^2} \leq Ch,
\]
and so it suffices to prove $\abs{\norm{y-z}_\infty}_{L^2} \leq Ch^{1-2/p-\varepsilon}$.

To this end, we employ the RDE Lipschitz estimate of Proposition \ref{p-lipschitz-rde} to find that for each $\alpha>0$, 
\begin{align}
\norm{y-z}_{\infty} 
&\leq \rho_{p\textup{-var;}[0,1]}(y,z)\notag\\
&\leq C_1 \rho_{p\textup{-var;}[0,1]}\bra{S_\kappa\bra{\WW^h},S_\kappa\bra{\XX^h}} \bra{1+ \norm{S_\kappa\bra{\WW^h}}_{p\textup{-var;}[0,1]}^{\floor{p}+1} +  \norm{S_\kappa\bra{\XX^h}}_{p\textup{-var;}[0,1]}^{\floor{p}+1}}\notag\\
&\gap\gap\gap\gap\gap\gap \cdot\exp\bra{C_1\seq{1+N_{\alpha,p}(S_\kappa\bra{\WW^h},[0,1])+N_{\alpha,p}(S_\kappa\bra{\XX^h},[0,1])}}\label{e-big-i}
\end{align}
for some (deterministic) constant $C_1=C_1(p,\gamma,\norm{V}_{\textup{Lip}^\gamma},\alpha)$. 
By the Lipschitz-continuity of the rough path lift and the fact that $\WW^h$ is the piecewise abelian approximation of enhanced Brownian motion $\WW$,  
\begin{equation}\label{e-p-var}
\norm{S_\kappa(\WW^h)}_{p\textup{-var;}[0,1]} 
\leq \norm{S_\kappa(\WW^h)}_{(2+\delta)\textup{-var;}[0,1]} 
\leq C_2(\kappa,\delta)\norm{\WW^h}_{(2+\delta)\textup{-var;}[0,1]} \leq C_2\norm{\WW}_{(2+\delta)\textup{-var;}[0,1]},
\end{equation}
for arbitrary $\delta>0$. 
The last quantity is integrable in $L^q$ for all $q\geq 1$ (\cite[Corollary 13.14]{FV}) with an identical statement holding for $S_\kappa(\XX^h)$. 

Next we establish the integrability of the exponential term in (\ref{e-big-i}) by following the proofs of Theorem 10 and 13 in \cite{bayer2013rough} along with \cite[Corollary 5]{friz2013integrability}. First let $1/r+1/s > 1$, where $r\in (2,3)$. Using the results of \cite[Remark 6.4]{cass2013integrability} and \cite[Corollary 2 and Remark 1]{friz2013integrability}, there exist constants $\alpha=\alpha(r,s)>0$ and $c_1=c_1(r,s)>0$ such that the tail estimate:
\[
\PP\bra{N_{\alpha,r}\bra{\WW,[0,1]}>u}\leq \exp\bra{-c_1\alpha^{2/r}u^{2/s}}
\] 
holds for all $u>0$ and step-sizes $h>0$. Since $\norm{\WW^h}_{r\textup{-var;}[s,t]} \leq \norm{\WW}_{r\textup{-var;}[s,t]}$, Lemma 2 of \cite{friz2013integrability} gives $N_{\alpha,\tilde{p}}(\WW^h,[0,1]) \leq N_{\alpha,r}(\WW,[0,1])$. Hence, 
\[
\PP\bra{N_{\alpha,r}\bra{\WW^h,[0,1]}>u}\leq \PP\bra{N_{\alpha,r}\bra{\WW,[0,1]}>u}.
\]
Then the Lipschitz-continuity of $S_\kappa(\cdot)$ and \cite[Lemma 2]{friz2013integrability} imply
\[
\PP\bra{N_{\alpha,r}\bra{S_\kappa(\WW^h),[0,1]} > u} 
\leq  
\PP\bra{N_{\alpha,r}\bra{\WW^h,[0,1]}>u}
\leq
\exp\bra{-c_1\alpha^{2/r}u^{2/s}} 
\]
also holds for all $u>0$, $h>0$, and for a possibly smaller $\alpha > 0$ (as in the proof of Theorem 13 of \cite{bayer2013rough}). 
The interpolation result of Lemma 8.16 of \cite{FV} tells us that 
\[
\omega_1(s,t):=\norm{S_\kappa(\WW^h)}^p_{p\textup{-var;}[s,t]} \leq \omega_2(s,t)\cdot\sup_{s\leq u < v \leq t} \norm{S_\kappa(\WW^h)_{u,v}}_{C}^{p-r},
\]
where $\omega_2(s,t):=\norm{S_\kappa(\WW^h)}^r_{r\textup{-var;}[s,t]}$. Thus taking $u<v\in [s,t]$ such that $\omega_2(u,v)\leq \alpha$, we then have $\omega_1(u,v) \leq \alpha^{p-r}\omega_2(u,v)$. Again appealing to \cite[Lemma 2]{friz2013integrability} yields
\[
N_{\alpha^{p-r}\alpha,p}\bra{S_\kappa(\WW^h), [s,t]} \leq N_{\alpha, r}\bra{S_\kappa(\WW^h), [s,t]}.
\]
If $\alpha\leq 1$ then $N_{\alpha,p}\bra{S_\kappa(\WW^h),[0,1]} \leq N_{\alpha,r}\bra{S_\kappa(\WW^h), [0,1]}$, and so
\begin{equation}\label{e-tail-estimate-1}
\PP\bra{N_{\alpha,p}\bra{S_\kappa(\WW^h),[0,1]} > u} 
\leq \exp\bra{-c_1\alpha^{2/r}u^{2/s}}.
\end{equation}
On the other hand, if $\alpha>1$ then Lemma 3 of \cite{friz2013integrability} gives
\[
N_{\alpha,p}\bra{S_\kappa(\WW^h),[0,1]} 
\leq 
\alpha^{p-r}\bra{1+2N_{\alpha^{p-r}\alpha,p}\bra{S_\kappa(\WW^h),[0,1]}}
\leq \alpha^{p-r}\bra{1+2N_{\alpha,r}\bra{S_\kappa(\WW^h),[0,1]}},
\]
which in turn gives the tail estimate:
\begin{equation}\label{e-tail-estimate-2}
\PP\bra{N_{\alpha,p}\bra{S_\kappa(\WW^h),[0,1]} > u} \leq c_2\exp\bra{-c_3\alpha^{2/r+2(r-p)/s}u^{2/s}}
\end{equation}
for some constants $c_i=c_i(r,s,\alpha)>0$, $i=1,2$. 
In either case, the tail estimates (\ref{e-tail-estimate-1}) and (\ref{e-tail-estimate-2}), together with the fact that $s\in (1,2)$, guarantee that for all $q\geq 1$, 
\[
\sup_{h>0} \abs{\exp\bra{C_1N_{\alpha,p}\bra{S_\kappa(\WW^h),[0,1]}}}_{L^q} \leq C_3=C_3(\kappa,q,\alpha) <\infty.
\]
By an identical argument, corresponding uniform estimates with $S_\kappa(\XX^h)$ instead of $S_\kappa(\WW^h)$ can be established. Consequently for all $q\geq 1$,
\begin{equation}\label{e-estimate-3}
\sup_{h > 0} \abs{\exp\bra{C_1\seq{1+N_{\alpha,{p}}\bra{S_\kappa(\WW^h),[0,1]}+N_{\alpha,{p}}\bra{S_\kappa(\XX^h),[0,1]}}}}_{L^q} \leq C_3 = C_3(\kappa,q,\alpha) < \infty.
\end{equation}
The proof is concluded by applying the Cauchy-Schwarz inequality to (\ref{e-big-i}) and then using the estimates provided by (\ref{e-p-var}), (\ref{e-estimate-3}) and Theorem \ref{t-main-bound}. 
\end{proof}


\section{Concluding remarks}

We discuss the possibility of two extensions of Theorem \ref{t-big}. 

\subsection{Fractional Brownian motion}

A natural question to ask is whether Theorem \ref{t-big} can be generalised to the case where the driving Brownian motion of the SDE (\ref{e-stratonovich-sde}) is replaced with a fractional Brownian motion with Hurst index $H \in (\frac{1}{3},\frac{1}{2})$, (turning (\ref{e-stratonovich-sde}) into a RDE). Recall that Brownian motion corresponds to $H=\frac{1}{2}$. The log-ODE method remains perfectly valid for this RDE, as does our piecewise abelian interpretation since the L\'{e}vy area of fractional Brownian motion is well-defined for $H>\frac{1}{4}$ (This is not the case for $H\leq \frac{1}{4}$ \cite[\S1]{neuenkirch2010discretizing}; even the  lift of the standard piecewise linear interpolation does not converge in $p$-variation under $L^1$ for $p>H^{-1}\geq 4$ \cite[\S4.5]{lyons2002system}). 
Instead, the main obstacle of this extension is the need to reproduce the main coupling result of Proposition \ref{p-kmt-coupling}. This is problematic because the proof of Davie in \cite{davie2014kmt} inherently relies upon the independence of increments in the Brownian motion case in order to perform the inductive coupling over finer dyadic intervals of $[0,1]$. Similarly, the original Koml\'{o}s-Major-Tusn\'{a}dy Theorem and the modern extensions of Zaitsev also critically rely upon the independence of increments of the random walk to be coupled with a Gaussian approximation. Since this is no longer the case when $H\neq\frac{1}{2}$, the authors see no way around this at present. In the even more extreme case of $H\in (\frac{1}{4}, \frac{1}{3}]$, for the log-ODE method to converge we require the first $\floor{H^{-1}}=3$ levels of the log-signature. The task of coupling these higher order terms is difficult, even in the Brownian case (as we now discuss). 

\subsection{Higher order log-ODE approximations}

This paper has dealt exclusively with the log-ODE method at level $m=2$; that is, the log-signature of the Brownian motion has been truncated to its first two levels. Thus we have only needed to couple the L\'{e}vy area increments with a Gaussian approximation, (conditional on the underlying Brownian increments). A natural extension would be to couple the higher order terms of the log-signature with Gaussian approximations, thus enabling us to use a higher order version of the log-ODE method. If this coupling were successful, then we could expect a better convergence rate in the Wasserstein metric for our resultant approximation scheme. Indeed, a truncation of the log-signature to level $m$ produces a log-ODE scheme with strong convergence in $L^2$ of order $O(h^{m/2})$ \cite[Theorem 4.1]{gyurko2008rough}. 

One difficulty is that one would need to extend the proof of Proposition \ref{p-kmt-coupling} to couple not only the L\'{e}vy area increments, but also the third iterated integrals, conditional on the underlying Brownian increments. This may not be possible without violating the matrix non-degeneracy conditions needed for Davie's coupling proof \cite[Theorem 1]{davie2014kmt}. A more pronounced obstacle is the task of establishing the necessary lifted rough path estimates of Section \ref{s-lifted-estimates}. The difference of the iterated Baker-Campbell-Hausdorff expansions of each piecewise abelian rough path would be significantly more complex because levels 2 and 3 of the group increments would not be equal. The authors see no solution at present.


\stepcounter{section}
\section*{Appendix: Iterated Baker-Campbell-Hausdorff formula}

The Baker-Campbell-Hausdorff formula links the structure of a Lie group with the corresponding structure on its associated Lie algebra. It does this by expressing the logarithm of the product of two Lie group elements as a Lie algebra element using only Lie algebraic operations.

\begin{Theorem}[Baker-Campbell-Hausdorff formula]\label{t-BCHf}
Let $\mathcal{G}$ be a Lie group with group product $\diamond$ and the corresponding Lie algebra $\mathfrak{g}$ defined over any field of characteristic $0$. Let $\exp: \mathfrak{g}\to\mathcal{G}$ be the exponential map. Then for every pair $x,y\in\mathfrak{g}$, 
\[
H(x,y):=\log\bra{\exp(x)\diamond\exp(y)},
\]
can be written as a formal infinite sum of elements of $\mathfrak{g}$. The first terms of order less than or equal to 5 are given by
\begin{align*}
H(x,y) &= x+y+\frac{1}{2}[x,y]+\frac{1}{12}\bra{\l[x,\l[x,y\r]\r]+[y,[y,x]]}
 -\frac{1}{24}[y,[x,[x,y]]]\\
&\gap\gap\gap\gap\gap\gap -\frac{1}{720}\bra{[[[[x,y],y],y],y]+[[[[y,x],x],x],x]}\\
&\gap\gap\gap\gap\gap\gap + \frac{1}{360}\bra{[[[[x,y],y],y],x]+[[[[y,x],x],x],y]}\\
&\gap\gap\gap\gap\gap\gap + \frac{1}{120}\bra{[[[[y,x],y],x],y] + [[[[x,y],x],y],x]} + \ldots
\end{align*}
\end{Theorem}

We can also consider the iterated product:
\[
H(x_1,\ldots,x_n) := \log\bra{e^{x_1} \diamond \ldots \diamond e^{x_n}}.
\]
This is the iterated version of the Baker-Campbell-Hausdorff formula which is used in the proof of the Chen-Strichartz development formula, where the latter gives an explicit expression for the logarithm  of the Brownian signature lift $\log S_N(\WW)_{s,t}$ (\cite{chen1957integration,strichartz1987campbell}). 

\begin{Theorem}[Iterated Baker-Campbell-Hausdorff-Formula]
The iterated Hausdorff coefficient has the form: 
\begin{align*}
H(x_1,\ldots,x_n) 
&= \sum_{k=1}^\infty \frac{(-1)^{k-1}}{k} \sum_{P\in B_k} \frac{1}{\bar{P}P!}x^P,
\end{align*}
where $B_k$, $\bar{P}$, $P!$ and $X^P$ are given by the following expressions:
\[
B_k = \seq{(p_i^j)_{\substack{i\in\seq{1,\ldots,n}\\{j\in\seq{1,\ldots,k}}}},\, p^j_i \in\NN : \forall j\in\seq{1,\ldots,k}, \sum^n_{i=1} p_i^j > 0 },\\
\]
\[
\bar{P} = \sum_{j=1}^k \sum_{i=1}^n p^j_i,\gap\gap P!=\prod_{j=1}^k \prod_{i=1}^n p^j_i!,
\]
\[
x^P = \underbrace{[x_1\ldots [x_1}_{p^1_1 \textup{ times}} \ldots \underbrace{[x_n \ldots [x_n}_{p^1_n \textup{ times}} \ldots \underbrace{[x_1 \ldots [x_1}_{p^k_1 \textup{ times}} \ldots \underbrace{[x_n \ldots [x_n, x_n]}_{p^k_n \textup{ times}} \ldots].
\]
\end{Theorem}

We refer to \cite[Appendix B]{baudoin2004introduction}, \cite[\S3.2]{castell1993asymptotic} and \cite[Theorem 3.11]{reutenauer1993free} for further details.
By induction it can be shown (cf. Example 3.2 of \cite{balogh2010exceptional}) that the first terms of the expansion up to nested Lie brackets of length 4 are given by
\begin{align}
&H(x_1,\ldots,x_n)\label{e-BCHf-d}\\
&= \sum^n_{i=1}  x_i + \frac{1}{2}\sum_{1\leq i < j \leq n} [x_i,x_j]\notag\\
&+ \frac{1}{4}\sum_{1\leq i < j < k \leq n} [[x_i,x_j],x_k]
+ \frac{1}{12} \sum_{i,j} \sum_{k>i\vee j} [x_i,[x_j,x_k]]
+ \frac{1}{12}\sum_{1\leq i<j \leq n} [x_j,[x_j,x_i]]\notag\\
&+\frac{1}{8}\sum_{1\leq i<j<k<l\leq n} [[[x_i,x_j],x_k],x_l] 
+\frac{1}{24}\sum_{i,j} \sum_{l>k>i\vee j} [[x_i,[x_j,x_k]],x_l]
+ \frac{1}{24}\sum_{1\leq i<j<k \leq n} [[x_j,[x_j,x_i]],x_k]\notag\\
&
+ \frac{1}{24}\sum_{1\leq i < j < k \leq n} [x_k,[x_k,[x_i,x_j]]]
- \frac{1}{24}\sum_{i,j} \sum_{k>i\vee j} [x_k,[x_i,[x_j,x_k]]]
+\ldots\notag
\end{align}
We now specialise to the case of $\bra{\Gg,\diamond,\mathfrak{g}} = \bra{G^{(\kappa)}(\RR^d),\otimes,\mathfrak{g}^{(\kappa)}(\RR^d)}$ and present a useful technical lemma for expressing the difference of two iterated BCHF expansions as the global difference at the Lie algebra, rather than the local difference at the Lie group level. 

\begin{Lemma}\label{l-useful}
Fix sequences $\{x_j\}_{j=1}^n, \{y_j\}_{j=1}^n \in G^{(\kappa)}(\RR^d)$ and set
\[
g_i = \pi_i\bra{H(x_1,\ldots,x_n)}, \gap h_i =\pi_i\bra{H(y_1,\ldots,y_n)} \in \mathfrak{g}^{(\kappa)}(\RR^d).
\]
Then for every integer $m\leq \kappa$, 
\begin{align*}
&\norm{\pi_m\bra{e^{x_1}\otimes \ldots\otimes e^{x_n} - e^{y_1}\otimes \ldots\otimes e^{y_n}}}_{(\RR^d)^{\otimes m}}\\
&\leq \sum^m_{k=1} \frac{1}{k!}\sum_{\substack{i_1,\ldots,i_k > 0\\i_1+\ldots+i_k=m}} \norm{ \sum^k_{j=1} g_{i_1} \otimes \ldots \otimes  g_{i_{j-1}} \otimes (g_{i_{j}} - h_{i_j}) \otimes h_{i_{j+1}} \otimes \ldots \otimes h_{i_k} }_{(\RR^d)^{\otimes m}}.
\end{align*}
\end{Lemma}
\begin{proof}
Recall the well-known non-commutative identity (cf. \cite[\S4]{bass2002extending}),
\begin{equation}\label{e-well-known-nci}
\bigotimes^n_{j=1} a_j - \bigotimes^n_{j=1} b_j = \sum^n_{j=1} \bigotimes^{j-1}_{i=1} a_i(a_j-b_j)\bigotimes_{i=j+1}^n b_i,
\end{equation}
for any sequences $\{a_j\}, \{b_j\}$, with the convention that $\otimes^0_{i=1} a_i = 1$. It follows that for every positive integer $m$, 
\begin{align*}
&\norm{\pi_m\bra{e^{x_1}\otimes \ldots\otimes e^{x_n} - e^{y_1}\otimes \ldots\otimes e^{y_n}}}_{(\RR^d)^{\otimes m}}\\
&= \norm{\pi_m\bra{e^{H(x_1,\ldots,x_n)}-e^{H(y_1,\ldots,y_n)}}}_{(\RR^d)^{\otimes m}}\\
&= \norm{\pi_m\bra{ \sum^m_{k=0} \frac{1}{k!} \bra{H(x_1,\ldots,x_n)^{\otimes k} - H(y_1,\ldots,y_n)^{\otimes k}}}}_{(\RR^d)^{\otimes m}}\\ 
&\leq \sum^m_{k=1} \frac{1}{k!} \norm{\pi_m\bra{H(x_1,\ldots,x_n)^{\otimes k} - H(y_1,\ldots,y_n)^{\otimes k}}}_{(\RR^d)^{\otimes m}}\\
&=\sum^m_{k=1} \frac{1}{k!}\norm{\sum_{\substack{i_1,\ldots,i_k > 0\\i_1+\ldots+i_k=m}}\bra{g_{i_1}\otimes \ldots \otimes g_{i_k} - h_{i_1}\otimes \ldots \otimes h_{i_k}}}_{(\RR^d)^{\otimes m}}\\
&\leq \sum^m_{k=1} \frac{1}{k!}\sum_{\substack{i_1,\ldots,i_k > 0\\i_1+\ldots+i_k=m}} \norm{ \sum^k_{j=1} g_{i_1} \otimes \ldots \otimes  g_{i_{j-1}} \otimes (g_{i_{j}} - h_{i_j}) \otimes h_{i_{j+1}} \otimes \ldots \otimes h_{i_k} }_{(\RR^d)^{\otimes m}}.
\end{align*}
The proof is complete. 
\end{proof}

We conclude the appendix with a useful technical lemma (see \cite[Lemma 3.1]{balogh2010exceptional}).

\begin{Lemma}\label{l-lie-tech}
Any Lie bracket of elements $x_1,\ldots,x_n\in\mathfrak{g}$ is a linear combination with coefficients of $\pm 1$ of nested commutators of the form 
\[
x_{i_1} * x_{i_2} * \ldots * x_{i_{n-1}} * x_{i_n} := [x_{i_1},[x_{i_2},[\ldots[x_{i_{n-1}},x_{i_n}]]]].
\]
\end{Lemma}
\begin{proof}
The proof relies upon the Jacobi identity and induction on the length of the commutator. The induction basis is the identity $[[x_1,x_2],[x_3,x_4]]=[x_4,[x_3,[x_1,x_2]]]-[x_3,[x_2,[x_1,x_2]]]$. 
\end{proof}


\section*{Acknowledgements}

The authors would like to thank Horatio Boedihardjo, Philippe Charmoy, Lajos Gyurko, Ben Hambly, Sean Ledger, Harald Oberhauser and Danyu Yang for many useful discussions. Special thanks go to Prof.~Davie of Edinburgh for answering many questions about his original proof in \cite{davie2014kmt}, as well as Dr.~Weijun Xu of Warwick for his help in Berlin. 
The research is supported by the European Research Council under the European Union's Seventh Framework Programme (FP7-IDEAS-ERC, ERC grant agreement nr. 291244). The authors are grateful for the support of the Oxford-Man Institute. 


\bibliographystyle{plain}
\bibliography{coupling_library}

\begin{thebibliography}{10}

\bibitem{alfonsi2014optimal}
A.~Alfonsi, B.~Jourdain, and A.~Kohatsu-Higa.
\newblock {Optimal transport bounds between the time-marginals of a
  multidimensional diffusion and its Euler scheme}.
\newblock {\em arXiv preprint arXiv:1405.7007}, 2014.

\bibitem{alfonsi2014pathwise}
A.~Alfonsi, B.~Jourdain, and A.~Kohatsu-Higa.
\newblock {Pathwise optimal transport bounds between a one-dimensional
  diffusion and its Euler scheme}.
\newblock {\em The Annals of Applied Probability}, 24(3):1049--1080, 2014.

\bibitem{alves2008monte}
C.J.S. Alves and A.B. Cruzeiro.
\newblock {Monte Carlo simulation of stochastic differential systems - a
  geometrical approach}.
\newblock {\em {Stochastic Processes and their Applications}}, 118(3):346--367,
  2008.

\bibitem{arous1989flots}
G.~Arous.
\newblock {Flots et s{\'e}ries de Taylor stochastiques}.
\newblock {\em Probability Theory and Related Fields}, 81(1):29--77, 1989.

\bibitem{bailleul2014flows}
I.~Bailleul.
\newblock {Flows driven by Banach space-valued rough paths}.
\newblock In {\em {S{\'e}minaire de Probabilit{\'e}s XLVI}}, pages 195--205.
  Springer, 2014.

\bibitem{balogh2010exceptional}
Z.~Balogh, R.~Berger, R.~Monti, and J.~Tyson.
\newblock {Exceptional sets for self-similar fractals in Carnot groups}.
\newblock In {\em Mathematical Proceedings of the Cambridge Philosophical
  Society}, volume 149, pages 147--172. Cambridge Univ Press, 2010.

\bibitem{bass2002extending}
R.F. Bass, B.M. Hambly, and T.J. Lyons.
\newblock {Extending the Wong-Zakai theorem to reversible Markov processes}.
\newblock {\em Journal of the European Mathematical Society}, 4(3):237--269,
  2002.

\bibitem{baudoin2004introduction}
F.~Baudoin.
\newblock {\em An introduction to the geometry of stochastic flows}.
\newblock {Imperial College Press}, 2004.

\bibitem{bayer2013rough}
C.~Bayer, P.~Friz, S.~Riedel, and J.~Schoenmakers.
\newblock {From rough path estimates to multilevel Monte Carlo}.
\newblock {\em arXiv preprint arXiv:1305.5779}, 2013.

\bibitem{blanchet2014epsilon}
J.~Blanchet, X.~Chen, and J.~Dong.
\newblock {Epsilon-Strong Simulation for Multidimensional Stochastic
  Differential Equations via Rough Path Analysis}.
\newblock {\em arXiv preprint arXiv:1403.5722}, 2014.

\bibitem{boutaib2013dimension}
Y.~Boutaib, L.G. Gyurk{\'o}, T.J. Lyons, and D.~Yang.
\newblock Dimension-free {E}uler estimates of rough differential equations.
\newblock {\em Rev. Roumanine Math. Pures Appl.}, 2013.

\bibitem{breuillard2005local}
E.~Breuillard.
\newblock {Local limit theorems and equidistribution of random walks on the
  Heisenberg group}.
\newblock {\em {Geometric and Functional Analysis}}, 15(1):35--82, 2005.

\bibitem{breuillard2009random}
E.~Breuillard, P.~Friz, and M.~Huesmann.
\newblock From random walks to rough paths.
\newblock {\em Proceedings of the American Mathematical Society},
  137(10):3487--3496, 2009.

\bibitem{clark1980maximum}
R.J. Cameron and J.M.C. Clark.
\newblock {\em The maximum rate of convergence of discrete approximations for
  stochastic differential equations}.
\newblock Springer, 1980.

\bibitem{cass2013integrability}
T.~Cass, C.~Litterer, and T.J. Lyons.
\newblock {Integrability and tail estimates for Gaussian rough differential
  equations}.
\newblock {\em The Annals of Probability}, 41(4):3026--3050, 2013.

\bibitem{cass2013evolving}
T.~Cass and T.J. Lyons.
\newblock Evolving communities with individual preferences.
\newblock {\em Proceedings of the London Mathematical Society}, pages 83--107,
  2014.

\bibitem{castell1993asymptotic}
F.~Castell.
\newblock {Asymptotic expansion of stochastic flows}.
\newblock {\em Probability theory and related fields}, 96(2):225--239, 1993.

\bibitem{castell1995efficient}
F.~Castell and J.~Gaines.
\newblock {An efficient approximation method for stochastic differential
  equations by means of the exponential Lie series}.
\newblock {\em Mathematics and computers in simulation}, 38(1):13--19, 1995.

\bibitem{castell1996ordinary}
F.~Castell and J.~Gaines.
\newblock {The ordinary differential equation approach to asymptotically
  efficient schemes for solution of stochastic differential equations}.
\newblock In {\em {Annales de l'Institut Henri Poincar{\'e}, Probabilit{\'e}s
  et Statistiques}}, volume~32, pages 231--250. Elsevier, 1996.

\bibitem{chen1957integration}
K.T. Chen.
\newblock {Integration of paths, geometric invariants and a generalized
  Baker-Hausdorff formula}.
\newblock {\em Annals of Mathematics}, pages 163--178, 1957.

\bibitem{cruzeiro2006numerical}
A.B. Cruzeiro and P.~Malliavin.
\newblock {Numerical approximation of diffusions in Rd using normal charts of a
  Riemannian manifold}.
\newblock {\em {Stochastic processes and their applications}},
  116(7):1088--1095, 2006.

\bibitem{cruzeiro2004geometrization}
A.B. Cruzeiro, P.~Malliavin, and A.~Thalmaier.
\newblock {Geometrization of Monte-Carlo numerical analysis of an elliptic
  operator: strong approximation}.
\newblock {\em Comptes Rendus Mathematique}, 338(6):481--486, 2004.

\bibitem{davie2014kmt}
A.~Davie.
\newblock {K}{M}{T} theory applied to approximations of {S}{D}{E}.
\newblock In {\em Stochastic Analysis and Applications 2014}, pages 185--201.
  Springer, 2014.

\bibitem{davie2014pathwise}
A.~Davie.
\newblock Pathwise approximation of stochastic differential equations using
  coupling.
\newblock {\em preprint}, 2014.

\bibitem{dereich2011multilevel}
S.~Dereich.
\newblock {Multilevel Monte Carlo algorithms for L{\'e}vy-driven SDEs with
  Gaussian correction}.
\newblock {\em {The Annals of Applied Probability}}, 21(1):283--311, 2011.

\bibitem{deya2012milstein}
A.~Deya, A.~Neuenkirch, and S.~Tindel.
\newblock A {M}ilstein-type scheme without {L}{\'e}vy area terms for {S}{D}{E}s
  driven by fractional {B}rownian motion.
\newblock In {\em Annales de l'Institut Henri Poincar{\'e}, Probabilit{\'e}s et
  Statistiques}, volume~48, pages 518--550. Institut Henri Poincar{\'e}, 2012.

\bibitem{dickinson2007optimal}
A.S. Dickinson.
\newblock {Optimal Approximation of the Second Iterated Integral of Brownian
  Motion}.
\newblock {\em Stochastic Analysis and Applications}, 25(5):1109--1128, 2007.

\bibitem{einmahl1989extensions}
U.~Einmahl.
\newblock {Extensions of results of Koml{\'o}s, Major, and Tusn{\'a}dy to the
  multivariate case}.
\newblock {\em Journal of multivariate analysis}, 28(1):20--68, 1989.

\bibitem{fliess}
D.~Fliees~M., Normand-Cyrot.
\newblock {Alg\`{e}bres de Lie nilpotentes, formule de Baker-Campbell-Hausdorff
  et int\'{e}grals it\'{e}r\'{e}es de K.T. Chen}.
\newblock {\em {S\'{e}minaire de Probabilit\'{e}s} LNM}, 920, 1982.

\bibitem{friz2014course}
P.~Friz and M.~Hairer.
\newblock A course on rough paths.
\newblock {\em Preprint}, 2014.

\bibitem{friz2009rough}
P.~Friz and H.~Oberhauser.
\newblock {Rough path limits of the Wong--Zakai type with a modified drift
  term}.
\newblock {\em Journal of Functional Analysis}, 256(10):3236--3256, 2009.

\bibitem{friz2013integrability}
P.~Friz and S.~Riedel.
\newblock Integrability of (non-) linear rough differential equations and
  integrals.
\newblock {\em Stochastic Analysis and Applications}, 31(2):336--358, 2013.

\bibitem{friz2014convergence}
P.~Friz and S.~Riedel.
\newblock {Convergence rates for the full Gaussian rough paths}.
\newblock In {\em Annales de l'Institut Henri Poincar{\'e}, Probabilit{\'e}s et
  Statistiques}, volume~50, pages 154--194. Institut Henri Poincar{\'e}, 2014.

\bibitem{friz2005approximations}
P.~Friz and N.~Victoir.
\newblock {Approximations of the Brownian rough path with applications to
  stochastic analysis}.
\newblock In {\em {Annales de l'Institut Henri Poincar{\'e}, Probabilit{\'e}s
  et Statistiques}}, volume~41, pages 703--724, 2005.

\bibitem{friz2008euler}
P.~Friz and N.~Victoir.
\newblock Euler estimates for rough differential equations.
\newblock {\em Journal of Differential Equations}, 244(2):388--412, 2008.

\bibitem{FV}
P.~Friz and N.~Victoir.
\newblock {\em Multidimensional stochastic processes as rough paths: theory and
  applications}.
\newblock {Cambridge University Press}, 2010.

\bibitem{gaines1994random}
J.G. Gaines and T.J. Lyons.
\newblock Random generation of stochastic area integrals.
\newblock {\em SIAM Journal on Applied Mathematics}, 54(4):1132--1146, 1994.

\bibitem{gelbrich1995simultaneous}
M.~Gelbrich.
\newblock Simultaneous time and chance discretization for stochastic
  differential equations.
\newblock {\em Journal of computational and applied mathematics},
  58(3):255--289, 1995.

\bibitem{gelbrich1996discretization}
M.~Gelbrich and S.~Rachev.
\newblock {Discretization for stochastic differential equations, Lp Wasserstein
  metrics, and econometrical models}.
\newblock {\em {Lecture Notes-Monograph Series}}, pages 97--119, 1996.

\bibitem{gobet2000weak}
E.~Gobet.
\newblock {Weak approximation of killed diffusion using Euler schemes}.
\newblock {\em Stochastic processes and their applications}, 87(2):167--197,
  2000.

\bibitem{gyurko2008rough}
L.G. Gyurko and T.J. Lyons.
\newblock Rough paths based numerical algorithms in computational finance.
\newblock {\em {Mathematics in Finance: UIMP-RSME Lluis A. Santal{\'o} Summer
  School}}, 2008.

\bibitem{hutzenthaler2011strong}
M.~Hutzenthaler, A.~Jentzen, and P.E. Kloeden.
\newblock {Strong and weak divergence in finite time of Euler's method for
  stochastic differential equations with non-globally Lipschitz continuous
  coefficients}.
\newblock {\em {Proc. R. Soc. London Ser. A}}, 467(2130):1563--1576, 2011.

\bibitem{jentzen2011taylor}
A.~Jentzen and P.E. Kloeden.
\newblock {\em {Taylor approximations for stochastic partial differential
  equations}}, volume~83.
\newblock SIAM, 2011.

\bibitem{math1989rate}
S.~Kanagawa.
\newblock {The rate of convergence for approximate solutions of stochastic
  differential equations}.
\newblock {\em {Tokyo J. Math}}, 12(1), 1989.

\bibitem{kantorovich_russian}
L.V. Kantorovich.
\newblock {On a problem of Monge (Russian)}.
\newblock {\em {Uspekhi Mat. Nauk.}}, 3:225--226, 1948.

\bibitem{Kloeden_Neuenkirch_2013}
P.E. Kloeden and A.~Neuenkirch.
\newblock {Convergence of numerical methods for stochastic differential
  equations in mathematical finance}.
\newblock {\em {Recent Developments in Computational Finance: Foundations,
  Algorithms and Applications, Interdisciplinary Mathematical Sciences
  Series}}, 14:49--80, 2013.

\bibitem{kloeden1992approximation}
P.E. Kloeden, E.~Platen, and I.~Wright.
\newblock {The approximation of multiple stochastic integrals}.
\newblock {\em Stochastic analysis and applications}, 10(4):431--441, 1992.

\bibitem{komlos1975approximation}
J.~Koml{{\'o}}s, P.~Major, and G.~Tusn{{\'a}}dy.
\newblock An approximation of partial sums of independent random variables and
  the sample distribution function.
\newblock {\em Zeitschrift f{\"u}r Wahrscheinlichkeitstheorie und verwandte
  Gebiete}, 32(1-2):111--131, 1975.

\bibitem{ledoux2002large}
M.~Ledoux, Z.~Qian, and T.~Zhang.
\newblock {Large deviations and support theorem for diffusion processes via
  rough paths}.
\newblock {\em {Stochastic processes and their applications}}, 102(2):265--283,
  2002.

\bibitem{lejay2003importance}
A.~Lejay and T.J. Lyons.
\newblock {On the importance of the L{\'e}vy area for studying the limits of
  functions of converging stochastic processes. Application to homogenization}.
\newblock In {\em Current trends in potential theory}, volume~7. The Theta
  foundation/American Mathematical Society, 2003.

\bibitem{lyons1994differential}
T.J. Lyons.
\newblock {Differential equations driven by rough signals. I. An extension of
  an inequality of L.C. Young}.
\newblock {\em Math. Res. Lett}, 1(4):451--464, 1994.

\bibitem{lyons1998differential}
T.J. Lyons.
\newblock Differential equations driven by rough signals.
\newblock {\em Revista Matem{\'a}tica Iberoamericana}, 14(2):215--310, 1998.

\bibitem{lyons2014rough}
T.J. Lyons.
\newblock {Rough paths, Signatures and the modelling of functions on streams}.
\newblock {\em arXiv preprint arXiv:1405.4537}, 2014.

\bibitem{lyons2004differential}
T.J. Lyons, M.~Caruana, and T.~L{\'e}vy.
\newblock Differential equations driven by rough paths.
\newblock {\em {\'E}cole d'{\'e}t{\'e} des probabilit{\'e}s de saint-flour},
  34:2007, 2004.

\bibitem{lyons2002system}
T.J. Lyons and Z.~Qian.
\newblock {\em System control and rough paths}.
\newblock {Oxford University Press}, 2002.

\bibitem{lyons2005sound}
T.J. Lyons and N.~Sidorova.
\newblock Sound compression: a rough path approach.
\newblock In {\em {Proceedings of the 4th international symposium on
  Information and communication technologies}}, pages 223--228. Trinity College
  Dublin, 2005.

\bibitem{lyons2004cubature}
T.J. Lyons and N.~Victoir.
\newblock {Cubature on Wiener space}.
\newblock {\em Proceedings of the Royal Society of London. Series A:
  Mathematical, Physical and Engineering Sciences}, 460(2041):169--198, 2004.

\bibitem{lyons2007extension}
T.J. Lyons and N.~Victoir.
\newblock An extension theorem to rough paths.
\newblock In {\em Annales de l'Institut Henri Poincare (C) Non Linear
  Analysis}, volume~24, pages 835--847. Elsevier, 2007.

\bibitem{danyu2014}
T.J. Lyons and D.~Yang.
\newblock Rough differential equations in {B}anach space driven by weak
  geometric p-rough paths.
\newblock {\em arXiv preprint arXiv:1402.2900}, 2014.

\bibitem{lyons2015theory}
T.J. Lyons and D.~Yang.
\newblock {The theory of rough paths via one-forms and the extension of an
  argument of Schwartz to RDEs}.
\newblock {\em arXiv preprint arXiv:1503.06175}, 2015.

\bibitem{magnus1954exponential}
W.~Magnus.
\newblock {On the exponential solution of differential equations for a linear
  operator}.
\newblock {\em Communications on pure and applied mathematics}, 7(4):649--673,
  1954.

\bibitem{malrieu2003convergence}
F.~Malrieu.
\newblock {Convergence to equilibrium for granular media equations and their
  Euler schemes}.
\newblock {\em The Annals of Applied Probability}, 13(2):540--560, 2003.

\bibitem{maruyama1955continuous}
G.~Maruyama.
\newblock {Continuous Markov processes and stochastic equations}.
\newblock {\em {Rendiconti del Circolo Matematico di Palermo}}, 4(1):48--90,
  1955.

\bibitem{mil1974approximate}
G.N. Milshtein.
\newblock {Approximate integration of stochastic differential equations
  (Russian)}.
\newblock {\em {Teor. Veroyatnost. i Primenen}}, 19(3):583--588, 1974.

\bibitem{monge1781memoire}
G.~Monge.
\newblock {M{\'e}moire sur la th{\'e}orie des d{\'e}blais et des remblais}.
\newblock {\em {M{\'e}moires de l'Acad{\'e}mie Royale des Sciences}},
  XVIII-XIX:666--704, 1781.

\bibitem{neuenkirch2010discretizing}
A.~Neuenkirch, S.~Tindel, and J.~Unterberger.
\newblock {Discretizing the fractional L{\'e}vy area}.
\newblock {\em {Stochastic Processes and Their Applications}}, 120(2):223--254,
  2010.

\bibitem{ninomiya2008weak}
S.~Ninomiya and N.~Victoir.
\newblock {Weak approximation of stochastic differential equations and
  application to derivative pricing}.
\newblock {\em Applied Mathematical Finance}, 15(2):107--121, 2008.

\bibitem{pages2011convergence}
G.~Pag{{\`e}}s and A.~Sellami.
\newblock {Convergence of multi-dimensional quantized SDEs}.
\newblock In {\em {S{\'e}minaire de probabilit{\'e}s XLIII}}, pages 269--307.
  Springer, 2011.

\bibitem{pap1993central}
G.~Pap.
\newblock {Central limit theorems on nilpotent Lie groups}.
\newblock {\em Probab. Math. Stat}, 14(2):287--312, 1993.

\bibitem{press2007numerical}
W.H. Press, S.A. Teukolsky, W.T. Vetterling, and B.P. Flannery.
\newblock {\em {Numerical recipes 3rd edition: The art of scientific
  computing}}.
\newblock {Cambridge University Press}, 2007.

\bibitem{rachev1998mass}
S.~Rachev and L.~Ruschendorf.
\newblock {Mass transportation problems, I and II: theory and applications},
  1998.

\bibitem{reutenauer1993free}
C.~Reutenauer.
\newblock {Free Lie algebras, volume 7 of London Mathematical Society
  Monographs. New Series}, 1993.

\bibitem{riedel2014talagrand}
S.~Riedel.
\newblock {Talagrand's transportation-cost inequality and applications to
  (rough) path spaces}.
\newblock {\em arXiv preprint arXiv:1403.2585}, 2014.

\bibitem{riedel2013simple}
S.~Riedel and W.~Xu.
\newblock {A simple proof of distance bounds for Gaussian rough paths}.
\newblock {\em Electron. J. Probab}, 18(108):1--22, 2013.

\bibitem{ryden2001simulation}
T.~Ryd{\'e}n and M.~Wiktorsson.
\newblock {On the simulation of iterated It{\^o} integrals}.
\newblock {\em Stochastic processes and their applications}, 91(1):151--168,
  2001.

\bibitem{sipilainen1993pathwise}
E.M. Sipil{\"a}inen.
\newblock {Pathwise view on solutions of stochastic differential equations}.
\newblock {\em {Ph.D Thesis, University of Edinburgh}}, 1993.

\bibitem{stein1970singular}
E.M. Stein.
\newblock {\em Singular integrals and differentiability properties of
  functions}, volume~2.
\newblock Princeton University Press, 1970.

\bibitem{strichartz1987campbell}
R.S. Strichartz.
\newblock {The Campbell-Baker-Hausdorff-Dynkin formula and solutions of
  differential equations}.
\newblock {\em Journal of Functional Analysis}, 72(2):320--345, 1987.

\bibitem{talay2002stochastic}
D.~Talay.
\newblock {Stochastic Hamiltonian systems: exponential convergence to the
  invariant measure, and discretization by the implicit Euler scheme}.
\newblock {\em Markov Process. Related Fields}, 8(2):163--198, 2002.

\bibitem{talay1990expansion}
D.~Talay and L.~Tubaro.
\newblock Expansion of the global error for numerical schemes solving
  stochastic differential equations.
\newblock {\em Stochastic analysis and applications}, 8(4):483--509, 1990.

\bibitem{vaserstein}
L.N. Vaserstein.
\newblock {Markov processes over denumerable products of spaces describing
  large system of automata (Russian)}.
\newblock {\em {Problemy Peredaci Informacii}}, 5:64--72, 1969.

\bibitem{cedric2003topics}
C.~Villani.
\newblock {\em {Topics in Optimal Transportation}}.
\newblock Number~58. American Mathematical Soc., 2003.

\bibitem{wiktorsson2001joint}
M.~Wiktorsson.
\newblock {Joint characteristic function and simultaneous simulation of
  iterated It{\^o} integrals for multiple independent Brownian motions}.
\newblock {\em Annals of Applied Probability}, pages 470--487, 2001.

\bibitem{young1936inequality}
L.C. Young.
\newblock {An inequality of the H{\"o}lder type, connected with Stieltjes
  integration}.
\newblock {\em Acta Mathematica}, 67(1):251--282, 1936.

\bibitem{zaitsevALL}
A.Y. Zaitsev.
\newblock {Multidimensional version of a result of {S}akhanenko in the
  invariance principle for vectors with finite exponential moments. I,II,III}.
\newblock {\em {Teor. Veroyatnost. i Primenen \textup{\textbf{45} (2000),
  718-738; \textbf{46} (2001), 535-561 and 744-769. Translations in} Theory
  Probab. Appl. \textup{\textbf{45} (2002), 626-641; \textbf{46} (2003),
  490-514 and 676-698.}}}

\bibitem{zaitsev1996estimates}
A.Y. Zaitsev.
\newblock Estimates for the quantiles of smooth conditional distributions and
  the multidimensional invariance principle.
\newblock {\em Siberian Mathematical Journal}, 37(4):706--729, 1996.

\bibitem{zaitsev1998multidimensional}
A.Y. Zaitsev.
\newblock Multidimensional version of the results of {K}oml{\'o}s, {M}ajor and
  {T}usn{\'a}dy for vectors with finite exponential moments.
\newblock {\em ESAIM: Probability and Statistics}, 2:41--108, 1998.

\end{thebibliography}

\end{document}